\documentclass{article}
\usepackage[utf8]{inputenc}
\usepackage{amsmath,graphicx,tikz,color,pgf,mathrsfs,verbatim}
\usepackage{amsfonts}
\usepackage{microtype}
\usepackage[margin=1.2in]{geometry}
\usepackage[english]{babel}
\usepackage{fancyhdr}
\usepackage{tikz-cd} 
\usepackage{enumitem}
\usepackage{xspace}
\usepackage{amssymb}
\usepackage{marginnote}
\usepackage{amsthm}
\usepackage{tikzsymbols}
\usepackage{mathtools}
\usetikzlibrary{arrows}
\usepackage{fontawesome,enumerate}
\usepackage[T1]{fontenc}
 \usepackage{float}
 \usepackage{hyperref}
 \usepackage[normalem]{ulem} 
\usepackage{pdflscape}

\newtheorem{thm}{Theorem}[subsection]
\newtheorem{lemma}[thm]{Lemma}
\newtheorem{prop}[thm]{Proposition}
\newtheorem{corollary}[thm]{Corollary}

\newtheorem{remark}[thm]{Remark}
\newtheorem{rmk}[thm]{Remark}

\newtheorem{definition}[thm]{Definition}

\newtheorem{example}[thm]{Example}
\theoremstyle{remark}
\numberwithin{equation}{section} 
\numberwithin{figure}{section}
\numberwithin{table}{section}

\newcommand{\oCP}{{\overline{\CP}}\!\,}

\newenvironment{itemlist}
   { \begin{list} {$\bullet$}
         { \setlength{\topsep}{.5ex}  \setlength{\itemsep}{.5ex} \setlength{\leftmargin}{2.5ex} } }
   { \end{list} }

\newcommand{\bw}{{\bf{w}}}

\newcommand{\bE}{{\bf{E}}}

\newcommand{\Cc}{{\mathcal C}}

\newcommand{\om}{{\omega}}
\newcommand{\de}{{\delta}}

\newcommand{\la}{{\lambda}}

\newcommand{\si}{{\sigma}}

\newcommand{\p}{{\partial}}

\newcommand{\er}{{\Diamond}}
\newcommand{\Z}{\mathbb{Z}}

\newcommand{\R}{\mathbb{R}}
\newcommand{\Q}{\mathbb{Q}}
\newcommand{\C}{\mathbb{C}}
\newcommand{\CP}{\mathbb{CP}}

\newcommand{\eps}{\varepsilon}

\newcommand{\m}{\mathfrak{m}}

\newcommand{\Tt}{\mathcal{T}}

\newcommand{\acc}{\mathrm{acc}}
\newcommand{\sembeds}{\stackrel{s}{\hookrightarrow}}

 \newcommand{\vol}{\operatorname{vol}}

\newcommand{\bm}{\mathbf{m}}
\usepackage[abs]{overpic}
\usepackage{tikz,pgfplots}
\pgfplotsset{compat=1.18}

\title{A classification of infinite staircases for Hirzebruch surfaces}

\author{Nicki Magill\footnote{Department of Mathematics, University of California, Berkeley} \and Ana Rita Pires\footnote{School of Mathematics and Maxwell Institute for Mathematical Sciences, University of Edinburgh} \and Morgan Weiler\footnote{Department of Mathematics, University of California, Riverside}}

\date{}

\begin{document}

\maketitle

\begin{abstract}
    The ellipsoid embedding function of a symplectic manifold gives the smallest amount by which the symplectic form must be scaled in order for a standard ellipsoid of the given eccentricity to embed symplectically into the manifold. It was first computed for the standard four-ball (or equivalently, the complex projective plane) by McDuff and Schlenk, and found to contain the unexpected structure of an ``infinite staircase'', that is, an infinite sequence of nonsmooth points arranged in a piecewise linear stair-step pattern. Later work of Usher and Cristofaro-Gardiner--Holm--Mandini--Pires suggested that while four-dimensional symplectic toric manifolds with infinite staircases are plentiful, they are highly non-generic. This paper concludes the systematic study of one-point blowups of the complex projective plane, building on previous work of Bertozzi-Holm-Maw-McDuff-Mwakyoma-Pires-Weiler, Magill-McDuff, Magill-McDuff-Weiler, and Magill on these Hirzebruch surfaces. 
    
    We prove a conjecture of Cristofaro-Gardiner--Holm--Mandini--Pires for this family: that if the blowup is of rational weight and the embedding function has an infinite staircase then that weight must be $1/3$. We show also that the function for this manifold does not have a descending staircase.
    Furthermore, we give a sufficient and necessary condition for the existence of an infinite staircase in this family which boils down to solving a quadratic equation and  computing the function at one specific value.
    Many of our intermediate results also apply to the case of the polydisk (or equivalently, the symplectic product of two spheres).

\vspace{5pt}

\end{abstract}

\tableofcontents

\vspace{1cm}

\section{Introduction}
\subsection{Main results}
The four-dimensional \textbf{ellipsoid} is the region
\[
E(c,d):=\left\{(z_1,z_2)\in\C^2\,\middle|\,\frac{\pi|z_1|^2}{c}+\frac{\pi|z_2|^2}{d}<1\right\}
\]
with the standard symplectic form $\omega_0$ on $\C^2=\R^4$. For $z \geq 1$, the \textbf{ellipsoid embedding function} of a symplectic four-manifold $(X,\omega)$ is
\[
c_X(z):=\inf\left\{\lambda\,\middle|\, (E(1,z),\omega_0)\sembeds(X,\la\omega)\right\}
\]
 where $\sembeds$ denotes a symplectic embedding. For this symplectic embedding to exist, the volume of $E(1,z)$ must be smaller than or equal to the volume of $\la X:=(X,\la \om)$, and when the volumes are equal we call it a {\bf full filling}. This inequality gives a lower bound 
\begin{equation}
\label{eqn:volb} c_X(z) \geq\sqrt{\frac{z}{\text{volume}(X)}}=:\vol_X(z).
\end{equation}
As proved in \cite[Proposition~2.1]{AADT}
, when $X$ is a finite toric blow up of $\CP^2$ or $\CP^1\times\CP^1$, its ellipsoid embedding function is continuous, nondecreasing, piecewise linear when not equal to $\vol_X(z),$ and is equal to $\vol_X(z)$ for sufficiently large values of $z.$ We say $c_X$ has an {\bf infinite staircase} or {\bf staircase} if it is nonsmooth at infinitely many points. In this case, by the properties noted above, these nonsmooth points must accumulate and by \cite[Theorem~1.13]{AADT}, they do so at a unique $z$-value called the {\bf accumulation point}, which is a solution to a quadratic equation whose coefficients are determined by $X$. Further, at this $z$-value, there is a full filling, i.e., $c_X(z)=\vol_X(z)$. 

The first ellipsoid embedding function was computed by McDuff and Schlenk in \cite{McDuffSchlenk12} with $X$ the standard 4-ball $B(1)=E(1,1)$, or equivalently $\CP^2$ (see \cite{McDuffPolt} or \cite[Theorem 1.4]{AADT}). Its piecewise linear structure is determined by the Fibonacci numbers and the accumulation point is $\tau^4$ where $\tau$ is the golden ratio. Further work by many authors identified and ruled out infinite staircases for several other targets $X$ (\cite{FrenkelMuller15}, \cite{CasalsVianna}, \cite{CGspecial}, \cite{ellipsoidpolydisc}, \cite{SPUR}). Of note, in \cite{AADT}  Cristofaro-Gardiner, Holm, Mandini, and Pires studied the ellipsoid embedding functions of a class of manifolds which generalizes finite toric blowups of $\CP^2$ and $\CP^1\times\CP^1$. They identified in \cite[Theorem 1.19]{AADT} a family of twelve Fano manifolds with infinite staircases which included all previously known infinite staircases. Further, they conjectured in \cite[Conjecture 1.23]{AADT} that if $X$ is a rational member of this class which has an infinite staircase, then (up to scaling) it must be one of those twelve. Here, rational means that its moment polygon can be scaled to be a lattice polygon. In contrast, Usher found in \cite{usher} an infinite family of infinite staircases when $X$ is a polydisk (equivalently, a $\CP^1\times\CP^1$, see \cite{McDuffPolt} or \cite[Theorem 1.4]{AADT}):
\[
P(c,d):=\left\{(z_1,z_2)\in\C^2\,\middle|\,\pi|z_1|^2<c,\;\pi|z_2|^2<d\right\},
\]
again with symplectic form $\omega_0$. For each of these polydisks with an infinite staircase, we have $d/c$ irrational, which means that the respective moment polygon cannot be scaled to a lattice polygon.

The present paper is the fifth in a sequence beginning with \cite{ICERM, MM, MMW, M} whose goal is to study the ellipsoid embedding functions of the one-parameter family of one-point blowups of $\CP^2$, which will therefore include rational and irrational targets.

These one-point blowups of $\CP^2$ are the Hirzebruch surfaces 
$$H_b:=\CP^2(1)\#\oCP^2(b),$$ 
with $b\in[0,1)$. One of the twelve Fano staircases of \cite{AADT} is the infinite staircase of $c_{H_{1/3}}$; this is the only one-point blowup of  $\CP^2$ on that list. Our first main result finishes the proof of Conjecture 1.23 in \cite{AADT} for this family, which had been started in \cite{MMW}:

\begin{thm}\label{thm:13proved}
For rational $b$, the ellipsoid embedding function of the Hirzebruch surface $H_b$ has an infinite staircase if and only if $b=1/3$.
\end{thm}

For irrational $b$-values, the papers \cite{ICERM, MM} identified families of $H_b$ with infinite staircases, including the first ``descending'' infinite staircases. The papers \cite{MMW, M} find further staircases by uncovering a fractal structure of $b$-values with infinite staircases near a $b$-value with an infinite staircase.

We will use the notation $c_b:=c_{H_b}$. By \cite[Theorem 1.13]{AADT}, if $c_b$ has an infinite staircase then its accumulation point is the unique real $z>1$ root of 
\begin{equation} \label{eq:accb}
z^2-\left(\frac{(3-b)^2}{1-b^2}-2\right)z+1=0.
\end{equation}
Extending this definition, for any $b \in [0,1)$, we define $\acc(b)$ to be the larger root of \eqref{eq:accb}, thus obtaining a function $\acc:[0,1)\to[3+2\sqrt{2},\infty).$

For the family $H_b$, we have
\[
\vol_b(z):=\vol_{H_b}(z)=\sqrt{\frac{z}{1-b^2}},
\]
and we define the \textbf{staircase obstruction} to be the quantity
\[
O(b):=c_b(\acc(b))-\vol_b(\acc(b)).
\]
As mentioned above, by \cite[Theorem~1.13]{AADT}, if $c_b$ has an infinite staircase, then $O(b)=0$. Figure \ref{fig:acc} illustrates the parameterized curve $b\mapsto(\acc(b),\vol_b(\acc(b))$.

Our second main result provides a converse to \cite[Theorem~1.13]{AADT} for the $H_b$ family:

\begin{figure}[ht]
\begin{center}
\begin{tikzpicture}
\begin{axis}[
	axis lines = middle,
	xtick = {5.5,6,6.5,7},
	ytick = {2.5,3,3.5},
	tick label style = {font=\small},
	xlabel = $z$,
	ylabel = $\lambda$,
	xlabel style = {below right},
	ylabel style = {above left},
	xmin=5.4,
	xmax=7.2,
	ymin=2.4,
	ymax=3.6,
	grid=major,
	width=3in,
	height=2.25in]
\addplot [red, thick,
	domain = 0:0.7,
	samples = 120
]({ (((3-x)*(3-x)/(2*(1-x^2))-1)+sqrt( ((3-x)*(3-x)/(2*(1-x^2))-1)* ((3-x)*(3-x)/(2*(1-x^2))-1) -1 ))},{sqrt(( (((3-x)*(3-x)/(2*(1-x^2))-1)+sqrt( ((3-x)*(3-x)/(2*(1-x^2))-1)* ((3-x)*(3-x)/(2*(1-x^2))-1) -1 )))/(1-x^2))});
\addplot [black, only marks, very thick, mark=*] coordinates{(6,2.5)};
\addplot [blue, only marks, very thick, mark=*] coordinates{(6.85,2.62)};
\addplot [green, only marks, very thick, mark=*] coordinates{(5.83,2.56)};
\end{axis}
\end{tikzpicture}
\end{center}
\caption{This is \cite[Figure~1.1.4]{ICERM}, which depicts the parameterized accumulation point curve $(z,\lambda)=(\acc(b), \vol_b(\acc(b)))$ for $0\le b< 1$.
The blue point indicates where $b=0$, at $(\tau^4, \tau^2)$. It is the accumulation point for the Fibonacci stairs of \cite{McDuffSchlenk12}. The green point, where $b=1/3$, is the accumulation point for $c_{1/3}$, with $z=3+2\sqrt{2}$. The plot makes it clear that this is the minimum of the function $b\mapsto\acc(b)$.  
The black point has $z=6, b=b_0=1/5$ (see \eqref{eqn:bidef}), and is the minimum of $\vol_b(\acc(b))$. 
} \label{fig:acc} 
\end{figure}
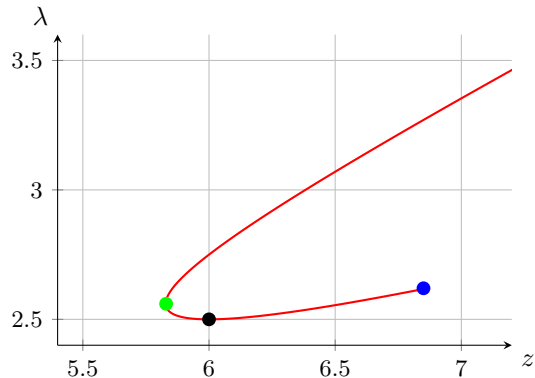

\begin{thm}\label{thm:main}
The ellipsoid embedding function $c_b$ of the Hirzebruch surface $H_b$ has an infinite staircase if and only if $\acc(b)$ is irrational and $O(b)=0$.
\end{thm}

Furthermore, many of our results also apply to polydisks, suggesting a method of proof for an analogous theorem in that case, see Remark~\ref{rmk:polydisk}.

Finally, our third main result computes $c_{1/3}$  from its accumulation point $\acc(1/3)=3+2\sqrt{2}$ up to 6; see also \cite[Lemma~2.2.12]{MM}.
\begin{thm}\label{thm:nodesc13} For $3+2\sqrt{2}\leq z\leq6$, 
    \[
    c_{1/3}(z)=\frac{3(z+1)}{8}.
    \]
    Consequently, there is no descending infinite staircase for $b=1/3$.
\end{thm}
We prove Theorem \ref{thm:nodesc13} in Section \ref{s:ghosts}. Our methods parallel those of \cite{McDuffSchlenk12} and \cite{FrenkelMuller15} for the proofs of analogous results about $c_X$ for $X=B(1)$ and $P(1,1)$, including the use of ``ghost stairs.''

Together and with previous work, these results complete the classification of $b\in[0,1)$ into those with $O(b)>0$, and those with $O(b)=0$ and either no staircase, an ascending staircase, a descending staircase, or both. While the outline of the classification is given in Section \ref{s:review}, the actual classification depends in a complicated way on the continued fraction of the accumulation point corresponding to the $b$-value. For a visual and a brief explanation see Figure \ref{fig:block}.

\subsection{Proofs of Theorems \ref{thm:13proved} and \ref{thm:main}}

In this section we provide the definitions necessary to understand Theorems \ref{thm:NoRatlAcc} and \ref{cor:rpt}, and then use these together with results in \cite{MMW} to prove Theorems \ref{thm:13proved} and \ref{thm:main}.

Define
\begin{align*}
Stair&:=\{b\in[0,1)\,|\,c_b\text{ has an infinite staircase}\}, \mbox{ and }
\\Block&:=\{b\in[0,1)\,|\,O(b)>0\}.
\end{align*}
Note that \cite[Theorem~1.13]{AADT} implies that $Stair$ and $Block$ are disjoint.

In \cite{MMW}, the authors completely classified {\it Block} and gave an almost complete classification of {\it Stair}. The missing part of the classification was an infinite sequence of rational $b$-values known not to be in {\it Block} and conjectured not to be in {\it Stair}. We define these $b$-values now.

The \textbf{shift function} $S(z)=6-\frac1z$ is one of the symmetries studied extensively in \cite{MM}. Repeatedly applying this shift to $z=6$ generates the sequence of \textbf{special rational $z$-values}:
\begin{equation}\label{eqn:yk}
\frac{y_2}{y_1}=6, \quad \frac{y_3}{y_2}=\frac{35}{6}, \quad \frac{y_4}{y_3}=\frac{204}{35},\quad \frac{y_5}{y_4}=\frac{1189}{204}, \quad\ldots \quad y_k=6y_{k-1}-y_{k-2},
\end{equation}
where the $y_k$ were defined in \cite[Lemma~2.1.1~(iii)]{MM}. We also use the notation $p_i/q_i=y_{i+2}/y_{i+1}$. Notice that $y_{i+2}/y_{i+1}\to3+2\sqrt{2}=\acc(1/3)$. By \cite[Theorem 1.1.1~(iv)]{MMW}, the special rational $z$-values are the only rational numbers that might have staircases accumulating to them. 

Corresponding to these special rational $z$-values, set $b_i=\acc^{-1}_\eps(\frac{p_i}{q_i})$, where  $\eps=\pm 1$ corresponds to whether we take the upper ($+1$) or lower ($-1$) inverse of the function $b \mapsto \acc(b)$.\footnote{Note that this indexing is two less than that used in the definition of the $b_i$ in \cite[(2.3.5)]{MMW}.} The value of $\varepsilon$ is $(-1)^{i+1}$, so that we have
\begin{equation}\label{eqn:bidef}
b_0=\frac{1}{5}, \quad b_1=\frac{11}{31}, \quad, b_2=\frac{59}{179}, \quad b_3=\frac{349}{1045} \quad \dots \quad b_i=\acc^{-1}_{(-1)^{i+1}}\left(\frac{p_i}{q_i}\right).
\end{equation}
By \cite[Lemma 2.2.1]{MM}, all such $b$-values will also be rational, and are thus called \textbf{special rational $b$-values}. Moreover, the sequence of $b_i$s converges to $1/3$. The following result extends \cite[Theorem~6]{ICERM} establishing \cite[Conjecture~4.3.7]{MMW}, and is proved in Section \ref{s:nosratb}.

\begin{thm} \label{thm:NoRatlAcc}
    The ellipsoid embedding functions for the special rational $b$-values $b_i$, with $i\geq 0$, do not have infinite staircases. Consequently, there are no rational numbers $z=p/q$ that have staircases accumulating at $z$.
\end{thm}

\begin{remark} \rm 
A consequence of \cite[Theorem~1.1.1~(i, ii)]{MMW} and Theorem \ref{thm:NoRatlAcc} is that 
$$Stair\cup Block\cup\{b_0,b_1,\dots\}=[0,1).$$
Furthermore, the special rational $b$-values are the only $b$-values not in $Block$ which have 
rational accumulation points, as a consequence of \cite[Lemma~4.1.8, Corollary~4.1.9, Proposition~4.3.2]{MMW}. We expect that Theorem \ref{thm:main} generalizes to all finite toric blowups of $\CP^2$ (using their versions of $\acc$ and $O$), which would imply that no infinite staircase can have a rational accumulation point. There are only two other known cases of rational convex toric domains which have rational accumulation points and a vanishing staircase obstruction, namely $E(3,4)$ and $\CP^2(3)\#\oCP^2(1)\#\oCP^2(1)\#\oCP^2(1)\#\oCP^2(1)\#\oCP^2(1)$. These were shown in \cite{CGspecial} and \cite{AADT} respectively to not have infinite staircases by using ECH capacities, a different method than we use. 

The evidence for \cite[Conjecture~1.23]{AADT} given in \cite[Section 6]{AADT} presumes that the accumulation point is irrational and different methods are needed to deal with the case for rational accumulation points. The methods used in this paper to show the special rational $b$-values do not have infinite staircases also work for the other known targets with this property. There is some hope these methods could generalize to show there can never be a rational accumulation point where the staircase obstruction vanishes. 
\hfill$\er$
\end{remark}

Combining Theorem~\ref{thm:NoRatlAcc} with \cite[Theorem~1.1.1]{MMW}, there are then two types of staircases corresponding to $b$-values in {\it Stair} other than $b=1/3$: they are called principal and non-principal staircases. We explain why the case of $b=1/3$ does not fit into this framework at the beginning of Section~\ref{s:ghosts}. The usual algebraic definitions of these are provided in Section \ref{ss:fractal}, but by \cite[Theorem~1.1.1~(ii)]{MMW} and Theorem \ref{thm:NoRatlAcc}, they are equivalent to the following characterization. An infinite staircase is a \textbf{principal staircase} if $b\neq1/3$ and the nonsmooth points appear in a monotone sequence, either increasing with $z$ (``ascending'') or decreasing with $z$ (``descending''), or equivalently, if the accumulation point has a neighborhood in which either all nonsmooth points are larger than it or all are smaller than it. By \cite[Theorem~1.1.1~(ii)]{MMW} and \cite[Remark~2.1.3~(i)]{MMW}, the principal staircases occur only for (quadratic) irrational $b$.
An infinite staircase is a \textbf{non-principal staircase} if $b\neq1/3$ and the nonsmooth points cannot be arranged into a monotone sequence, in which case we also say that $c_b$ ``has both an ascending and a descending staircase''. We prove the following result in Section \ref{s:blocked}:

\begin{thm} \label{cor:rpt}
    The $b$-values for which $c_b$ has a non-principal staircase are irrational.
\end{thm}

The idea behind the proof is the following: if $b$ were rational, then by \eqref{eq:accb} $\acc(b)$ would be either rational or quadratic irrational. It is shown in \cite[Theorem~1.1.1~(iv)]{MMW} that if $\acc(b)$ is rational then $c_b$ does not have a non-principal staircase as it cannot have an ascending staircase. We prove in Section \ref{s:blocked} that if $\acc(b)$ is quadratic irrational then it also cannot have a non-principal staircase. We do this by characterizing the continued fractions resulting from the ``mutation'' procedure developed by \cite{MMW} and related to ATF mutation by \cite{M}. This investigation results in a proof of \cite[Conjecture 2.2.4]{MMW} and computes the numerics of all ``blocking classes.''

We now prove Theorem \ref{thm:13proved}.

\begin{proof}[Proof of Theorem \ref{thm:13proved}]
By \cite[Theorem~1.1.1~(ii)]{MMW}, there are four possibilities for $b$-values with infinite staircases: $b=1/3,$ special rational $b$-values, and $b$-values with either non-principal or principal staircases. Theorem \ref{thm:NoRatlAcc} shows that the special rational $b$-values do not have infinite staircases; note, this implies that the definition of principal and non-principal given in the current section agree with the definitions in Section~\ref{ss:fractal} and \cite{MMW}. Theorem \ref{cor:rpt} and \cite[Remark~2.1.3~(i)]{MMW} show, respectively, that the non-principal and principal staircases do not have rational $b$-values. Hence, the only rational $b$-value with a staircase is $b=1/3.$
\end{proof}
Next, we prove Theorem \ref{thm:main}.
\begin{proof}[Proof of Theorem \ref{thm:main}]
    
    Assume $c_b$ has an infinite staircase. Then $O(b)=0$ by \cite[Theorem~1.13]{AADT}. By \cite[Theorem~1.1.1~(iv)]{MMW}, the only rational numbers that might be accumulation points of staircases are the special rational $z$-values, which are shown not to be staircase accumulation points in Theorem~\ref{thm:NoRatlAcc}. Hence, $\acc(b)$ is irrational.

    For the opposite implication, assume that $\acc(b)$ is irrational and $O(b)=0$. By \cite[Theorem~1.1.1~(i, ii)]{MMW}, all $b$-values with $O(b)=0$ are either in $Stair$ or are special rational $b$-values. 
    Because $\acc(b)$ is irrational, then $b$ is not a special rational $b$-value, and hence, it must be in {\it Stair}. 
\end{proof}

\subsection{Future directions}
This paper concludes the study of the structure of infinite staircases for one-point blowups of $\CP^2$, and it would be natural to try to extend this study to more general targets. 

For example, \cite{CGMM} extends the definition of $\acc$ and the staircase obstruction $O$ to a broad class of symplectic four-dimensional manifolds, but goes on to prove that a simple extension of Theorem \ref{thm:main} does not hold in that case: if $z$ is irrational then $\acc(E(1,z))=z$ and $O(E(1,z))=0$ but $c_{E(1,z)}$ does not have an infinite staircase. 

A natural class of targets to consider is that of finite blowups of $\CP^2$ or $\CP^1\times\CP^1$. 
As we discuss in the remark below, ECH capacities provide a heuristic for why these targets, amongst (open or closed) symplectic four-manifolds with a symplectic torus action, are central to the problem. See \cite[Remark 1.21]{AADT} for a discussion of closed targets with no symplectic torus action.

\begin{remark} \rm
 In many cases, the ECH capacities $c_k$ of \cite{qECH} characterize $c_X$: see \cite{concaveconvex}. In all cases checked, the outer corners of infinite staircases occur when both the target $X$ and the ellipsoid $E(1,z)$ have infinitely many strings $c_k=\dots=c_{k+n}$ of equal ECH capacities (specifically, when the first entry of such a sequence for $E(1,z)$ has the same index as the last entry of such a sequence for $X$). By \cite[Lemma~3.1]{CGM}, if $\partial X$ is smooth and contains such a sequence of equal ECH capacities then $\partial X$ is Besse, that is, the induced Reeb vector field is periodic. If $X$  has a symplectic torus action, is open with smooth boundary, and $c_X$ has an infinite staircase, then $X$ must be a rational ellipsoid $E(1,1)$, $E(1,2)$, $E(2,3)$ (see \cite[Figure~1.10]{AADT} and \cite{CGspecial}). Therefore, the only cases this heuristic leaves to consider are when $\p X$ is not smooth. When we consider the closed counterparts of those, we get exactly finite blowups of $\CP^2$ or $\CP^1 \times \CP^1.$
\hfill$\er$
\end{remark}

Within the class of finite blowups of $\CP^2$ or $\CP^1\times\CP^1$, the following case is within reach: the moduli space of standard toric blowups $\CP^2(1)\#\oCP(b_1)\#\oCP(b_2)$ with $1\geq b_1\geq b_2\geq0$ and $0<b_1+b_2\leq1$, depicted in Figure \ref{fig:2bl}. \footnote{Denoting by $\sim$ the relation of having the same ellipsoid embedding functions, by \cite[Theorem~1.4]{AADT}, we have $\CP^2(1)\sim B(1)$ and
$\CP^2(1)\#\oCP^2(b_1)\#\oCP^2(1-b_1)\sim P(b_1,1-b_1)\sim \CP^1(b_1)\times \CP^1(1-b_1)$.}

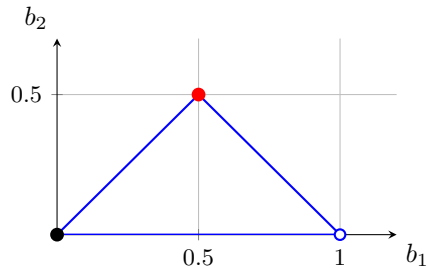
\begin{figure}[ht!]
\begin{center}
\begin{tikzpicture}
\begin{axis}[
	axis lines = middle,
	xtick = {0,1/2,1},
	ytick = {0,0.5},
	tick label style = {font=\small},
	xlabel = $b_1$,
	ylabel = $b_2$,
	xlabel style = {below right},
	ylabel style = {above left},
	xmin=0,
	xmax=1.2,
	ymin=0,
	ymax=0.7,
	grid=major,
	width=2.4in,
	height=1.65in]
 \addplot[blue, thick, domain = 1/2:1]{1-x};
\addplot[blue, thick, domain = 0:1]{0};
\addplot[blue, thick, domain = 0:1/2]{x};
\addplot [black, only marks, very thick, mark=*] coordinates{(0,0)};
\addplot [red, only marks, very thick, mark=*] coordinates{(1/2,1/2)};
\addplot [blue, fill = white, only marks, thick, mark=*] coordinates{(1,0)};
\end{axis}
\end{tikzpicture}
\end{center}
\caption{The black dot represents the ball and the red dot the ``cube'' $P(1/2, 1/2)$. Points along the $b_1$ axis are Hirzebruch surfaces $H_{b_1}$, while points along the line $b_1+b_2=1$ are polydisks $P(b_1,b_2)$. The balanced blowups are represented by the line $b_1=b_2$. }\label{fig:2bl} 
\end{figure}
 
Together with \cite{SPUR} and \cite{usher}, this paper suggests that the numerics of the ellipsoid embedding functions of polydisks and $H_b$ surfaces are very closely paired (see in particular Remark \ref{rmk:polydisk}) and that the structure of $Stair$ and $Block$ for these two one-parameter families of manifolds represented by the right and bottom edges of the triangle in Figure \ref{fig:2bl} will be nearly identical. See \cite[Conjecture~1.2.1]{SPUR} for a precise statement.

For the family of balanced blowups of $\CP^2$, represented by the left edge of the triangle in Figure \ref{fig:2bl}, the numerics seem to be influenced by those of both infinite staircases at the end vertices of that edge: the Fibonacci staircase of the ball and the Pell staircase of the polydisk $P(1/2,1/2)$, but we expect some of the same structures to appear. 

Preliminary investigations in \cite{M2} about what happens in the interior of the triangle indicate that $Block$ includes overlapping two-dimensional wedges which, restricted to the edges of the triangle, give the corresponding $Block$ sets. We expect (the appropriate versions of) Theorems \ref{thm:13proved} and \ref{thm:main} to extend to the whole triangle.

\subsection{Outline of the paper}

In Section \ref{s:review} we collect from previous papers and the current paper to give a full answer to the question of for which $b$-values the embedding function of $H_b$ has an infinite staircase, no proofs included. To do this, we review the tools we use to study infinite staircases. In particular, we explain ``exceptional classes,'' certain second homology classes in blowups of $\CP^2$ which provide the steps of an infinite staircase. In Section \ref{ss:fractal} we review the fractal ``mutation'' framework of \cite{MMW}, and in \ref{ss:symmetries} we review the linear fractional transformations of the $z$-variable from \cite{MM} which allow us to restrict our attention to $b\in((3-\sqrt{5})/2,1)$ when proving Theorem \ref{cor:rpt}.

We open Section \ref{s:blocked}, in which we prove Theorem \ref{cor:rpt}, with the auxiliary Theorem \ref{thm:CFmain}, which allows us to describe the effect of mutation using the continued fraction of the $z$-variable. One notable consequence is Corollary \ref{cor:MMWconj}, in which we prove \cite[Conjecture~2.2.4]{MMW}. Theorem \ref{thm:CFmain} is proved in Appendix \ref{app:pf}. In Section \ref{ss:afternumerics} we complete the proof of Theorem \ref{cor:rpt}.

 Theorem \ref{thm:NoRatlAcc} is proved in Section \ref{s:nosratb}. The main tool in this section is the ``Cremona move,'' a change-of-basis on second homology classes of a blowup of $\CP^2$.

Finally, in Section \ref{s:ghosts}, we prove that $c_{1/3}$ does not have a descending infinite staircase. The proof heavily relies on \cite[Section 6]{FrenkelMuller15}, which itself is inspired by \cite[Section 4]{McDuffSchlenk12}. In particular, we again use a version of the ``ghost stairs,'' which are a special set of exceptional classes that in the case of $B(1)$ have proven relevant to the stabilized embedding problem: see \cite{cghm}, for example.

Section \ref{s:review} is crucial to understanding the rest of the paper, but Sections \ref{s:blocked} to \ref{s:ghosts} are logically independent and can be read in any order, except for the use of Cremona moves (defined in Section 4) in Section 5.

\subsection{Acknowledgements}

We would like to thank the collaborators from the original paper \cite{ICERM}, Maria Bertozzi, Tara Holm, Emily Maw, Dusa McDuff, and Grace Mwakyoma, as well as AWM and ICERM for organizing and hosting the Women in Symplectic and Contact Geometry and Topology workshop in 2019, where this project began. We would like to further thank Tara, Dusa, and Grace for inspiring conversations about this paper specifically and Dan Cristofaro-Gardiner and Felix Schlenk for invaluable conversations about Section \ref{s:nosratb} and Section \ref{s:ghosts} respectively. We also thank Caden Farley, Jemma Schroder, Zichen Wang, and Elizaveta Zabelina for helpful discussions. We sincerely thank the referees for their outstandingly careful reading and the improvements they suggested. Finally, we thank the ICMS for hosting us under a Research in Groups in 2022. Morgan Weiler and Nicki Magill were supported by NSF grant DMS 2103245 and NSF Graduate Research Grant DGE-1650441, respectively, and Ana Rita Pires  was supported by an LMS Emmy Noether Fellowship. 

\section{For which $b$ does $H_b$ have an infinite staircase?}\label{s:review}

In this section, we detail the structure of $Block$ and $Stair$. In Section \ref{ss:reviewEC} we review exceptional classes in blowups of $\CP^2$, explain their relevance to the problem of embedding ellipsoids into finite type convex toric domains, and collect some useful definitions and facts.

Next, in Section \ref{ss:fractal} we detail the structure of $Block$ and $Stair$ for values of $b$ in $((3-\sqrt{5})/2,1)$, following \cite{MMW}. In Section \ref{ss:symmetries} we see how applying symmetries to the intervals in Section \ref{ss:fractal} determines $Block$ and $Stair$  for values of $b$ in $[0,(3-\sqrt{5})/2]$, following \cite{MM, MMW}.

Before we begin, we review the definition of {\bf weight decomposition} of a number $z\geq1$: the weight decomposition $\bw(z):=\left(w_1^{\times \ell_1},w_2^{\times \ell_2},w_3^{\times \ell_3},\hdots\right)$ is obtained by  inductively decomposing a rectangle with side lengths 1 and $z$ into squares as large as possible. 
Specifically, we set $w_0=z$ (note, $w_0$ is not in $\bw(z)$), $w_1=1$, and $w_{k}=w_{k-2}-\ell_{k-1} w_{k-1}$ where $\ell_{k}=\lfloor \frac{w_{k-1}}{w_{k}} \rfloor$. The continued fraction of $z$ is then $[\ell_1,\dots,\ell_n]$.

The \textbf{weights} of $z$ are the entries in its weight expansion. Rational numbers have finite weight expansions (by the fact that the Euclidean algorithm terminates), while irrational numbers have infinite weight expansions. We will also use the notation $\bw(z)=(w_1,\dots,w_N)$ repeating the entries according to their multiplicities, and our convention will be clear from context.

We note that by \cite[Lemma~1.2.6]{McDuffSchlenk12}, if $\mathbf{w}(p/q)=(w_1,\dots,w_N)$, then
    \begin{equation} \label{eq:rmkweights}
        \sum_{i=1}^Nw_i^2=\frac{p}{q} \quad \text{and} \quad 
        \sum_{i=1}^Nw_i=\frac{p}{q}+1-\frac{1}{q}.
    \end{equation}

\subsection{Review of exceptional classes}\label{ss:reviewEC}

Embeddings of rational ellipsoids into finite type convex toric domains are completely characterized by symplectically embedded spheres in blowups of $\CP^2$, a method due to McDuff (particularly the proof of \cite[Proposition~3.2]{m2}) and generalized by Cristofaro-Gardiner (in the proof of \cite[Theorem~1.2]{concaveconvex}). We review these ideas here, following the proof of \cite[Proposition~3.2]{m2}.

Expanding on \cite[Theorem~3.11]{Mc0} as in \cite[Theorem~2.1]{concaveconvex}, if $\mathbf{w}(z)=(w_1,\dots,w_N)$, then 
\begin{equation}\label{eqn:ballembeddings}
E(1,z)\sembeds \lambda H_b\iff \bigsqcup_{i=1}^N B(w_i) \bigsqcup B(\lambda b) \sembeds B(\lambda).
\end{equation}
A symplectic embedding of $N+1$ balls into a target ball can be translated to an $N+1$-fold blowup of $\CP^2$ in the following way: the boundary of the target ball is collapsed to create the $\CP^2$ and each of the embedded balls is symplectically blown up to create an exceptional divisor. 
Thus our problem now becomes one of identifying symplectic forms on blowups of $\CP^2$; more precisely, work of McDuff and Polterovich in \cite{McDuffPolt} implies that the embedding on the right-hand side of \eqref{eqn:ballembeddings} exists if and only if there is a symplectic form $\omega$ on the $(N+1)$-fold blowup $\CP^2\#(N+1)\overline{\CP}^2$ such that:
\begin{itemize}
    \item[(A)] its cohomology class satisfies
\[
PD([\omega])=\lambda L-\lambda b E_0-\sum_{i=1}^Nw_i E_i,
\]
where $L$ denotes the homology class of the line in $\CP^2$ and $E_i$ denotes the class of the $i$th exceptional divisor, and
\item[(B)] its first Chern class satisfies
\[
PD(c_1(\omega))=-K:=3L-\sum_{i=0}^NE_i.
\]
\end{itemize}

It is now useful to introduce the notion of \textbf{exceptional class}, i.e., $\bE\in H_2\left(\CP^2\#(N+1)\overline{\CP}^2\right)$ of self-intersection -1 (that is,  $\bE\cdot\bE=-1$), such that $-K(\bE)=1$, and which can be represented by a symplectically embedded $S^2$. Writing $\bE$ as
\begin{equation}\label{eqn:E}
\bE=(d;m,\bm)=dL-mE_0-\sum_{i=1}^Nm_iE_i,
\end{equation}
where $\bm=(m_1,\dots,m_N)$, the conditions  $-K(\bE)=1$ and $\bE\cdot\bE=-1$ become the following \textbf{Diophantine equations}:
    \begin{equation}\label{eqn:D}
    3d-m=\sum_im_i+1 \mbox{ and } d^2-m^2=\sum_im_i^2-1.
    \end{equation}

Here we used that the intersection pairing of the basis elements is:
\[
L\cdot L=1, \quad L\cdot E_i=E_i\cdot E_j=0, \quad E_i\cdot E_i=-1.
\]
The condition that $\bE$ is representable by a symplectically embedded sphere is,  by work of Li and Li \cite{li-li}, equivalent to the vector $(d;m,\bm)$ reducing to $(0;-1)$ under repeated ``Cremona moves'', which essentially express $\bE=(d;m,\bm)$ in a new basis. See the beginning of Section \ref{s:nosratb} for the definition of Cremona moves.

Finally, work of Li and Liu in \cite{li-liu} using Taubes' symplectic interpretation of Seiberg-Witten theory implies that a symplectic form $\omega$ on $\CP^2\#(N+1)\overline{\CP}^2$ satisfying (A) and (B) above exists in any cohomology class $\alpha$ for which
\begin{itemize}
    \item[(1)] $\alpha^2>0$, and
    \item[(2)] $\alpha(\bE)>0$ for all exceptional classes $\bE$.
\end{itemize}

So we have seen that the symplectic embedding problem of \eqref{eqn:ballembeddings} is equivalent to the existence of a symplectic form $\omega$ whose class $\alpha=[\omega]\in H^2\left(\CP^2\#(N+1)\overline{\CP}^2\right)$ satisfies (1) and (2), or equivalently (A) and (B).

Using conditions (1) and (A) we obtain the volume lower bound \eqref{eqn:volb}:
$$\left(\lambda L-\lambda b E_0-\sum_{i=1}^Nw_iE_i\right)^2>0 \iff \lambda^2>\frac{z}{1-b^2} \iff \lambda>\vol_b(z).$$

Using conditions (2) and (A), we obtain another lower bound:
\begin{equation}\label{eqn:muEb}
\left(\lambda L-\lambda b E_0-\sum_{i=1}^Nw_iE_i\right)\left(dL-mE_0-\sum_{i=1}^Nm_iE_i\right)>0 \iff \lambda>\frac{\bm\cdot\mathbf{w}(z)}{d-bm}=:\mu_{\bE,b}(z)
\end{equation}

These two bounds completely determine the ellipsoid embedding function for rational $z$ values:
\begin{equation}\label{eqn:cbsup}
c_b(z)=\sup_{\bE\text{ exceptional}}\left\{\mu_{\bE,b}(z),\vol_b(z)\right\}.
\end{equation}
The formula \eqref{eqn:cbsup} holds even when $z$ is irrational by appending infinitely many zeroes to the end of $\bm$ (the fact that the obstructions $\mu_{\bE,b}$ are continuous when larger than $\vol_b(z)$ follows from \cite[Proposition~2.30]{AADT}).

A tuple $\bE=(d;m,\bm)$ of integers is called \textbf{Diophantine} if its entries satisfy \eqref{eqn:D}, irrespective of whether or not $\bE$ can be represented by a symplectically embedded sphere. 
As seen in Hutchings in \cite[Theorem~2]{Hutch}, the supremum in \eqref{eqn:cbsup} can in fact be taken over all Diophantine classes, which greatly simplifies computations:
\[
c_b(z)=\sup_{\bE\text{ Diophantine}}\left\{\mu_{\bE,b}(z),\vol_b(z)\right\}.
\]

Next, we collect definitions and key facts about exceptional classes from \cite{AADT, ICERM, MM}.

An obstruction $\mu_{\bE,b}$ where $\bE$ is Diophantine class is called \textbf{live} at $z$ if 
\begin{equation}\label{eqn:live}c_b(z)=\mu_{\bE,b}(z)>\vol_b(z).\end{equation}
If an obstruction satisfies the inequality part of \eqref{eqn:live} at the accumulation point $\acc(b)$
\begin{equation}\label{eqn:Eblocks}
\mu_{\bE,b}(\acc(b))>\vol_b(\acc(b)),
\end{equation}
then by \cite[Theorem~1.13]{AADT} the function $c_b$ cannot have an infinite staircase. For that reason, we say that the Diophantine class $\bE$ \textbf{blocks} $b$ or is a \textbf{blocking class} for $b$.

If $\bE$ is a blocking class, then its \textbf{blocked $b$-interval} is the set of $b$-values $J_\bE\subset[0,1)$ for which \eqref{eqn:Eblocks} holds. Its \textbf{blocked $z$-interval} is the set $I_\bE=\acc(J_\bE)$. We will use the simpler phrase ``blocked interval'' when $b$ or $z$ is clear from context.

We say that $\bE=(d;m,\bm)$ is \textbf{quasi-perfect} if it is Diophantine and $\mathbf{m}=q\,\bw(p/q)$ for coprime integers $p>q$,\footnote{The vector $q\,\bw(p/q)$ is the smallest integer scaling of the weight expansion $\bw(p/q)$.} and we say that $\bE$ is \textbf{perfect} if it is also exceptional. Quasi-perfect classes turn out to be the most relevant classes in proving that an infinite staircase exists; their properties below will be used throughout the rest of the paper.

We note that when $\bE$ is quasi-perfect, by \eqref{eq:rmkweights} the Diophantine equations in \eqref{eqn:D} become
\begin{equation}\label{eqn:Ds}
    3d-m=p+q \quad \text{ and } \quad d^2-m^2=pq-1.
\end{equation}
For quasi-perfect $\bE$, we call $p/q$ its \textbf{center}. Computing $\mu_{\bE,b}$ at the center of $\bE$ is particularly simple by \eqref{eq:rmkweights}:
\[
    \mu_{\bE,b}\left(p/q\right)=\frac{q\bw(p/q)\cdot\bw(p/q)}{d-bm}=\frac{p}{d-bm}.
\]
In fact, we can compute $\mu_{\bE,b}$ on the entire $z$-interval $I$ where $\mu_{\bE,b}(z)>\vol_b(z)$: it is a linear (increasing) function of $z$ up to $p/q$ and then constant; see Lemma \ref{lem:cghmp23}, Figure \ref{fig:2}, and \cite[Lemma~16]{ICERM}.

Following \cite[(2.2.3), Corollary~2.2.2]{MM}, we can solve the system of equations in \eqref{eqn:Ds} for $d$ and $m$ in terms of $p$ and $q$ and get
\begin{equation}\label{eq:dmfrompqt}d=\tfrac18(3(p+q)+\eps t) \quad \text{and} \quad 
 m=\tfrac18((p+q)+3\eps t)
 \end{equation}
where $\eps=\pm1$ is whichever one makes the formulas above integers \cite[Lemma~2.2.1]{MM} (in our case,  if $m/d>1/3$ then $\eps=1$, and otherwise $\eps=-1$) and
we define
\begin{equation}\label{eq:deft}
t:=\sqrt{p^2+q^2-6pq+8}.
\end{equation}
 
 In this sense, $d$ and $m$ are linear functions of $p,q$, and $t$, and these determine the quasi-perfect class $\bE$. Therefore we will frequently use the notation
\begin{equation}\label{eqn:ECF}
\bE=(d,m,p,q,t)=(d,m,p,q)=\bE_{[a_0,\hdots,a_k]},
\end{equation}
where $[a_0,\hdots,a_k]=CF(p/q)$ denotes the continued fraction 
$$\frac pq=a_0+\frac{1}{a_1+\frac{1}{a_2+\ldots}}.$$

The variable $t$ is crucial to the fractal structure explained in Section \ref{ss:fractal} and the proof of Theorem \ref{thm:CFmain} in Appendix \ref{app:pf}. Another way to encode its definition \eqref{eq:deft} is by writing 
\[
\mathbf{x}=\begin{pmatrix}p\\q\\t\end{pmatrix}, \mbox{ where }\mathbf{x}^TA\mathbf{x}=8 \mbox{ for }A=\begin{pmatrix}-1&3&0\\3&-1&0\\0&0&1\end{pmatrix}.
\]

Given two quasi-perfect classes $\bE_0,\bE_1$ with $\bE_i=(d_i,m_i,p_i,q_i,t_i)$ and $\mathbf{x}_i=(p_i,q_i,t_i)$, we can take a linear combination of the form
\[
\bE_2=\nu\bE_1-\bE_0 \quad \mbox{and} \quad \mathbf{x}_2=\nu\mathbf{x}_1-\mathbf{x}_0.
\]
By \cite[Lemma~3.1.2]{MM}, if $\bE_0$ and $\bE_1$ are \textbf{$\nu$-compatible}, i.e., if $$\mathbf{x}_0^TA\mathbf{x}_1=4\nu,$$
then the resulting class $\bE_2$ is also quasi-perfect. Furthermore, $\bE_1$ and $\bE_2$ are $\nu$-compatible. Thus, by iterating this process we can produce a sequence of quasi-perfect classes from the initial {\bf seeds} $\bE_0$ and $\bE_1$.

\subsection{The structure of $Block$ and $Stair$ for $b\in((3-\sqrt{5})/2,1)$}\label{ss:fractal}

In this section, we will first review from \cite{ICERM} the exceptional classes which block the existence of an infinite staircase for the $b$-values on the red intervals in Figure \ref{fig:block}. Then we will explain the fractal structure of $Stair$ in the gaps between those intervals, following \cite{MMW}. 

The interior of each red interval in Figure \ref{fig:block} is the blocked $b$-interval $J_{\bE_{[n]}}$, with $n\geq 6$ and even,
where $\bE_{[n]}$ according to \eqref{eqn:ECF} is the perfect class\footnote{In \cite{ICERM, MM, MMW}, for $k\geq0$ the class $\bE_{[2k+6]}$ is denoted $\mathbf{B}^U_k$.} 
\[
\bE_{[n]}=\left(\frac n2,\frac n2-1,n,1,n-3\right).
\]
While this sequence of blocked $b$-intervals converges to the point $b=1$ (by \cite[Remark~2.3.6]{MM}, $(n-4)/(n-2) \in J_{\bE_{[n]}}$ for each $n$), the corresponding blocked $z$-intervals $I_{\bE_{[n]}}$ go off to infinity, with $n \in I_{\bE_{[n]}}$ because $n$ is the center of $\bE_{[n]}$. Explicit formulas for the endpoints of the intervals $J_{\bE_{[n]}}$ can be obtained by solving 
$$\mu_{\bE,b}(\acc(b))=\vol_b(\acc(b)))$$
for $b$ (see \eqref{eqn:Eblocks}) and are given in \cite[Theorem~1]{ICERM}.

\begin{landscape}

\begin{figure} 
\centering
\includegraphics[width=8.7in]{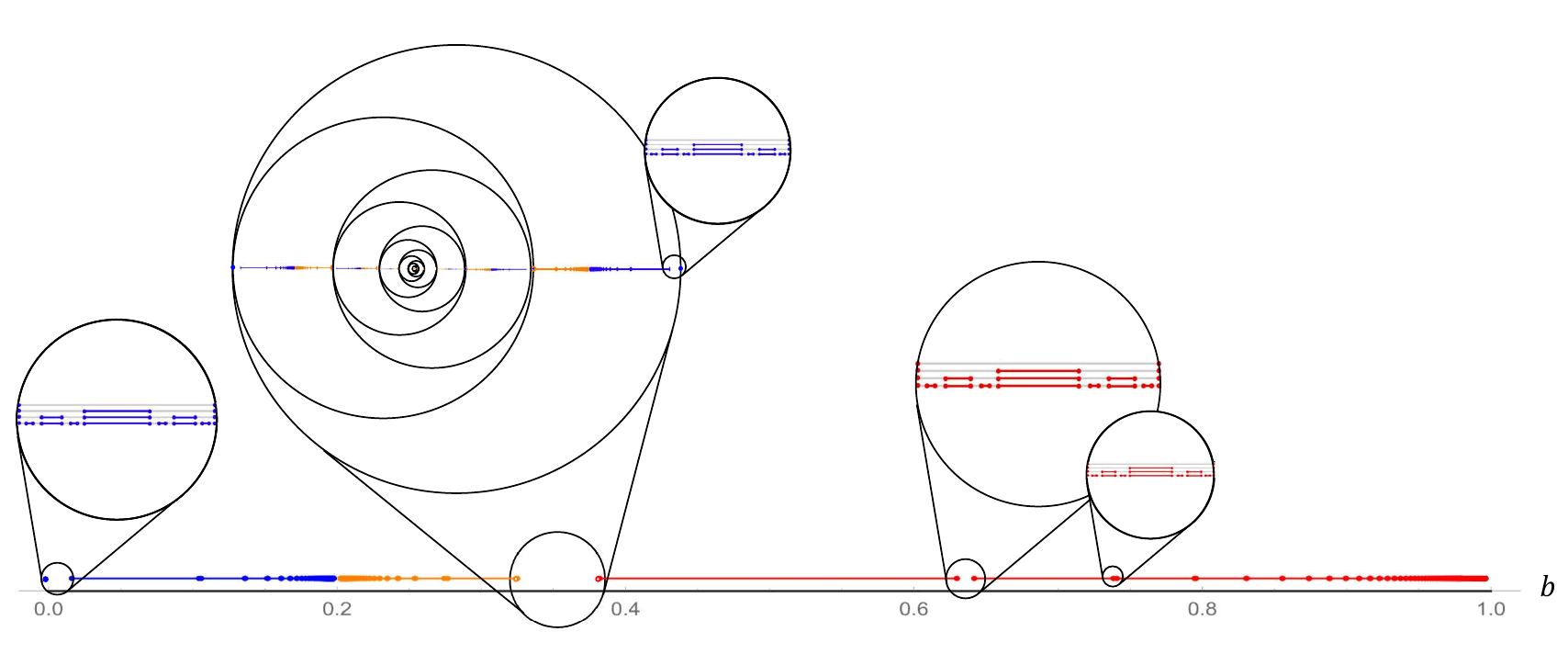} 
 \caption{This figure depicts the moduli space of one-point blowups of $\CP^2(1)$ parametrized by the size of the blowup $b\in[0,1)$. The set $Block$ of $b$-values with $O(b)>0$ consists of the union of the interiors of the colored intervals. The complement of $Block$ consists of $Stair$ and the special rational $b$-values.
 The special rational $b$-values are the common limit points of the blue and orange intervals (e.g. $b=0.2$) and they do not have infinite staircases. The set $Stair$ consists of $b$-values with principal staircases, $b$-values with non-principal staircases, and the exceptional value $b=1/3$.
 The principal staircases occur at the endpoints of the colored intervals, where one endpoint has only an ascending staircase and the other only a descending one. 
 In the figure, there are circles with Cantor sets; the non-principal staircases occur at the $b$-values in those Cantor sets that are not endpoints of colored intervals and have both an ascending and descending staircase.
 Finally, the exceptional value $b=1/3$, which has only an ascending staircase, appears as the limit of the nested circles in the figure.
 The structure of the region with red intervals is explained in Section \ref{ss:fractal} and the structure of the region with orange and blue intervals is explained in Section \ref{ss:symmetries}.}
 \label{fig:block}
\end{figure}

\end{landscape}

The endpoints of each $J_{\bE_{[n]}}$  are $b$-values for which $c_b$ has an infinite staircase. The staircases corresponding to the left endpoints of the $J_{\bE_{[n]}}$ intervals are ascending, whereas the ones corresponding to the right endpoints are descending. This family was discovered in \cite{ICERM}. To show that these are infinite staircases, for each value of $b$ an infinite sequence of perfect classes $\bE_k$ is constructed such that each obstruction $\mu_{\bE_k,b}$ is live near the center $p_k/q_k$ of the class, and furthermore the sequence of centers converges to the accumulation point:
$$\lim_{k \to \infty} \frac{p_k}{q_k}=\acc(b).$$
Each of these classes contributes a non-smooth point (and hence a step) to the graph of the function $c_b$, thus proving that it is an infinite staircase.

For $n>6$ even, the sequence of classes with seeds $\bE_{[n-2]}$ and $\bE_{[n-1,n-4]}$ (see \eqref{eqn:ECF}) proves the existence of an ascending staircase corresponding to the left endpoint of the blocked interval $J_{\bE_{[n]}}$. For the sake of brevity, we will say that the staircase has seeds $\bE_{[n-2]}$ and $\bE_{[n-1,n-4]}$.
For $n\geq6$ even, the descending staircase corresponding to the right endpoint of the blocking interval $J_{\bE_{[n]}}$ has seeds $\bE_{[n+2]}$ and $\bE_{[n+1,n-2]}$. The left endpoint of the leftmost red blocking interval $J_{\bE_{[6]}}$, which is 
$$b_{[6]}^-:=(3-\sqrt{5})/2,$$
also corresponds to an ascending infinite staircase and it will be discussed in Section \ref{ss:symmetries}.

Next, we look at the fractal structure of $Stair$ in the gaps between the red blocked intervals in Figure \ref{fig:block}. Each gap sits between a $b$-value with a descending staircase to its left (the right endpoint of $J_{\bE_{[n]}}$) and a $b$-value with an ascending staircase to its right (the left endpoint of $J_{\bE_{[n+2]}}$): we denote this gap by $G_n$. As explained above, the descending staircase has seeds $\bE_{[n+2]}$ and $\bE_{[n+1,n-2]}$, and the ascending staircase has seeds $\bE_{[n]}$ and $\bE_{[n+1,n-2]}$. These three classes
\begin{equation}\label{eq:threeclasses}
\bE_{[n]},\bE_{[n+1,n-2]},\bE_{[n+2]}
\end{equation}
will generate the whole structure of $Block$ and $Stair$ in $G_n$. To understand how, we review from \cite{MMW} the notions of adjacent classes, generating triple and mutation of a triple.

We say that two quasi-perfect classes $\bE_0,\bE_1$ with $\bE_i=(d_i,m_i,p_i,q_i,t_i)$
where $p_0/q_0<p_1/q_1$ and $\varepsilon_0=\varepsilon_1$ are \textbf{adjacent} if 
$$(p_0+q_0)(p_1+q_1)-t_0 t_1=8p_0 q_1.$$  
By \cite[Lemma~2.1.2~(iii)]{MMW}, two adjacent classes $\bE_0$ and $\bE_1$ are $\nu$-compatible if and only if $|p_1 q_0-p_0 q_1|=\nu$. 

Following {\cite[Definition~2.1.6]{MMW}}, we say that three quasi-perfect classes    \[
    \bE_\la=(d_\la, m_\la, p_\la, q_\la, t_\la),\quad \bE_\mu=(d_\mu, m_\mu, p_\mu, q_\mu, t_\mu),\quad \bE_\rho=(d_\rho, m_\rho, p_\rho, q_\rho, t_\rho)
    \]
    with $p_\la/q_\la<p_\mu/q_\mu<p_\rho/q_\rho$ and $t_\la, t_\mu, t_\rho\geq3$ form a \textbf{generating triple} $\Tt=(\bE_\la,\bE_\mu,\bE_\rho)$ if
    \begin{itemize}
        \item[{\rm (a)}] $\bE_\la$ and $\bE_\rho$ are adjacent,
        \item[{\rm (b)}] $\bE_\la$ and $\bE_\mu$ are adjacent and $t_\rho$-compatible, i.e. $t_\rho=|p_\mu q_\la-p_\la q_\mu|$,
        \item[{\rm (c)}] $\bE_\mu$ and $\bE_\rho$ are adjacent and $t_\la$-compatible, i.e. $t_\la=|p_\mu q_\rho-p_\rho q_\mu|$,
        \item[{\rm (d)}] $t_\la t_\rho-t_\mu=p_\rho q_\la-p_\la q_\rho$, and
        \item[{\rm (e)}] $\acc(m_\la/d_\la), \acc(m_\rho/d_\rho)>\acc(m_\mu/d_\mu)$.
    \end{itemize}

A mutation of a generating triple is an operation which takes a linear combination of two of the compatible classes in the original triple and uses it to create a new generating triple, as in {\cite[Proposition~2.1.9]{MMW}}:\footnote{An incipient notion of $x$-mutation appears already in \cite[Section 4.6]{usher} when he constructs the $\hat{A}$ classes from the $A$ classes.}
\begin{itemize}
\item the \textbf{$x$-mutation} of $\Tt=(\bE_\la,\bE_\mu,\bE_\rho)$ is the triple $x\Tt:=(\bE_\la,\bE_{x\mu},\bE_\mu)$, where $\bE_{x\mu}:=t_\la\bE_\mu-\bE_\rho$,
\item the \textbf{$y$-mutation} of $\Tt=(\bE_\la,\bE_\mu,\bE_\rho)$ is the triple $y\Tt:=(\bE_\mu,\bE_{y\mu},\bE_\rho)$, where  $\bE_{y\mu}:=t_\rho\bE_\mu-\bE_\la.$
\end{itemize}

\begin{figure}[htbp]\label{fig:2} 
\vspace{-1.1in}  
\hspace{1in}
\includegraphics[width=7in]{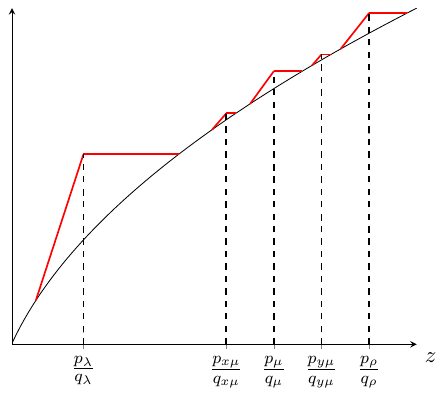} 
\vspace{-6in}
   \caption{\cite[Fig.~1.3]{MMW} This schematic figure shows the relative locations of the centers 
   of the classes obtained from a triple $\Tt=(\bE_\la,\bE_\mu,\bE_\rho)$ by $x$ and $y$ mutations along with the obstruction functions near their centers. The curve in black is the graph of the function $z\mapsto  V_b(z)$. If any of these obstructions are live for this value of $b$, then those will contribute a step towards the graph of $c_b$, possibly forming an infinite staircase (or two).}
\end{figure}

The fact that $x\Tt$ is a generating triple means in particular that the center of the new quasi-perfect class $\bE_{x\mu}$ is between the centers of the classes $\bE_\lambda$ and $\bE_\mu$, and similarly for $y\Tt$ and its classes: (see also Figure \ref{fig:2})
$$\frac{p_\lambda}{q_\lambda}<\frac{p_{x\mu}}{q_{x\mu}}<\frac{p_\mu}{q_\mu}<\frac{p_{y\mu}}{q_{y\mu}}<\frac{p_\rho}{q_\rho}.$$

A generating triple plus a sequence of $x$'s and $y$'s determines a sequence of generating triples via mutations. A sequence $(y,x,y,y,y\ldots)$, which we will usually encode in the form of a word $w=\ldots y y y x y$, when applied to the triple $\Tt=(\bE_\lambda,\bE_\mu,\bE_\rho)$ generates the triples 
\begin{equation}\label{eq:sequencewords}
w_1\Tt=y\Tt,\,\,\, w_2\Tt=xy\Tt,\,\,\, w_3\Tt=yxy\Tt, \,\,\,w_4\Tt=yyxy\Tt, \,\,\,w_5\Tt=yyyxy\Tt,\,\,\, \ldots
\end{equation}
In fact, each triple in this sequence can be uniquely represented by its middle class, since the middle class of $w_k\Tt$ is exactly $\bE_{w_k\mu}=:w_k\bE_{\mu}$, where $w_k$ is a finite word on $x$ and $y$ obtained by truncating the given sequence as in \eqref{eq:sequencewords}.

We are now ready to see how the three classes in \eqref{eq:threeclasses}  generate the structure of $Block$ and $Stair$ on the gap $G_n$ for each $n\geq6$ even.
As proved in \cite[Example~2.1.7]{MMW}, these classes form a generating triple (in \cite{MMW}, the notation used is $\Tt^{n/2}_*$):
  \[
    \Tt_n:=(\bE_{[n]},\bE_{[n+1,n-2]},\bE_{[n+2]}).
    \]
For each such $n$, we consider all the sequences of generating triples obtained by applying a sequence of $x$'s and $y$'s to the generating triple $\Tt_n$.
On the gap $G_n$, the set $Block$ consists exactly of the interiors of all intervals $J_{w\bE_{[n+1,n-2]}}$ blocked by the middle class $w\bE_{[n+1,n-2]}$ of a triple $w\Tt_n$, for some word $w$ on $x$ and $y$. Figure \ref{fig:block} shows the relative position of these blocking $b$-intervals; together they form a set that is the complement of the Cantor set (\cite[Theorem~1.1.1~(iii)]{MMW}). 

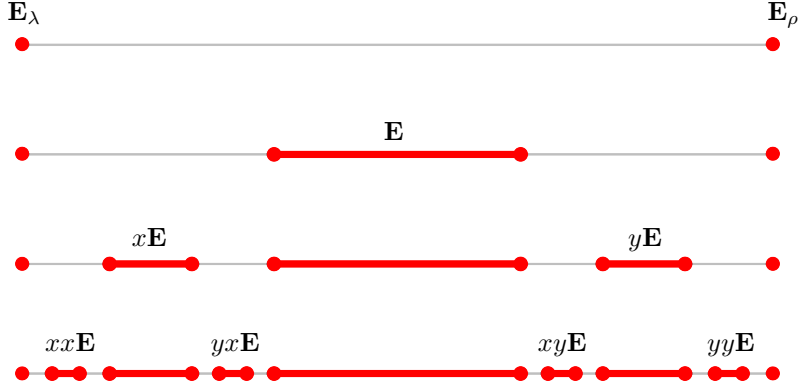
\begin{figure}[htbp]\label{fig:cantor}
\begin{center}
\vspace{0.5 cm}
\begin{overpic}[width=4in, unit=1mm]{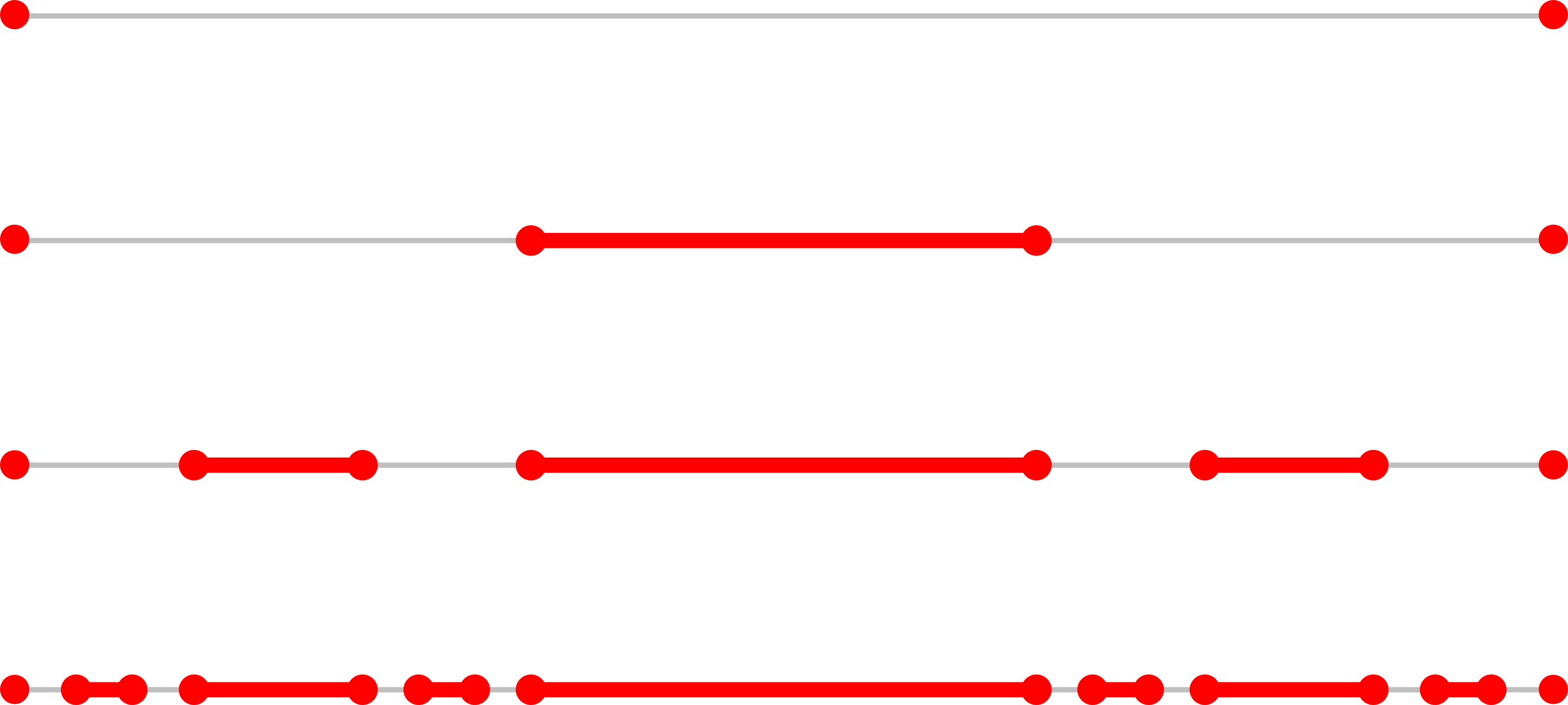}
\put(-1,48){$\bE_\la$}
\put(100,48){$\bE_\rho$}
\put(49,32){$\bE$}
\put(15.5,18){$x\bE$}
\put(81.5,18){$y\bE$}
\put(4,4){$xx\bE$}
\put(26,4){$yx\bE$}
\put(69.5,4){$xy\bE$}
\put(92,4){$yy\bE$}
\end{overpic}
\end{center}
   \caption{This figure shows the relative position of the blocking $b$-intervals $J_{w\bE}$ where $w$ is a word on $x$ and $y$, $w\bE$ is the middle class of the generating triple $w\Tt$, and $\Tt=(\bE_\la,\bE,\bE_\rho)$.}
\end{figure}

The set $Stair$ restricted to each gap $G_n$ is homeomorphic to the Cantor set (\cite[Theorem~1.1.1~(iii)]{MMW}). The (middle third) Cantor set consists of the numbers in $[0,1]$ which can be represented in base 3 using only zeros and twos. Replacing the zeros and twos by $x$'s and $y$'s respectively determines a word $w$ that can then be applied to the generating triple $\Tt_n$:
$$0.20222\ldots  \longrightarrow \ldots y y y x y =: w.$$
This gives an infinite sequence of classes of the form $w_k\bE_{[n+1,n-2]}$ (with centers $p_k/q_k$) and corresponding intervals $J_{w_k\bE_{[n+1,n-2]}}$ on $G_n$. These intervals actually converge to a single $b$-value on $G_n$ and, as shown in \cite[Section 3]{MMW}, infinitely many of the classes $w_k\bE_{[n+1,n-2]}$ are live for this $b$-value and contribute with a step to form an infinite staircase accumulating at $\acc(b)$, with 
\begin{equation} \label{eq:centerLimit} \lim_{k \to \infty} \frac{p_k}{q_k}=\acc(b).\end{equation}

There are two types of points in the Cantor set: those whose base three representation ends with either infinitely many zeros or infinitely many twos (these are endpoints of one of the removed intervals), and those whose do not.
For a $b$-value in $G_n$ that is the left endpoint of some interval $J_{w\bE_{[n+1,n-2]}}$ (for some word $w$), we form the infinite word $y^\infty x w$. The centers of the classes obtained via the final infinite string of $y$-mutations increase to $\acc(b)$, forming an ascending infinite staircase. Similarly, for a $b$-value that is the right endpoint of some interval $J_{w\bE_{[n+1,n-2]}}$, by considering the infinite word $x^\infty y w,$ we obtain a sequence that determines a descending staircase. These staircases are called \textbf{principal staircases}. 

The other $b$-values in $Stair|_{G_n}$ correspond to words $w$ which do not contain a string of infinitely many consecutive $x$'s or $y$'s. As a result, the word $w$ must contain infinitely many $x$'s and infinitely many $y$'s. Therefore, applying this word to the triple $\Tt_n$ generates a sequence of classes where infinitely many of them are obtained from the previous one by an $x$-mutation and infinitely many by a $y$-mutation. This determines two subsequences of centers converging to $\acc(b)$ (one increasing and one decreasing), forming two staircases (one ascending and one descending), both accumulating at $\acc(b)$. These are called \textbf{non-principal staircases}.

It was explained in \cite[Remark~2.1.3~(i)]{MMW} that principal staircases have quadratic irrational accumulation points and $b$-values, and proved in \cite[Theorem~1.1.1~(iv)]{MMW} that non-principal staircases have irrational accumulation points. Our Theorem \ref{cor:rpt} proves that non-principal staircases also have irrational $b$-values.

\begin{remark} \label{rmk:Tnodd} \rm
    While for our case of the Hirzebruch surface $H_b$, the relevant blocking classes have center at $p/q=n$ even, the blocking classes that show up in \cite{usher} for the polydisk case are perfect classes with center at $p/q=n$ odd. 
    
    For the polydisk we represent a quasi-perfect class in $H_2(S^2 \times S^2 \# n\oCP^2,\Z)$ by the tuple $(e,f;p,q,t)$, where $e,f$ are the coefficients of the $[S^2]$ classes and $\bw(p/q)$ determines the rest of the coefficients. By a method analogous to \eqref{eq:dmfrompqt}, $e,f$ are determined by $p,q,t$ (and there is no $\varepsilon$), see \cite[Lemma~2.2.5]{SPUR}.
 These classes are represented by 
 $$\bE_{[n]}=(\tfrac12(n-1),1,n,1,n-3).$$

For $n$ odd, in verifying the triples $\Tt_n:=(\bE_{[n]},\bE_{[n+1,n-2]},\bE_{[n+2]})$ are generating triples, 
 checking conditions (a)-(d) when $n$ is odd is identical to the case when $n$ is even, but condition (e) on the relative positions of the $\acc(m/d)$s would need to be amended to use the accumulation point function for the polydisk. \hfill$\er$
\end{remark}

\subsection{The structure of $Block$ and $Stair$ for $b\in[0,(3-\sqrt{5})/2]$}\label{ss:symmetries}

Following \cite[Section 2.2]{MM} and \cite[Section 2.3]{MMW}, in this section we will see how the structure of $Block$ and $Stair$ for $b\in[0,(3-\sqrt{5})/2]$ is obtained from that of $b\in((3-\sqrt{5})/2,1)$ under iterated symmetries. Each of the blue and orange sets in Figure \ref{fig:block} is a copy of the red set (and the Cantor sets in its gaps) which was detailed in Section \ref{ss:fractal}. We begin by introducing the relevant symmetries.

The two basic symmetries acting on the $z$-coordinate are the \textbf{shift} $S$ (orientation-preserving) and the \textbf{reflection} $R$ (orientation-reversing):
\begin{align*}
S(z):=6-\frac{1}{z}\qquad\mbox{and}\qquad R(z):=6+\frac{1}{z-6}.
\end{align*}
These symmetries on the $z$-coordinate induce corresponding symmetries on the $b$-coordinate:
\begin{align*}
\underline{S}(b)&:=\begin{cases}
\acc^{-1}_{+1}\circ S\circ\acc(b)\qquad \mbox{ for } b\in[0,1/3] \\
\acc^{-1}_{-1}\circ S\circ\acc(b)\qquad \mbox{ for } b\in[1/3,1)
\end{cases}\\
\underline{R}(b)&:=\acc^{-1}_{-1}\circ R\circ\acc(b)\qquad\,\,\,\,\mbox{ for } b\in[b_{[6]}^+,1)
\end{align*}
where $\acc^{-1}_{\varepsilon}$ with $\varepsilon=\pm1$ are the two branches that give the inverse of $\acc$. Note that if $b\geq1/3$, then $\underline{S}(b)\leq 1/3$ and vice-versa. The same happens for $\underline{R}$. The value 
$$b_{[6]}^+:=\tfrac{3}{44}(7+\sqrt{5})$$
is the right endpoint of the first red interval in Figure \ref{fig:block}. On the $b$-coordinate, the symmetry $\underline S$ is orientation-reversing while $\underline R$ is orientation-preserving; both are injective.

We further consider \textbf{symmetries} that are compositions of the form $S^i$ and $S^iR$ (for $z$-values), and $\underline{S}^i$ and $\underline{S}^i\underline{R}$ (for $b$-values), for $i\geq0$. Next, we will describe how these symmetries act on $(b_{[6]}^-,1)$ to tile the whole $[0,1)$ interval minus an infinite set of special rational values of $b$ and $1/3$.
We have 
$$\underline S\left((b_{[6]}^-,1)\right)=\left(\tfrac15,\underline{S}(b_{[6]}^-)\right),$$ where $\underline{S}(b_{[6]}^-)$ is the right endpoint of the biggest orange interval in Figure \ref{fig:block}, and
$$\underline R\left([b_{[6]}^+,1)\right)=[0,\tfrac15).$$
The image of the red set of intervals in Figure \ref{fig:block} under a symmetry of the form $\underline S^i$ is always an orange set, whereas under $\underline S^i\underline R$ it is always a blue set. Finally, apart from $b=1/3$, the gap $[\underline{S}(b_{[6]}^-),b_{[6]}^-]$ is tiled as depicted in Figure~\ref{fig:block} by 
$$\underline{S}^i\left([0,\underline{S}(b_{[6]}^-))\right), \qquad i\geq1.$$
The common limit point of the $i^\text{th}$ copy of the blue and orange intervals is exactly the
special rational $b$-value $\underline{S}^{i-1}(\tfrac15)$. 
On the other side, the $i^\text{th}$ copy of the orange set connects to the $i+2^\text{nd}$ copy of the blue set in the following way:
$$\underline S^i(b_{[6]}^-)=\underline S^{i+1} (0)=
\underline S^{i+1} \underline R(b_{[6]}^+).$$
The second blue set connects to the first red interval:
$$b_{[6]}^-=\underline S(0)=\underline S \underline R (b_{[6]}^+).$$
The limit point of the nested circles in Figure \ref{fig:block} is exactly $b=1/3$, which we ultimately prove in Theorem \ref{thm:13proved} is the only nonzero rational $b$-value for which $c_b$ has an infinite staircase.

These symmetries transport the structure of $Block$ and $Stair$ described in Section \ref{ss:fractal} for $b\in(b_{[6]}^-,1)$ to $[0,b_{[6]}^-]$, thus completing the picture for the whole parameter space $[0,1)$. This happens because the $z$-symmetries $S^i$ and $S^iR$ induce actions $(S^i)^\sharp$ and $(S^iR)^\sharp$ on quasi-perfect classes, and further on generating triples, thus taking blocking classes to blocking classes and staircase steps to staircase steps. In particular, if $c_b$ has an ascending (respectively descending) infinite staircase then each $c_{\underline S^i(b)}$ has an ascending (respectively descending) infinite staircase as well, whereas each $c_{\underline S^i\underline R(b)}$ has a descending (respectively ascending) infinite staircase. The image of a principal staircase under these symmetries is called a principal staircase, and similarly for non-principal staircases.

More specifically, $S$ and $R$ act linearly on $(p,q)$ for $z=p/q$ and the value of $t$ defined in \eqref{eq:deft} is fixed by the symmetries, which by \cite{MM} and \cite[Section 2.3]{MMW} implies that there is a well-defined action on quasi-perfect classes $\bE_{CF(p/q)}$. By \cite[Proposition~2.3.2]{MMW}, the action $(S^i)^\sharp$ (respectively, $(S^iR)^\sharp$) sends generating triples to generating triples (respectively, sends generating triples to triples which, when the first and last entry are swapped, are generating triples).
By \cite[Corollary~2.3.4]{MMW}, we know that $(S^i)^\sharp$ commutes with $x$ and $y$-mutations whereas $(S^iR)^\sharp$ switches them. By \cite[Theorem~1.1.1~(i)]{MMW}, the symmetries commute with taking the $b$-blocked and $z$-blocked intervals in the sense that the endpoints of the blocked $b$-interval $J_{(S^i)^\sharp\bE}$ are the images of the endpoints of $J_\bE$ under $\underline{S}^i$, and similarly for $(S^iR)^\sharp$ and also for the $z$-intervals $I_\bE$. By \cite[Lemma~4.1.2, 4.1.3]{MM}, if $\bE$ is perfect, then its image under any symmetry is also perfect.

\begin{remark} \rm
    Note that the special rational $z$-values $y_{i+2}/y_{i+1}$ defined in \eqref{eqn:yk} are the centers of the blocking classes $(S^i)^\sharp(\bE_{[6]})$. One of their preimages under $\acc$ is the special rational $b$-value $b_i$ while the other is blocked. \hfill$\er$
\end{remark}

\section{$b$-values for non-principal staircases are irrational}\label{s:blocked}

 The main goal of this Section is to prove Theorem \ref{cor:rpt}, which states that the $b$-values for which $c_b$ has a non-principal staircase are irrational. We do this by showing that the accumulation point for a non-principal staircase, while being irrational, is not a quadratic irrationality, that is, does not have a periodic continued fraction. By the accumulation point formula \eqref{eq:accb} this will then imply that $b$ is irrational. As explained in Section \ref{ss:fractal}, the accumulation point of one of these staircases is the limit \eqref{eq:centerLimit} of the sequence of centers of the classes $s_k\bE$, where $s_k$ are the successive truncations of an infinite word $w$ on $x$ and $y$ encoding the mutations on the triple $\Tt$ with middle class $\bE$.
For this reason, we are going to study how applying mutations to classes affects the continued fraction of their centers. In doing so, we also prove \cite[Conjecture~2.2.4]{MMW}.

\subsection{Obtaining the continued fractions of the centers of perfect blocking classes}\label{ss:numerics}

The aim of this section is to recursively generate the continued fractions of the centers of all perfect blocking classes. As stated in \cite[Corollary~4.2.3]{MMW}, all perfect blocking classes are obtained from the triples $\Tt_n$ via a sequence of mutations and symmetries. Theorem \ref{thm:CFmain} explains how performing one mutation changes the continued fraction of the center of the class. Figure \ref{fig:Fareystairs} shows the continued fractions of the centers of the first few classes obtained by performing mutations on the triple $\Tt_{6}=(\bE_{[6]},\bE_{[7,4]},\bE_{[8]})$, where we can already see some patterns emerging. One such pattern is that the length of the continued fraction of the center of the middle class of a triple is equal to the sum of the lengths of those of the left and right classes. This and some other patterns were observed and conjectured in \cite[Conjecture~2.2.4]{MMW} and are proved below as Corollary \ref{cor:MMWconj}. The example below hints at the pattern proved in Theorem \ref{thm:CFmain}.

\begin{figure}[h]
    \centering
\includegraphics[width=0.5\textwidth]{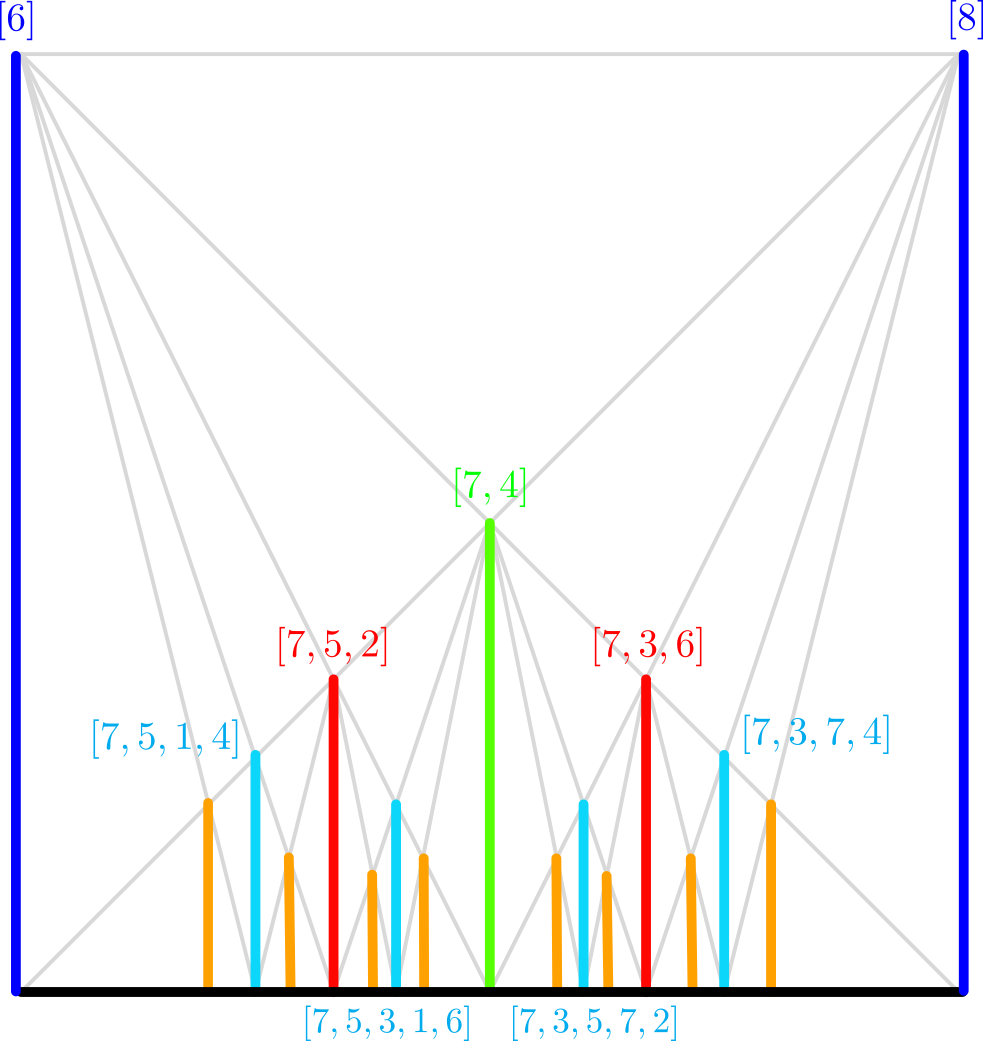}
    \caption{\cite[Figure~4.1]{MMW} This diagram represents blocking classes with centers between $6$ and $8$. It is structured like a Farey diagram to show how these classes arise via successive mutations from the triple $\Tt_{6}$. Vertical lines of the same color represent classes that arise as the middle class after the same number of mutations (except for the classes $[6]$ and $[8]$). 
    Given a triple $(\bE_\lambda,\bE_\mu,\bE_\rho)$, the height of $\bE_\mu$ is given by the intersection of the gray lines from the top of $\bE_\lambda$ to the bottom of $\bE_\rho$ and from the top of $\bE_\rho$ to the bottom of $\bE_\lambda$. In this way, each triangle with three gray sides which is not subdivided into smaller triangles represents a triple. A consequence of Corollary \ref{cor:MMWconj} will be that the height of the line corresponding to a class $\bE$ equals the reciprocal of the length of the continued fraction of its center (assuming the square has height 1).}
    \label{fig:Fareystairs}
\end{figure}

\begin{example} \rm 
When performing an $x$-mutation on the triple
\[
yxx\Tt_6=\left(\bE_{[\underbracket[0.140ex]{\scriptstyle{7,5,1,4}}]},\bE_{[7,5,1,3,5,1,6]},\bE_{[7,\overbracket[0.140ex]{\scriptstyle{5,2}}]}\right),
\]
the middle class of the new triple $xyxx\Tt_6$ has center 
    \[
\frac{t_{[7,5,1,4]}p_{[7,5,1,3,5,1,6]}-p_{[7,5,2]}}{t_{[7,5,1,4]}q_{[7,5,1,3,5,1,6]}-q_{[7,5,2]}}=\frac{89\cdot7066-79}{89\cdot985-11}=\frac{628,795}{87,654}=[\underbracket[0.140ex]{7,5,1,3},\underbracket[0.140ex]{5,1,5,7}_{\leftarrow},1,\overbracket[0.140ex]{5,2}].
    \]
\hfill$\er$
\end{example}

\begin{remark}\label{rmk:nn-1} \rm
In the statement and proof of Theorem \ref{thm:CFmain}, if we are considering a triple $(\bE_\la,\bE_\mu,\bE_\rho)$ where the class $\bE_\rho$ is centered at an integer $n+2$, we assume we are working with its two-digit continued fraction expansion $[n+1,1]$.
Furthermore, we make the simplification
\[
[a_0,\dots,a_{i-2},a_{i-1},0,a_{i+1},a_{i+2},\dots,a_m]=[a_0,\dots,a_{i-2},a_{i-1}+a_{i+1},a_{i+2},\dots,a_m].
\]
\hfill$\er$
\end{remark}

\begin{thm} \label{thm:CFmain}  
Let $n\geq6$ and $\left(\bE_\la,\bE_\mu,\bE_\rho\right)$ be a triple obtained from $\Tt_n$ via a sequence of $x$ and $y$ mutations. If the continued fraction expansions of the centers of the left and right classes are given by
\[
\frac{p_\la}{q_\la}=[l_0,\dots,l_{s}] \mbox{ and } \frac{p_\rho}{q_\rho}=[r_0,\dots,r_m],
\]
then those of the center of the middle classes of the triples $x\Tt$ and $y\Tt$ are given by the following formulas:
\begin{equation}\label{eqn:ymuCF}
\frac{p_{y\mu}}{q_{y\mu}}=[r_0\hdots,r_{m-1},r_m+(-1)^{m+1},r_m+(-1)^m,r_{m-1},\hdots,r_1,n-5,l_0,\hdots,l_{s}],
\end{equation}
and
\begin{equation}\label{eqn:xmuCF}
\frac{p_{x\mu}}{q_{x\mu}}=[l_0,\hdots,l_{s-1},l_{s}+(-1)^s,l_{s}+(-1)^{s+1},l_{s-1},\hdots,l_0,n-5,r_1,\hdots,r_m].
\end{equation}
Note that the only case when $s=0$ is that of $\bE_\lambda=[n]$, in which case we use the indexing $[n]=[l_s]$.
\end{thm}

The proof of this theorem is given in Appendix \ref{app:pf}, but the key tool is writing the action of the mutation on the centers of the classes in the triple in matrix form: in particular, $M_\lambda \begin{pmatrix}p_\rho\\ q_\rho\end{pmatrix}=\begin{pmatrix}p_{x\mu}\\ q_{x\mu}\end{pmatrix}$ and $M_\rho \begin{pmatrix}p_\la \\ q_\la\end{pmatrix}=\begin{pmatrix}p_{y\mu}\\ q_{y\mu}\end{pmatrix}$, where
\[
M_\la:= \begin{pmatrix} p_\la^2 & p_\la q_\la-6p_\la^2+1 \\ p_\la q_\la-1 & q_\la^2-6p_\la q_\la +6 \end{pmatrix} \quad \text{and} \quad M_\rho:=\begin{pmatrix} p_{\rho}^2-6p_\rho q_\rho+6 & p_\rho q_\rho-1 \\ p_\rho q_\rho-6 q_\rho^2+1 & q_\rho^2 \end{pmatrix}.
\]
Recall that an $x$-mutation leaves the left class $\bE_\la$ of the triple unchanged, hence performing repeated $x$-mutations will correspond to applying powers of the matrix $M_\la$ to the centers of $\bE_\mu$ and $\bE_\rho$. Similarly, powers of $M_\rho$ give the $y$-mutations.

We begin by proving a quick consequence of Theorem \ref{thm:CFmain}, which we will use to prove Conjecture 2.2.4 in \cite{MMW}. Corollary \ref{cor:a_0=n+1} also follows from the fact that $\bE_{[n]}$ blocks a staircase accumulating to $n+1$: see the last formula in the statement of \cite[Thm.~1]{ICERM}.

\begin{corollary}\label{cor:a_0=n+1}
    If $\bE$ is the middle class of a triple obtained from $\Tt_n:=\left(\bE_{[n]},\bE_{[n+1,n-2]},\bE_{[n+2]}\right)$ via a sequence of mutations, then the first entry of the continued fraction of its center is $n+1$.
\end{corollary}

\begin{proof} 
This corollary was proved as Claim A in the proof of Theorem~\ref{thm:CFmain}, but we chose to highlight it separately. 
\end{proof}

Conjecture 2.2.4 in \cite{MMW} is about the length of the continued fractions of the centers of the classes obtained via mutations from $\Tt_6$ (but holds further for all $\Tt_n$ with $n\geq 6$), as suggested by Figure \ref{fig:Fareystairs}.\footnote{\cite[Conjecture~2.2.4]{MMW} is stated in terms of ``$\mathcal{CS}$-length,'' whose defining property is the recursion proved in Corollary \ref{cor:MMWconj}. Thus we do not need to use the term ``$\mathcal{CS}$-length'' here.} The \textbf{CF-length $\ell_{CF}$} of a class is the number of entries in the continued fraction of its center:
    \[
\ell_{CF}(\bE_{[a_0,\dots,a_m]})=m+1.
    \]
Note that when computing $\ell_{CF}$ we use the shorter possible continued fraction, i.e. we assume $a_m\neq1$.

As proven in \cite{ICERM}, the steps of the staircase accumulating to the left and right endpoints of $J_{\bE_{[n]}}$ are given by $y^k\bE_{[n-1,n-4]}$ and $x^k\bE_{[n+1,n-2]}$, respectively, and have continued fractions of the form
\[ y^k\bE_{[n-1,n-4]}=[\{n-1,n-5\}^{\lceil k/2 \rceil},end_k] \quad \text{and} \quad x^k\bE_{[n+1,n-2]}=[n+1,\{n-1,n-5\}^{\lfloor k/2 \rfloor},end_k]\] 
 where $end_k$ depends on the parity of $k$. We note that the periodic part of each of these continued fractions has length $2=2\ell_{CF}(\bE_{[n]}).$ Below, we show a similar result holds for $J_{\bE}$ where $\bE$ is obtained by mutation from $\bE_{[n]}.$

\begin{corollary}[{\cite[Conjecture~2.2.4]{MMW}}]\label{cor:MMWconj}
Given a triple $(\bE_\lambda,\bE_\mu,\bE_\rho)$ obtained via mutations from $\Tt_n$ with $n \geq 6,$ the following hold:
    \begin{itemize}
        \item[{\rm (i)}] We have
        \[\ell_{CF}\left(\bE_\mu\right)=\ell_{CF}\left(\bE_\la\right)+\ell_{CF}\left(\bE_\rho\right).\]

          \item[{\rm (ii)}] The steps of the staircase accumulating at the left endpoint of $J_{\bE_\mu}$ given by the centers of $y^kx\bE_\mu$ have continued fractions of the form
          $$[\text{start}_k,\{\text{period}\}^{\lceil{\frac k2}\rceil},\text{end}_k],$$
          where the ``period'' part has length equal to $2\ell_{CF}\left(\bE_\mu\right)$ and ``$\text{start}_k$'' and ``$\text{end}_k$'' depend only on the parity of $k$. A similar result holds for the steps of the staircase at the right endpoint of $J_{\bE_\mu}$, with centers $x^ky\bE_\mu$, where if $\bE_\rho=\bE_{[n+2]}$, then the exponent on $\{period\}$ is $\lfloor k/2 \rfloor.$
          
          \item[{\rm (iii)}]Moreover, the ``periods'' corresponding to the staircases at the left and right endpoints of $J_{\bE_\mu}$ have the same digits but in reverse (cyclic) order.
    \end{itemize}
\end{corollary}

\begin{proof}
   
   For (i), the case for $\Tt_n=(\bE_{[n]},\bE_{[n+1,n-2]},\bE_{[n+2]})$ is clear. Hence, it suffices to show that if (i) holds for $\Tt,$ then it also holds for $x\Tt$ and $y\Tt.$  The statement holds for $y\Tt=(\bE_\mu,\bE_{y\mu},\bE_\rho)$ as 
   \[\ell_{CF}(\bE_{y\mu})=2\ell_{CF}(\bE_\rho)+\ell_{CF}(\bE_\la)=\ell_{CF}(\bE_\rho)+\ell_{CF}(\bE_\mu)\] where the first equality holds by \eqref{eqn:ymuCF} and the second equality holds by assumption that (i) holds for $\Tt.$ The proof for the $x$-mutation is analogous using \eqref{eqn:xmuCF}.
  
    For the proof of (ii), let $\bE_\mu=[a_0,\hdots,a_m].$
    We now show that the conclusion for the left endpoint of $J_{\bE_\mu}$ in (ii) follows from \eqref{eqn:ymuCF}. First, consider the case $k=1,$ where \eqref{eqn:ymuCF} implies that the continued fraction of $p_{yx\mu}/q_{yx\mu}$ is determined from the triple $(\bE_\la,\bE_{x\mu},\bE_\mu)$, so differs from that of $p_\la/q_\la$ by appending the $2\ell_{CF}\left(\bE_\mu\right)$ digits $a_0,\dots,a_{m-1},a_m+(-1)^{m+1},a_m+(-1)^m,a_{m-1},\dots,a_1,n-5$ to the beginning. The continued fraction of $p_{y^kx\mu}/q_{y^kx\mu}$ will be determined from the triple $(\bE_{y^{k-2}x\mu},\bE_{y^{k-1}x\mu},\bE_\mu)$. Thus, $p_{y^kx\mu}/q_{y^kx\mu}$ differs from that of $p_{y^{k-2}x\mu}/q_{y^{k-2}x\mu}$ 
    by appending the same $2\ell_{CF}\left(\bE_\mu\right)$ digits to the beginning coming from $\bE_\mu.$

    The conclusion for the right endpoint is similar following from \eqref{eqn:xmuCF} as here the continued fraction of $p_{x^ky\mu}/q_{x^ky\mu}$ differs from that of $p_{x^{k-2}y\mu}/q_{x^{k-2}y\mu}$ by inserting the $2\ell_{CF}\left(\bE_\mu\right)$ digits 
    $$a_1,\dots,a_{m-1},a_{m}+(-1)^m,a_{m}+(-1)^{m+1},a_{m-1},\dots,a_0,n-5$$
    after the initial $n+1$ (see Corollary \ref{cor:a_0=n+1}). Note, if $\bE_\rho = \bE_{[n+2]},$ the continued fraction of $p_{xy\mu}/q_{xy\mu}$ and $\bE_\rho$ do not differ by these digits, and as such we have the exponent $\lfloor k/2 \rfloor$ on $\{period\}$.

    To prove (iii), and using the notation from the proof of (ii) above, we need to show that (up to cyclic reordering),
    \begin{multline*}
        (a_0,\dots,a_{m-1},a_m+(-1)^{m+1},a_m+(-1)^m,a_{m-1},\dots,a_1,n-5)\\=(n-5,a_0,\dots,a_{m-1},a_m+(-1)^{m+1},a_m+(-1)^m,a_{m-1},\dots,a_1),
    \end{multline*}
    where these are the two periodic parts we saw appearing above, with the second one having its order reversed. This is immediate; simply move the $n-5$.
\end{proof}

Theorem \ref{thm:CFmain} applies to triples who have $\bE_\la$ with centers $z\geq 6$. To understand the continued fractions of the centers of the rest of the blocking classes, we recall from \cite[Lemma~2.1.5]{MM} the effect of the symmetries $S$ and $R$ on the continued fraction of $z$. When $z=[5+k,CF(\tfrac{1}{z-5-k})]$ for some $k\geq 0$, we have
\begin{equation} \label{eq:SCF}
    S(z)=\left[5,1,4+k,CF\left(\frac{1}{z-5-k}\right)\right]. \end{equation}
When $z=[6+k,CF(\frac{1}{z-6-k})]$ for some $k\geq 1$, we have
\begin{equation} \label{eq:RCF}
    R(z)=\left[6,k,CF\left(\frac{1}{z-6-k}\right)\right]. \end{equation}

Finally, we note that Theorem \ref{thm:CFmain} allows us to prove Corollary \ref{cor:+n-6}, which gives an explicit relationship between the relevant exceptional classes for the polydisk and the Hirzebruch surface $H_b$. 

\begin{corollary}\label{cor:+n-6}
    If $w$ is a finite word of $x$- and $y$- mutations, then the continued fractions of the centers of $w\Tt_n$ are given by adding $n-6$ entry by entry to the centers of $w\Tt_{[6]}.$
\end{corollary}
\begin{proof}
    As $\Tt_{[n]}$ has centers $([n],[n+1,n-3],[n+2]),$ the statement is immediate for the empty word. For all other words,  the statement follows by Theorem~\ref{thm:CFmain}.
\end{proof}

\begin{remark}\label{rmk:polydisk} \rm 
As seen in Section~\ref{ss:fractal}, the blocking classes for the Hirzebruch surface $H_b$ are given by $w\Tt_n$ and its symmetries for $n$ even and $w$ a finite word of $x$- and $y$-mutations. By Remark \ref{rmk:Tnodd}, the blocking classes and generating triples relevant from the work of \cite{usher} are $\Tt_n$ and its shifts for $n$ odd. The paper \cite{SPUR} showed that $y\Tt_7$ has an ascending staircase associated to it, and the results of this paper we expect to be generalizable to $w\Tt_n$ for $n$ odd and $w$ a finite word of $x$- and $y$-mutations. 

As such, for $n$ odd we expect $w\Tt_n$ and their symmetries to be the blocking classes for the polydisk. Theroem~\ref{thm:CFmain} applies to $n$ of either parity. In particular, Corollary \ref{cor:+n-6} and Remark \ref{rmk:Tnodd} imply that if \[[a_0,\dots,a_m]\] is a perfect class for the $H_b$ with center $>7$ then \[[a_0+1,\dots,a_m+1]\] is a quasi-perfect class for the polydisk with center $>8$. Note that we have not yet proved that they are perfect, nor that they are all the perfect classes for the polydisk with center $>8$.

We now explain this observation in terms of where we expect staircases to occur for the polydisk. Let $\acc_P$ be the 1-1 accumulation point function for the polydisk defined in \cite[(1.14)]{AADT}, and $CF^+$ be the operation which sends a number with continued fraction $[a_0,\dots,a_n]$ to $[a_0+1,\dots,a_n+1]$. Then, when $b\geq(3-\sqrt{5})/2$, given a staircase in $c_b$, we expect there to be an infinite staircase for $P(1,\beta)$, where
\begin{equation}\label{eqn:CF-}
\beta=\acc_P^{-1}\circ CF^+\circ\acc.
\end{equation}
 (When $b<(3-\sqrt{5})/2$ the relationship is more complicated due to the fact that the operation $CF^+$ and the symmetries do not commute.) Note, that Corollary~\ref{cor:MMWconj}~(i) and all of the results in Section~\ref{ss:afternumerics} are stated for $n$ of either parity besides the proof of Theorem~\ref{cor:rpt}. 
 \hfill$\er$
\end{remark}

\subsection{Proof of Theorem \ref{cor:rpt}} \label{ss:afternumerics}

Here, we complete the proof of Theorem \ref{cor:rpt}, which states non-principal staircases have irrational $b$-values. To do this, we want to consider the set of all classes that are obtained from the ones in the generating triples $\Tt_n$ via any sequence of mutations and symmetries: let $\Cc:=\Cc_{odd} \sqcup \Cc_{even}$, where 
$$\Cc_{odd}:=\left\{ (S^iR^\delta)^\sharp(\bE) \ \middle\vert \begin{array}{l} \text{$i \geq 0$ and $\de=0$ or $1$,}\\ \text{$\bE=\bE_{[n]}$ or $\bE=w\bE_{[n+1,n-2]}$ with $n \geq 7$ odd, $w$ word on $x,y$}\end{array}\right\}.$$
$$\Cc_{even}:=\left\{ (S^iR^\delta)^\sharp(\bE) \ \middle\vert \begin{array}{l} \text{$i \geq 0$ and $\de=0$ or $1$,}\\ \text{$\bE=\bE_{[n]}$ or $\bE=w\bE_{[n+1,n-2]}$ with $n \geq 6$ even, $w$ word on $x,y$}\end{array}\right\}.$$

Even though the proof of Theorem \ref{cor:rpt} only involves classes in $\Cc_{even}$, many of the auxiliary results used in that proof hold also for classes in $\Cc_{odd}$. As discussed in Remarks \ref{rmk:Tnodd} and \ref{rmk:polydisk}, those classes are the ones governing the polydisk case, and for this reason, we include here results and proofs pertaining to them.

\begin{lemma} \label{lem:genclasses}
\begin{itemlist} Assume $\bE \in \Cc$ with $i=0$ and $\delta=0$ has center $[a_0,\hdots,a_m]$ where $a_m \neq 1$. Then
	\item[{\rm (i)}] $a_j \neq a_{j+1}$ for all $j$. 
	\item[{\rm (ii)}] The parity of $a_m$ differs from the parity of $a_j$ for all $ j < m$.
\end{itemlist} 
\end{lemma} 

\begin{proof}
We begin by proving two auxiliary results that hold for any arbitrary class $\bE\in\Cc$ with either $\bE=\bE_{[n]}$ or $\bE=w\bE_{[n+1,n-2]}$:

\vspace{0.2cm}

\noindent \textbf{Claim 1:} For the class $\bE$ with center $[a_0,\hdots,a_m]$, we have $a_1=n-1,n-2$ or $n-3$ (except in the case where $\bE=\bE_{[n]}$ and then $a_1=0$).
\begin{proof}(of Claim 1)
The claim holds if the center of $\bE$ is $[n]$ or $[n+1,n-2].$ Otherwise, we may assume the class $\bE$ is $x\bE_{\mu}$ (resp. $y\bE_{\mu}$) where $\bE_\mu$ is the middle entry of the triple $(\bE_\la,\bE_\mu,\bE_\rho)$. Then, by Theorem~\ref{thm:CFmain}, the continued fraction of $\bE_\la$ (resp. $\bE_\rho$) will determine $a_1.$ 
If $\bE_{\la}$ (resp. $\bE_\rho$) has CF-length less than or equal to two, then $\bE_{\la}$ (resp. $\bE_\rho$) must have center $[n],[n+1,n-2]$ (resp. $[n+2],[n+1,n-2]$). In these cases, following the process described in Theorem~\ref{thm:CFmain}, we find that $a_1$ must be either $n-1,n-3$, or $n-2$ only if $\bE=\bE_{[n+1,n-2]}$. See the proof of Corollary \ref{cor:a_0=n+1} for these computations. If $\bE_\la$ (resp. $\bE_\rho)$ has CF-length larger than two, than the second entry of $\bE$ will have the same second entry of $\bE_\la$ or $\bE_\rho$, so they must be $n-1,n-3$ as well.
\end{proof}

\noindent \textbf{Claim 2:} For the class $\bE$ with center $[a_0,\hdots,a_m]$, we have $a_{m-1} \neq a_m \pm 1$. 

\begin{proof}(of Claim 2)
The claim holds if the center of $\bE$ is $[n]$ or $[n+1,n-2]$. Otherwise, the class $\bE$ is $x\bE_{\mu}$ (resp. $y\bE_{\mu}$) where $\bE_\mu$ is the middle entry of the triple $(\bE_\la,\bE_\mu,\bE_\rho)$. The final digits of $\bE$ will be determined by $\bE_\rho$ (resp. $\bE_\la$). As above, if $\bE_\rho$ (resp. $\bE_\la$) has CF-length at least three (resp. two), the claim also follows for $\bE$. In the case where the CF-length is less than three (resp. two) we must check that $r_1 \neq n-5 \pm 1$ (resp. $\ell_0 \neq n-5 \pm 1$). These follow from Claim 1 and Corollary~\ref{cor:a_0=n+1}. 
\end{proof}

We now prove (i) by complete induction on the length of the word $w$. 
In the base case of length zero, the class $\bE$ has center $[n]$ or $[n+1,n-2]$ and so the claim holds. Next we assume that the claim holds for all classes where $w$ has length less than $k$ and aim to prove it for a class $\bE=w\bE_{[n+1,n-2]}$ with center $[a_0,\ldots,a_m]$ where $w$ has length $k$. 
By Theorem~\ref{thm:CFmain} and using the induction hypothesis, we can only have $a_i=a_{i+1}$ if there exists also a class obtained by performing fewer than $k$ mutations which has center $[b_0,\hdots,b_r]$ and for which $b_0=n-5$, $b_1=n-5$, or $b_{r-1}=b_r\pm 1$. 
But by Corollary~\ref{cor:a_0=n+1}, Claim 1, and Claim 2, no class can have these properties.

For (ii), we again use induction on the length of the word $w$. The base case where $w$ is the empty word is again clear. Next, by Theorem~\ref{thm:CFmain} and assuming $\bE_\rho$ does not have center $[n+2],$ we have that the last digit of any $x$ or $y$ mutation is the last digit of a previous class, so the parity of the last digit will stay the same. All other digits are either a previous non-last digit or the last digit plus or minus one, and hence their parity will always differ from that of the last digit. The only other term is $n-5,$ which has the same parity as $n+1,$ the first digit by Corollary~\ref{cor:a_0=n+1}. If $\bE_\rho$ has center $[n+2]$, then in Theorem~\ref{thm:CFmain} we write $[n+2]=[n+1,1].$ Say we have a triple $([b_0,\hdots,b_{r}],\bE,[n+1,1])$. By induction, we assume the statement holds for $[b_0,\hdots,b_r]$. Performing a $y$-mutation to the triple, the new class has center
\[ [n+1,2,0,n-5,b_0,\hdots,b_r]=[n+1,n-3,b_0,\hdots,b_r],\] and the statement holds. If we perform a $x$-mutation, the new class has center 
\[ [b_0,\hdots,b_0,n-5,1]=[b_0,\hdots,b_0,n-4],\] and the statement holds. 

\end{proof}

The first part of the following lemma is adapted from \cite[Proposition~4.1.3]{MMW} about perfect classes.

\begin{lemma} \label{lem:adjEq}
\begin{itemlist}
    \item[{\rm (i)}] Let $\bE$ and $\bE'$ be two perfect classes, with centers $[a_0,\hdots,a_m]$ and $[a'_0,\hdots,a'_{m'}]$ respectively. If $\bE$ and $\bE'$ are adjacent, then $a_j=a'_j$ for $0 \leq j \leq \min(m,m')-1.$
 \item[{\rm (ii)}] Assume $\bE$ and all mutations of $\bE$ are perfect. If $\bE$ has center $[a_0,\hdots,a_{m-1},a_m]$, then for every word $w$ the center of $w\bE$ is of the form $[a_0,\hdots,a_{m-1},a'_m,\hdots]$. 
 \item[{\rm (iii)}] If $\bE\in\Cc$ has center $[a_0,\hdots,a_m]$, then for every word $w$ the center of 
$w\bE$ is of the form $[a_0,\hdots,a_{m-1},a_m',\hdots]$.
 \end{itemlist}
\end{lemma}
\begin{proof}

(i) is a restatement of  \cite[Lemma~4.1.1~(iii)]{MMW}  which uses \cite[Proposition~4.1.3]{MMW}. 

Now (ii) is obtained by inductively applying (i), since performing an $x$ or $y$-mutation to a class yields a class adjacent to it.

For (iii), note that by \cite[Corollary~7.3]{M} and \cite[Lemma~3.1.4]{MMW}, every class in $\Cc_{even}$ is perfect, and so in particular the claims holds for these classes by (ii). For classes $\bE \in \Cc_{odd},$ by Remark \ref{rmk:polydisk}, the continued fraction of the center of $w\bE_{[n+1,n-2]} \in \Cc_{odd}$ is equal to the continued fraction of the center of $w\bE_{[n,n-3]}] \in \Cc_{even}$ with $1$ added to each entry of the continued fraction. Thus, (iii) also holds for classes in $\Cc_{odd}.$ 
\end{proof} 

As observed in Corollary \ref{cor:MMWconj}, when we perform (enough) repeated $x$-mutations to a class $\bE$, the continued fraction of the center of the resulting class will be of the form 
$$[start,\{period\}^s,end],$$
for some fixed set of digits $period$.
The next result uses Lemmas \ref{lem:genclasses} and \ref{lem:adjEq} to show that once we perform one $y$-mutation (followed by any number of $x$ or $y$-mutations), the resulting classes cannot pick up more than two repetitions of $period$.

\begin{prop} \label{lem:repeat}
    Let $A$ denote some fixed finite list of numbers. Assume $\bE\in \Cc$ has center 
    $$[A,\{a_{1},\hdots,a_\ell\}^{k},c_1,\hdots,c_m],$$
    with $k$  maximal in the sense that the center of $\bE$ cannot be written as $[A,\{a_{1},\hdots,a_\ell\}^{k+1},c_{\ell+1},\hdots,c_m]$. 

  Then, for all words $w$ and for all $s \geq 2$, the centers of the classes $wyx\bE$ and $wxy\bE$ cannot have centers of the form
	\[ [A,\{a_{1},\hdots,a_\ell\}^{k+s},end] \] where $end$ is any arbitrary finite list of numbers. 
\end{prop}
\begin{proof} 
For brevity, in this proof, we will often denote a class simply by the continued fraction of its center: $\bE_{CF[z]}=CF(z)$.

We begin assuming $\bE=[A,\{a_{1},\hdots,a_\ell\}^k,c_1,\hdots,c_m]$ is constructed by mutation from a class centered at $[n+1,n-2]$. We will first show that for all words $w$ and $s\geq 2$, \[wyx\bE \neq  [A,\{a_{1},\hdots,a_\ell\}^{k+s},end]. \]

First, note that the statement holds if $m=0$ because by Lemma~\ref{lem:genclasses}~(ii), the last digit of $\bE$ has a different parity than all other digits of $\bE$, so the last digit cannot appear in a part of the continued fraction repeated at least two times. Thus, we can now assume that $m \geq 1$.

Given $A=[A_0,A_1,\hdots,A_p]$, we define $\overline{A}=[A_p,A_{p-1},\hdots,A_1]$.
By Theorem~\ref{thm:CFmain}, we have 
\[ yx\bE=[A,\{a_1,\hdots,a_\ell\}^k,c_1,\hdots,c_{m-1},c_m \pm 1,c_m \mp 1, c_{m-1},\hdots,c_1,\{a_\ell,\hdots,a_1\}^k,\overline{A},n-5,\hdots] .\]
For the sake of contradiction, assume that
$ yx\bE=[A,\{a_{1},\hdots,a_\ell\}^{k+2},end]$, that is 
\[ [A,\{a_{1},\hdots,a_\ell\}^{k},c_1,\hdots,c_{m-1},c_m \pm 1,c_m \mp 1, c_{m-1},\hdots,c_1,\{a_\ell,\hdots,a_1\}^k,...] = 
[A,\{a_{j},\hdots,a_\ell\}^{k+2},end],\]
which is equivalent to
\begin{equation} \label{eq:equalCF}
    [c_1,\hdots,c_{m-1},c_m \pm 1, c_m \mp 1, c_{m-1},\hdots,c_1,\{a_\ell,\hdots,a_1\}^k,\hdots]=[\{a_1,\hdots,a_\ell\}^2,end]
\end{equation} 
as the first $|A|+k\ell$ digits are the same. 

Note, if $ \ell \leq m-1,$ then \eqref{eq:equalCF} implies that 
\[ \{a_1,\hdots,a_\ell\}=\{c_1,\hdots,c_\ell\},\]
which contradicts the assumption $k$ is maximal. By Lemma~\ref{lem:adjEq}~(iii), the first digits of $wyx\bE$ must agree will all but the last digit of $yx\bE$, so the contradiction persists. Thus either $m=\ell$ or $m<\ell$. 

\noindent\underline{Case $m=\ell$:} For \eqref{eq:equalCF} to hold, the entries of each column of the following table must be equal: 
\[ \begin{array}{|c|c|c|c|c|c|c|c|}\hline
a_1&a_2&\hdots&a_i&a_{i+1}&\hdots&a_{\ell-1}&a_\ell \\ \hline
c_1 & c_2&\hdots&c_i & c_{i+1} &\hdots & \hdots & c_\ell \pm 1 \\ \hline
c_\ell \mp 1& c_{\ell-1} & \hdots & \hdots & \hdots & \hdots & \hdots & c_1 \\ \hline 
\end{array}\]
Considering the first and last column of the table, this would require $c_1=c_\ell \pm 1=c_\ell \mp1$, which is not possible. This implies that 
\[ yx\bE \neq [A,\{a_{1},\hdots,a_\ell\}^{k+s},end]\] for $s \geq 2$. 
More generally, we now show $wyx\bE$  is also not equal to $ [A,\{a_{1},\hdots,a_\ell\}^{k+s},end]$ for $s \geq 2$. Note, by Lemma~\ref{lem:genclasses}~(i), $a_i \neq a_{i+1}$, so we cannot have a period of $1$, i.e. $\ell>1$. Therefore, in considering the entries of $yx\bE$ in the table, we have at least two columns in the table and one entry after $c_\ell \mp 1$ in the last row. Then, by Lemma~\ref{lem:adjEq}~(iii), the first digits of $wyx\bE$ must agree with all but the last digit of $yx\bE$, so as the problem of $c_\ell \pm 1=c_\ell \mp 1$ does not occur in the last digit, this will persist and we conclude that if $m=\ell$, then for $s\geq 2$, we have
 \[ wyx\bE \neq [A,\{a_{1},\hdots,a_\ell\}^{k+s},end].\]

\noindent\underline{Case $m<\ell$:} For the first $m+1$ digits on each side of \eqref{eq:equalCF} to be equal we must have: 
\[ [a_1,\hdots,a_{m-1},a_m,a_{m+1}]=[c_1,\hdots,c_{m-1},c_m \pm 1,c_m \mp 1].\]
Then, after the $c_m \mp 1$ on the right hand side, we will then reflect the digits $\{c_{m-1},\hdots,c_1\}$ for the next $m-1$ digits. Then, we will start reflecting the $\{a_1,\hdots,a_\ell\},$ so we will have that $a_{m+1}$ and $a_m$ will appear again $m-1+\ell-m=\ell-1$ and $\ell$ places after $c_m \mp 1$. Therefore, 
the entries of each column of the following table must be equal: 

\[ \begin{array}{|c|c|c|c|c|c|c|c|}\hline
a_1&a_2&\hdots&a_m&a_{m+1}&\hdots&a_{\ell-1}&a_\ell \\ \hline
c_1 & c_2&\hdots&c_m \pm 1 & c_m \mp 1 &\hdots & \hdots & \hdots \\ \hline
\hdots & \hdots & \hdots & a_{m+1} & a_m & \hdots & \hdots & \hdots \\ \hline 
\end{array}\]
Note, that if $2m-\ell<0$, it is possible that the last row of the table could have some empty entries. Regardless, for all three rows to be equal we would require $a_m=c_{m} \pm 1=a_{m+1}=c_m \mp 1$, i.e. $c_m \pm 1=c_m \mp 1$, which is not possible.
More generally for $wyx\bE$, and again by Lemma~\ref{lem:adjEq}~(iii),  the first digits of $wyx\bE$ must agree with all but the last digit of $yx\bE$, and as the issue of $a_{m+1}=c_m \pm 1=c_m \mp 1$ occurs before the last digit in $yx\bE$, this problem will persist, and we conclude that if $m<\ell$, and for $s \geq 2$,
 \[ wyx\bE \neq [A,\{a_{1},\hdots,a_\ell\}^{k+s},end].\]

The result for $wxy\bE$ will follow similarly. This concludes the proof where $\bE$ is a mutation of $[n+1,n-2].$

By  \eqref{eq:SCF} and \eqref{eq:RCF}, applying a symmetry $(S^i)^\sharp$ or $(S^i R^\delta)^\sharp$ to $\bE$, $wyx\bE$, and $wxy\bE$ changes the beginning of the continued fractions of their centers by adding the same preamble to each of them. Thus, the statement of the lemma holds also for these classes.

\end{proof}

Recall from \eqref{eq:sequencewords} that given an infinite word $w$ (or equivalently, a sequence), we can define its finite subwords $w_k$. The next result uses Proposition~\ref{lem:repeat} to show that if the limit of the sequence of centers of classes $w_k \bE$ has periodic continued fraction, then $w$ has infinitely many consecutive $x$'s or $y$'s.

\begin{corollary}
    \label{prop:rpt}
    Let $A$ denote some fixed finite list of numbers and $w$ be an infinite word on $x$ and $y$, with corresponding sequence of finite subwords $(w_k)_{k\in\mathbb{N}}$. For $\bE \in \Cc$, if the sequence of centers of $w_k\bE$ converges to $[A,\{a_{1},\hdots,a_n\}^\infty]$, then $w$ must be of the form $w=x^\infty w_\star$ or $w=y^\infty w_\star$ for some finite word $w_\star$.
\end{corollary}

 \begin{proof} 
 As in the proof of Proposition~\ref{lem:repeat}, we will often use the notation $\bE_{CF[z]}=CF(z)$.
 
  Denote $\bE_k:=w_k\bE.$ 
If the sequence of centers of the classes $\bE_k$ converges to $[A,\{a_{1},\hdots,a_n\}^\infty]$, then as $k$ increases, the continued fraction of the center of $\bE_k$ agrees with more digits of $[A,\{a_{1},\hdots,a_n\}^\infty]$ from the left. 

This implies that after some $k=k_{min}$, for all $k>k_{min}$, all classes are of the form
$$\bE_{k}=[A,\{a_{1},\hdots,a_n\}^{j_k},end_{k}],$$ 
where $j_k \geq 1$ and $end_{k}$ is some finite list of integers and $j_k$ is maximal in the sense that $\bE_{k} \neq [A,\{a_{1},\hdots,a_n\}^{j_k+1},end]$ for some other $end.$ The $j_k$ must be a non decreasing sequence converging to infinity.

We show by contradiction that $w=x^\infty w_{k_{min}}$ or $w=y^\infty w_{k_{min}}.$ If not, we may write $w=vxyw_{k_{min}}$ or $w=vyxw_{k_{min}}$ where $v$ is an infinite word on $x$ and $y$. Then, by 
Proposition~\ref{lem:repeat}, neither $vxyw_{k_{min}}\bE$ nor $vyxw_{k_{min}}\bE$ can equal $ [A,\{a_1,\hdots,a_n\}^{r},\hdots]$ for all $r \geq j_{k_{min}}+2$. Hence, the sequence $j_k$ would be bounded by $j_{k_{min}}+2,$ which is a contradiction.  
\end{proof}

We now translate Corollary \ref{prop:rpt} to the language of accumulation point of infinite staircases for certain $b$-values to prove Theorem \ref{cor:rpt}.

\begin{proof}[Proof of Theorem \ref{cor:rpt}]
Recall from Section~\ref{ss:fractal} that applying a word $w$ to a triple $\Tt_n$ (or a symmetry of it) with middle class $\bE$ generates an infinite staircase with accumulation point the limit of the center of the classes $w_k\bE$, where $(w_k)_{k\in\mathbb{N}}$ is the sequence of finite subwords of $w$.
When the word is not of the form $w=x^\infty w_\star$ or $w=y^\infty w_\star$ for some finite word $w_\star$ we obtain what we call a non-principal staircase.

    Hence, by Corollary~\ref{prop:rpt}, the accumulation point of a non-principal staircase does not have a periodic continued fraction, or equivalently, it is not a quadratic irrational. 
    
    By \eqref{eq:accb}, the accumulation points are the roots of a quadratic equation with coefficients in $\Q(b)$, and as the root is not a quadratic irrational, the $b$-value is not a rational number.  
\end{proof}

\section{Special rational $b$-values do not have infinite staircases}\label{s:nosratb}

In this section, we prove Theorem \ref{thm:NoRatlAcc}. A key tool in this section, that is also used in Section \ref{s:ghosts}, is the Cremona move. Here we review its definition and introduce some related notation.

A \textbf{Cremona move} is a linear automorphism of $\mathbb{R}^{m+1}$ of one of the following two types. The first is of the form

\[ c_{ijk}(d;n_1,\hdots,n_m)=(\tilde{d};\tilde{n}_1,\hdots,\tilde{n}_m), \quad \text{with }  \begin{cases} \tilde{d}= d+\delta_{ijk} \\ 
\tilde{n}_\ell =n_\ell+\delta_{ijk} & \text{if $\ell=i,j,k$}\\
\tilde{n}_\ell=n_\ell & \text{if $\ell \neq i,j,k$}
\end{cases} \]
where $\delta_{ijk}=d-n_i-n_j-n_k$ is called the \textbf{defect}.

The second type of Cremona move is a reordering operation $c_{ij}$ where $c_{ij}(d;n_1,\hdots,n_m)$ swaps the entries $n_i$ and $n_j$. The first element of the vector can never be reordered. 

 We say two vectors are {\bf Cremona equivalent} if one can be obtained from the other via a series of Cremona moves.
 A \textbf{reduced vector} is an ordered vector  (i.e.\ $n_1 \geq n_2 \geq n_3 \geq \hdots$) for which $\delta_{123}$, the defect of the Cremona move $c_{123}$, would be non-negative. 

\subsection{Proof of Theorem~\ref{thm:NoRatlAcc}}

Theorem~\ref{thm:NoRatlAcc} is a consequence of the following more specific result:

\begin{prop} \label{prop:Cremona}
The ellipsoid embedding functions for the special rational $b$-values $b_i$, with $i\geq 1$, do not have descending infinite staircases. 
\end{prop}

\begin{proof}(of Theorem ~\ref{thm:NoRatlAcc})
In \cite[Theorem~6]{ICERM}, it was shown that the ellipsoid embedding function for $b_0=\frac15$ does not have an infinite staircase.
In \cite[Corollary~4.3.5]{MMW}, it was shown that the ellipsoid embedding functions $c_{b_i}$ for $i\geq 1$ do not have ascending staircases and Proposition~\ref{prop:Cremona} guarantees that they also do not have descending staircases.

Finally, by \cite[Theorem~1.1.1~(iv)]{MMW}, the only rational $z$-values which could be accumulation points of infinite staircases are the special rationals $z_i=\text{acc}(b_i)$.
\end{proof}

We will now prove Proposition~\ref{prop:Cremona}; the proof uses some auxiliary results which can be found in Section~\ref{subsec:tech}.

\begin{proof}(of Proposition~\ref{prop:Cremona})
For $i\geq1$, let $b_i$ be a special rational $b$-value. We will show that $c_{b_i}$ does not have a descending infinite staircase, which would necessarily accumulate at $z_i=\text{acc}(b_i)$, by showing that for $z$-values immediately to the right of $z_i$, the ellipsoid embedding function $c_{b_i}(z)$ is equal to   the linear function 
\[
\lambda_{b_i}(z)=\frac{1+z}{3-{b_i}}.
\]
 
As proved in \cite[Example~32]{ICERM}, to the right of $z_i$ we have
$$\lambda_{b_i}(z)\leq c_{b_i}(z),$$
so we must prove that in a small neighbourhood to the right of $z_i$ we also have
$$\lambda_{b_i}(z)\geq c_{b_i}(z),$$
which is \cite[Conjecture~4.3.7]{MMW}. 
We do this by showing that for such a neighbourhood there exists an embedding 
$$E(1,z) \hookrightarrow \lambda_{b_i}(z) H_{b_i}.$$

Following \cite[Theorem~2.1]{concaveconvex}, the existence of such an embedding is equivalent to the existence of a ball packing
$$\sqcup_{j}B(w_j)\sqcup B(b_i \lambda_{b_i}(z))\hookrightarrow B(\lambda_{b_i}(z)),$$
where the $w_i$ are the entries in the weight sequence of $z$ and the disjoint union is taken with multiplicities.

As explained in \cite[Section 2.2, Method 2]{ellipsoidpolydisc} and proved in \cite[Section 2.3]{BP} and \cite[Section 6.3]{KK}, the existence of such a ball packing is implied by (and is in fact equivalent to) the vector $(\lambda_{b_i}(z);b_i \lambda_{b_i}(z),\bw(z))$ becoming, under a series of Cremona moves with negative defect, a reduced vector with all non-negative entries. But this fact is proved in Proposition~\ref{prop:FinalReduced} so we are done.
\end{proof}

\begin{remark} \rm
    The statement of Proposition \ref{prop:Cremona} holds also in the case $b_0=\frac15$, this was shown in \cite[Theorem~6]{ICERM} relying on an argument involving ECH capacities. It can actually also be proved using a Cremona reduction argument but the proof above would have to be modified for that special case because in a neighbourhood to the right of $z_0$, the function $c_{b_0}(z)$ does not equal $\lambda_{b_0}(z)$. \hfill$\er$
\end{remark}

\subsection{Auxiliary results}\label{subsec:tech}

Throughout this section, we will assume (unless otherwise noted) that $i\geq 1$, that $b_i$ is a special rational $b$-value, that $z_i=\text{acc}(b_i)=:\frac{p_i}{q_i}$ is the corresponding special rational $z$-value, and that $\eps_i=(-1)^{i}$.

We are left with proving Proposition~\ref{prop:FinalReduced}, which states that for $z$ immediately to the right of $z_i$, the vector $(\lambda_{b_i}(z);{b_i} \lambda_{b_i}(z),\bw(z))$  is Cremona equivalent to a reduced vector with all non-negative entries. Since we only considering cases when $5<z<6$ and ${b_i}\lambda_{b_i}(z)<1$, we reorder the vector $(\lambda_{b_i}(z);{b_i} \lambda_{b_i}(z),\bw(z))$ and denote it by
\begin{equation}\label{eq:defBl} v_{b_i}(z):=\left(\lambda_{b_i}(z);1^{\times 5},{b_i}\lambda_{b_i}(z),\bw(z-5)\right).
\end{equation}

First we note some facts about the entries of the vector $v_{b_i}(z)$:

\begin{lemma} \label{lem:deltaWeight}
For $z=z_i+\delta$ we have:
\begin{align*}
    \lambda_{b_i}(z)&=\frac{(p_i+q_i+\de q_i)(3p_i+3q_i+\eps_i)}{8 q_i(p_i+q_i)}=\frac{z+1}{8}\left(3-\frac{\eps_i}{p_i+q_i}\right) \\
    {b_i}\lambda_{b_i}(z)&=\frac{(p_i+q_i+\de q_i)(p_i+q_i+3\eps_i)}{8 q_i(p_i+q_i)}=\frac{z+1}{8}\left(1-3\frac{\eps_i}{p_i+q_i}\right).
\end{align*}
Furthermore, assuming that $\delta>0$ is small enough we have for $z_1=\frac{35}{6}$:
\begin{equation} \label{eq:weightde} \bw(z-5)=\left(z-5,6-z^{\times 5},6z-35^{\times \ell_3},6+35\ell_3-z(6\ell_3+1)^{\times \ell_4},w_5^{\times\ell_5},\hdots\right) \end{equation}
\noindent where $\ell_3=\lfloor \frac{1-6\de}{36\de} \rfloor=\lfloor \frac{6-z}{6z-35} \rfloor$, and for $i\geq 2$:
\begin{equation*}
\bw(z-5)=\left(z-5,6-z^{\times 4},\hdots\right).
\end{equation*}

\end{lemma}

\begin{proof} 
The first two identities are straightforward computations using the following fact proved in \cite[Lemma~2.1.1~(iv)]{MM}: 
$$b_i=\frac{p_i+q_i-3\eps_i}{3(p_i+q_i)-\eps_i}$$ where $\eps_i=(-1)^i$.\footnote{Note that in \cite{MM}, the authors use the notation $\acc^{-1}_U=\acc^{-1}_{+1}$ and $\acc^{-1}_L=\acc^{-1}_{-1}$.}

By \cite[Lemma~2.1.5~(i)]{MM}, for $i \geq 1$ we have the continued fraction expansion
$z_{i}=[5,\{1,4\}^{i-1},1,5].$ Let $\delta$ be small enough such that the continued fraction of $z=z_i+\de$ has the form $[5,\{1,4\}^{i-1},1,5,a,\ldots]$.\footnote{The inequality $z_i<[5,\{1,4\}^{i-1},1,5,a,\ldots]$ holds because the continued fraction of $z_i$ has odd length.} To compute the weight expansions, we set $w_{0}=1$ and $w_1=z-5$, and then inductively compute $\ell_{k}=\lfloor \frac{w_{k-1}}{w_{k}} \rfloor$ and $w_{k}=w_{k-2}-\ell_{k-1} w_{k-1}$. 
\end{proof}

Next, we define two auxiliary functions and state some of their properties. By abuse of notation, we treat as equal vectors that are  identical except for being padded on the right with a different number of zeros:

\begin{lemma} \label{lem:crGen} We define the composition of Cremona moves 
$$\Xi=c_{367}c_{345}c_{127}c_{456}c_{123}$$ 
and a  dropping\footnote{This function will only ever be used when $x_1=x_2=x_4=x_5=x_6=0$, so dropping these entries is in fact a reordering move: under our abuse of notation it amounts to moving zeros all the way to the end of the vector.} and reordering function 
$$\Gamma\left((x_0;x_1,\ldots,x_{11},\vec{y})\right)=(x_0;x_3,x_8,x_{9},x_{10},x_{11},x_7,\vec{y}).$$ 
Each of these functions is linear, and they have the following properties:
\vspace{0.3cm}
\begin{itemize}[leftmargin=*]
    \item $\!\begin{aligned}[t]
                    \Xi(w;x^{\times 5},y,z)=(8w-3 (5x+y+z); & 3w-6x-y-z^{\times 2}, 3w-5 x-y-2z,
\\ & 3w-6x-y-z^{\times 3}, 3w-5x-2y-z),
                \end{aligned}$
\vspace{0.3cm}
    \item $\Xi(w+e;x^{\times 5},y+3e,z)=\Xi(w;x^{\times 5},y,z)+(-e;0^{\times 6},-3e),$
\vspace{0.3cm}
    \item $(\Gamma \circ\Xi)((1;0^{\times 5},3,\vec{0}))=-(1;0^{\times 5},3,\vec{0}).$  
\end{itemize}
\end{lemma}

\begin{proof}
These are straightforward computations. 
\end{proof}

\begin{remark}\rm
The composition $\Xi$ is the same that was introduced in \cite[Section 4.1]{MM} to show that two exceptional classes related by the symmetry $S$ are Cremona equivalent. \hfill$\er$
\end{remark}

The first result will allow us to reduce to the case of $b_1=11/31$ and $z_1=35/6.$

\begin{prop} 
 For $\delta>0$ small enough we define $\tilde{z}_1=z_1+\delta$ and $\tilde{z}_k=S^k(\tilde{z}_1)$. Then we have: 
\begin{equation}\label{eq:from lemma}
\tilde z_1\cdots \tilde z_{k-1} \left(\Gamma\circ\Xi\right)^{k-1}\left(v_{b_k}(\tilde z_k)\right)=v_{b_1}(\tilde z_1)-e_k (1;0^{\times 5},3,\vec{0}),
\end{equation}
where
\begin{equation}\label{eq:defEk}e_k=\eps_1q_1\delta\sum_{i=1}^{k-1}\frac{1}{(7p_i-q_i)(p_i+q_i)}.
\end{equation}
\label{prop:reduceCase}
\end{prop}

\begin{proof}
Let $\delta=\delta_1$ be small enough that Lemma \ref{lem:deltaWeight} holds.
We will prove this result by induction on $k$. 
The base case of $k=1$ is immediate, as both sides of identity \eqref{eq:from lemma} equal $v_{b_1}(\tilde z_1)$.
For the induction step we assume that \eqref{eq:from lemma} holds for $k$ and aim to show that: 
\begin{align}\label{eq:goalinduction}
\tilde z_1\cdots \tilde z_{k} &\left(\Gamma\circ\Xi\right)^{k}\left(v_{b_{k+1}}(\tilde z_{k+1})\right)=\nonumber\\
&=\tilde z_1\cdots \tilde z_{k-1} \left(\Gamma\circ\Xi\right)^{k-1}\left(v_{b_k}(z_k)\right)+\frac{\eps_1q_1\delta_1}{(7p_k-q_k)(p_k+q_k)} (1;0^{\times 5},3,\vec{0}).
\end{align}

We begin by proving the following auxiliary fact, where $\delta_k=\tilde{z}_k-z_k$:
\begin{equation}\label{eq:sublemma}
\tilde{z}_k\,\left(\Gamma\circ\Xi\right)(v_{b_{k+1}}(\tilde{z}_{k+1}))=v_{b_k}(\tilde{z}_k)-\frac{\de_k \eps_k q_k}{(7p_k-q_k)(p_k+q_k)}(1;0^{\times 5},3,\vec{0}).
\end{equation} 
First we use \eqref{eq:defBl} and Lemma~\ref{lem:deltaWeight} to expand $v_{b_{k+1}}(\tilde z_{k+1})$:

$$v_{b_{k+1}}(\tilde z_{k+1})=\left(\lambda_{b_{k+1}}(\tilde z_{k+1});1^{\times 5},{b_{k+1}}\lambda_{b_{k+1}}(\tilde z_{k+1}),\tilde z_{k+1}-5,(6-\tilde z_{k+1})^{\times 4},\bw_{3\cdots}(\tilde z_{k+1}-5)\right),$$
where $\bw_{3\cdots} (x)$ represents the weight sequence of $x$ from position six onward. Applying Lemma~\ref{lem:crGen}, and simplifying using Lemma \ref{lem:deltaWeight} we obtain:

\begin{align} 
\tilde z_{k}\left(\Gamma\circ\Xi\right)\left( v_{b_{k+1}}(\tilde z_{k+1})\right)=&\tilde z_{k}\bigg( \frac18 \left(3(7-\tilde z_{k+1})+\frac{\eps_{k+1} (1+\tilde z_{k+1})}{p_{k+1}+q_{k+1}}\right); \nonumber\\ 
&\phantom{space}(6-\tilde{z}_{k+1})^{\times 5},\frac18 \left((7-\tilde z_{k+1})+\frac{3 \eps_{k+1} (1+\tilde z_{k+1})}{p_{k+1}+q_{k+1}}\right),\bw_{3\cdots}(\tilde z_{k+1}-5)  \bigg)\nonumber \\ 
=&\bigg( \frac18 \left(3(\tilde{z}_k+1)-\frac{7\eps_k\de_k}{7p_k-q_k}-\frac{\eps_k}{q_k}\right); \nonumber\\
&\phantom{space}1^{\times 5},\frac18 \left((\tilde{z}_k+1)-\frac{21\eps_k\de_k}{7p_k-q_k}-\frac{3\eps_k}{q_k}\right),\tilde{z}_k \,\bw_{3\cdots}\left(1-\frac{1}{\tilde{z}_k}\right)\bigg). \label{eq:threelines}
\end{align}

\noindent In the last equality, the simplifications use the facts that
\begin{equation}\label{eq:useful twice}\eps_{k+1}=-\eps_k\mbox{ \,\,\,\,\,and \,\,\,\,\, } \tilde{z}_{k+1}=6-\frac{1}{\tilde z_k}\quad\mbox{ and }\quad\de_{k+1}=\frac{q_k^2\de_k}{p_k^2+p_k q_k \de_k}.
\end{equation}
From \eqref{eq:threelines} we obtain \eqref{eq:sublemma} as desired by using the formulas of Lemma~\ref{lem:deltaWeight} and noting that 
$$\tilde{z}_k\, w_{3\cdots}\left(1-\frac{1}{\tilde{z}_k} \right)=\bw(\tilde{z}_k-5)$$ 
because if we decompose a $(1-\frac{1}{\tilde{z}_k})\times 1$ rectangle into squares we get the weight sequence
$$\tilde{z}_k \bw\left(1-\frac{1}{\tilde{z}_k}\right)=\left( \tilde z_k-1,1^{\times 4},\bw(\tilde{z}_k-5)\right).$$

Now we turn back to proving \eqref{eq:goalinduction} towards the induction step. Because $\Gamma\circ\Xi$ is linear, applying $\tilde z_1\cdots \tilde z_{k-1}(\Gamma\circ\Xi)^{k-1}$ to both sides of \eqref{eq:sublemma} yields
\begin{align*}
\tilde z_1\cdots \tilde z_k&\,\left(\Gamma\circ\Xi\right)^k(v_{b_{k+1}}(\tilde z_{k+1}))=\\
&=\tilde z_1\cdots \tilde z_{k-1} \left(\Gamma\circ\Xi\right)^{k-1}v_{b_k}(\tilde z_k)-\frac{(-1)^{k-1}\tilde z_1\cdots \tilde z_{k-1}\de_k \eps_k q_k}{(7p_k-q_k)(p_k+q_k)}(1,0^{\times 5},3,\vec{0});
\end{align*}
here we also use the third property in Lemma \ref{lem:crGen}.
Comparing this to \eqref{eq:goalinduction}, we are left to prove that 
\[
\eps_1q_1\delta_1=(-1)^{k-1}\tilde z_1\cdots \tilde z_{k-1}\de_k \eps_k q_k.
\]
This is straightforward to prove by induction, using \eqref{eq:useful twice}.
\end{proof}

The next two lemmas concern the sum in the term $e_k$ defined in \eqref{eq:defEk}. The first one is used in the proof of the second. Recall from \eqref{eqn:yk} the definition of the sequence $y_k$, where  $y_k=q_{k-1}=p_{k-2}$.

\begin{lemma} \label{lem:partialSum} 
For $k\geq1$ the following equality holds:
\begin{equation}\label{eq:thatsum} \sum_{i=1}^k \frac{1}{(7p_i-q_i)(p_i+q_i)}=  \sum_{i=1}^k \frac{1}{(7y_{i+2}-y_{i+1})(y_{i+2}+y_{i+1})}= \frac{y_k}{41(35(y_{k+1}+y_k)-6(y_k+y_{k-1}))}.
\end{equation}
\end{lemma}
\begin{proof}
The first equality just comes from changing the $p_i,q_i$ notation to the $y_i$ notation. 
As
\begin{equation}\label{eq:top of proof} y_{k+3}=6y_{k+2}-y_{k+1}=6(6y_{k+1}-y_{k})-y_{k+1}=35y_{k+1}-6y_{k},
\end{equation}
we can rewrite the second equality of \eqref{eq:thatsum} as
\begin{equation}\label{eq:line below as}
\sum_{i=1}^k \frac{1}{(y_{i+3}+y_{i+2})(y_{i+2}+y_{i+1})}= \frac{y_k}{41(y_{k+3}+y_{k+2})}.
\end{equation}

We now prove \eqref{eq:line below as} by induction. The case for $k=1$ can be easily checked. To prove the inductive step, it suffices to check that 

$$\frac{1}{(y_{k+4}+y_{k+3})(y_{k+3}+y_{k+2})}=\frac{y_{k+1}}{41(y_{k+4}+y_{k+3})}-\frac{y_k}{41(y_{k+3}+y_{k+2})}.$$
Clearing denominators this is equivalent to showing that
$$41=y_{k+1}y_{k+3}-y_ky_{k+4}+y_{k+1}y_{k+2}-y_ky_{k+3},$$
which holds true by \eqref{eq:top of proof}, the recursion $y_k=6y_{k-1}-y_{k-2}$, and the identity $y_k^2-y_{k+1}y_{k-1}=1$, proved in \cite[Lemma~2.1.3~(ii)]{MM}.
\end{proof}

\begin{lemma}\label{lem:bound} 
We have
\[ \lim_{k \to \infty} \sum_{i=1}^k \frac{1}{(7p_i-q_i)(p_i+q_i)}=\frac{29}{82}-\frac{1}{2\sqrt{2}}.\]
Furthermore, for all $k\geq 1$, 
\[ 0<\sum_{i=1}^k \frac{1}{(7p_i-q_i)(p_i+q_i)}<\frac{29}{82}-\frac{1}{2\sqrt{2}}.\]
\end{lemma}
\begin{proof}
By Lemma~\ref{lem:partialSum}, the partial sums are equal to 
\[  \frac{y_k}{41(35(y_{k+1}+y_k)-6(y_k+y_{k-1}))}= \frac{1}{41(35(\frac{y_{k+1}}{y_k}+1)-6(1+\frac{y_{k-1}}{y_k}))}\]
The conclusion follows from the limits from \cite[Lemma~2.1.5]{MM} 
$$\lim_{k \to \infty} \frac{y_{k+1}}{y_k}=3+2\sqrt{2} \qquad \mbox{ and } \qquad \lim_{k \to \infty} \frac{y_k}{y_{k+1}}= 3-2\sqrt{2}.$$
The second statement is clear because the summands are all positive. 
\end{proof} 

By Proposition~\ref{prop:reduceCase} we know that any vector $v_{b_k}(\tilde{z_k})$ is Cremona equivalent to (a rescaling of) $v_{b_1}(\tilde{z_1})+e(1;0^{\times 5},3,\vec{0})$, now we show that this latter vector is Cremona equivalent to a reduced vector:

\begin{prop} \label{prop:genEquiv}
For $\delta>0$ small enough and $e=6\delta\xi$ with $0 \leq \xi \leq \frac{29}{82}-\frac{1}{2\sqrt{2}}$, the vector 
$$v_{b_1}(z_1+\de)+e(1;0^{\times 5},3,\vec0)$$
is Cremona equivalent to a reduced vector with all non-negative entries.
\end{prop}

\begin{proof}
We expand $v_{b_1}(z_1+\de)$ using \eqref{eq:defBl} and Lemma~\ref{lem:deltaWeight} and apply the following sequence of Cremona moves to $v_{b_1}(z_1+\de)+e(1;0^{\times 5},3,\vec0)$: 
\[c_{8,9,10}c_{3,11,12}c_{8,9,10}\,\Xi.\] 
This yields the vector
\[ \left(\tfrac{1}{12}+\tfrac{829}{82}\de-5e;0^{\times 2},\tfrac{174}{41}\de-2e,0^{\times 3},\tfrac{1}{12}-\tfrac{11}{82}\de-3e,\tfrac{174}{41}\de-2e^{\times 5},6\de^{\times \ell_3},\tfrac16-(6\ell_3+1)\de^{\times \ell_4},w_5^{\times\ell_5},\hdots\right).\] 
Reordering and setting $e=-6\delta \xi$ we get 
\begin{equation}\label{eq:weight3} 
\left(\tfrac{1}{12}+(\tfrac{829}{82}+30\xi)\de;\tfrac{1}{12}-(\tfrac{11}{82}-18\xi)\de,6\de^{\times \ell_3},(\tfrac{174}{41}+12\xi)\de^{\times 6},\tfrac16-(6\ell_3+1)\de^{\times \ell_4},w_5^{\times\ell_5},\hdots\right),
\end{equation}
where $\ell_3=\lfloor \frac{1-6\de}{36\de} \rfloor$.

We want to put all but the first entry of \eqref{eq:weight3} in decreasing order before we apply a series of Cremona moves which will lead us to a reduced vector.

For small enough $\de$ and using the bounds on $\xi$, we have 
$$\tfrac{1}{12}-(\tfrac{11}{82}-18\xi)\de>6\de>(\tfrac{174}{41}+12\xi)\de>0$$
and also
$$\tfrac16-(6\ell_3+1)\de \geq w_5 \geq \ldots \geq 0$$ 
since these entries appear in this order on the weight vector of $z_1+\de$ as in\eqref{eq:weightde}. We have to consider two cases: 
$$(\tfrac{174}{41}+12\xi)\de\geq\tfrac16-(6\ell_3+1)\de \mbox{\,\,\,\,\, or \,\,\,\,\,} (\tfrac{174}{41}+12\xi)\de<\tfrac16-(6\ell_3+1)\de.$$

 \vspace{0.2cm}

{\bf Case 1:} When $(\tfrac{174}{41}+12\xi)\de\geq\tfrac16-(6\ell_3+1)\de$, the vector as written in \eqref{eq:weight3} already has all entries but the first one in decreasing order. Note that this case includes the case when $\ell_4=0$ as in this case $w_4=\tfrac16-(6\ell_3+1)\de=0$. We do $\lfloor \frac{\ell_3}{2}\rfloor$ Cremona moves,  each with defect $(12\xi-\tfrac{72}{41})\de <0$. The Cremona moves we perform are 
\[ c_{1,2\lfloor \ell_3/2 \rfloor-1,2\lfloor \ell_3/2 \rfloor}\hdots c_{1,i,i+1} \hdots c_{123}\]

If $\ell_3$ is odd, we obtain
$$\left(\tfrac{1}{12}+(\tfrac{901-72\ell_3}{82}+6(4+\ell_3)\xi)\de;
\tfrac{1}{12}+(\tfrac{61-72\ell_3}{82}+6(2+\ell_3)\xi)\de,
6\de,
(\tfrac{174}{41}+12\xi)\de^{\times (\ell_3+5)},\hdots\right).$$
For small enough $\de$ this vector has all but the first entry in decreasing order since
\begin{equation} \label{eq:order1} \tfrac{1}{12}+(\tfrac{61-72\ell_3}{82}+6(2+\ell_3)\xi)\de>6\de, \end{equation}
 which can be checked using the lower bound on $\xi$ and the bound $\ell_3=\lfloor \frac{1-6\de}{36\de} \rfloor\leq \frac{1-6\de}{36\de} $. The next Cremona move would have defect equal to zero. Further, as all entries of the vector are non-negative, we are done.

If $\ell_3$ is even, we obtain
$$\left(\tfrac{1}{12}+(\tfrac{829-72\ell_3}{82}+6(5+\ell_3)\xi)\de;
\tfrac{1}{12}+(\tfrac{-11-72\ell_3}{82}+6(3+\ell_3)\xi)\de,
(\tfrac{174}{41}+12\xi)\de^{\times (\ell_3+6)},
\hdots\right).$$
As above, for small enough $\de$ we have that indeed 
\begin{equation} \label{eq:order2} \tfrac{1}{12}+(\tfrac{-11-72\ell_3}{82}+6(3+\ell_3)\xi)\de > (\tfrac{174}{41}+12\xi)\de, \end{equation} 
so the vector has all but the first entry in decreasing order. The next standard Cremona move would have defect $(\tfrac{72}{41}-12\xi)\de$, which is positive, so the vector is reduced. Further, as the entries are non-negative, we are done.

{\bf Case 2:} When $(\tfrac{174}{41}+12\xi)\de<\tfrac16-(6\ell_3+1)\de$, we begin by showing the following:

\vspace{0.2cm}

\noindent\textbf{Claim:} 
If  $(\tfrac{174}{41}+12\xi)\de<\tfrac16-(6\ell_3+1)\de$ then  $\ell_4\leq1$.  

\begin{proof}
We aim to prove the contrapositive and therefore assume that $\ell_4\geq2$. Rewriting $\ell_4$ (see the beginning of Section \ref{ss:reviewEC}) as
$$\ell_4=\left\lfloor\frac{w_3}{w_4}\right\rfloor=\left\lfloor\frac{1}{\frac{1-6\de}{36\de}-\ell_3}\right\rfloor$$
we get that the inequality $\ell_4\geq2$ is equivalent to
$$\ell_3\geq \frac{1-6\de}{36\de}-\frac12.$$
Using this bound and then as $\xi \leq \tfrac{29}{82}-\tfrac{1}{2\sqrt{2}}$, we obtain, as desired,
$$\tfrac16-(6\ell_3+1)\de\leq3\de<(\tfrac{174}{41}+12\xi)\de.$$\end{proof}

Case 1 above includes the case of $\ell_4=0$, so we are left to consider $\ell_4=1$.

\vspace{0.2cm}

\noindent\textbf{Claim:} 
When $\ell_4=1$, we have $(\tfrac{174}{41}+12\xi)\de>w_5$.

\begin{proof}
As in \eqref{eq:weightde}, we have $w_3=6\de.$ In how we preform the weight expansion, this implies that $1/2(w_3)=3\de>w_5$. We can then easily check that
\[(\tfrac{174}{41}+12\xi)\de>3\de \]
\end{proof}

Also, $6\de>\tfrac16-(6\ell_3+1)\de$ as these are consecutive terms in the weight sequence of $z_1+\de.$ Ordering all but the first entry of vector \eqref{eq:weight3} so that they are in decreasing order we obtain:
\begin{equation*}\label{eq:weight4} 
(\tfrac{1}{12}+(\tfrac{829}{82}+30\xi)\de;
\tfrac{1}{12}-(\tfrac{11}{82}-18\xi)\de,
6\de^{\times \ell_3},
\tfrac16-(6\ell_3+1)\de,
(\tfrac{174}{41}+12\xi)\de^{\times 6},
w_5^{\times\ell_5},\hdots).
\end{equation*}

We do $\lfloor \frac{\ell_3}{2}\rfloor$ Cremona moves,  each with defect $(12\xi-\tfrac{72}{41})\de$, and  collect like terms after each move.

If $\ell_3$ is odd, we obtain
\begin{multline} \label{eq:oddCase2}
(\tfrac{1}{12}+(\tfrac{901-72\ell_3}{82}+6(4+\ell_3)\xi)\de;
\tfrac{1}{12}+(\tfrac{61-72\ell_3}{82}+6(2+\ell_3)\xi)\de,
6\de,\\
\tfrac16-(6\ell_3+1)\de,
(\tfrac{174}{41}+12\xi)\de^{\times (\ell_3+5)},w_5^{\times\ell_5}\hdots).
\end{multline}

If $\ell_3$ is even, we obtain
\begin{multline} \label{eq:evenCase2} 
(\tfrac{1}{12}+(\tfrac{829-72\ell_3}{82}+6(5+\ell_3)\xi)\de;
\tfrac{1}{12}+(\tfrac{-11-72\ell_3}{82}+6(3+\ell_3)\xi)\de,\\
\tfrac16-(6\ell_3+1)\de,
(\tfrac{174}{41}+12\xi)\de^{\times (\ell_3+6)},w_5^{\times\ell_5}\hdots).
\end{multline}

If $\ell_3$ is odd, then by \eqref{eq:order1}, the vector \eqref{eq:oddCase2} is correctly ordered and further, all entries are non-negative. The defect of the next Cremona move is $-\frac16+(\frac{215}{41}+6\ell_3+12\xi)\de$. If it is non-negative, then the vector is reduced. If not, we perform the Cremona move with this defect, and claim after this no reordering is necessary as we add the assumed negative defect to the third entry and obtain the fourth entry: 
\[ \tfrac{1}{6}-(6\ell_3+1)\de+(-\tfrac16+(\tfrac{215}{41}+6\ell_3+12\xi)\de)=(\tfrac{174}{41}+12\xi)\de\]
for small enough $\de.$ Then, the following Cremona move will have defect 
$$\frac16-\left(\frac{215}{41}+6\ell_3+12\xi\right)\de>0.$$ 

If $\ell_3$ is even, then by \eqref{eq:order2}, the vector is written in decreasing order. The defect of the next standard Cremona move is $-\frac16+(7+6\ell_3)\de$, which we can check is non-negative since $\ell_3=\lfloor \frac{1-6\de}{36\de}\rfloor>\frac{1-6\de}{36\de}-1$ and furthermore, all the entries of \eqref{eq:evenCase2} are non-negative.

In either case, the result is a reduced vector with non-negative entries. 

\end{proof}

We now prove the missing final result for the proof of Proposition \ref{prop:Cremona} and thus conclude the proof of Theorem \ref{thm:NoRatlAcc}.

\begin{prop} \label{prop:FinalReduced} 
For each $k\geq1$ there exists a $\delta_k>0$ such that for all $0<\delta<\delta_k$, the vector 
\begin{equation}\label{eq:unorder}
(\lambda_{b_k}(z_k+\delta);{b_k} \lambda_{b_k}(z_k+\delta),w(z_k+\delta))\end{equation}
is Cremona equivalent to a reduced vector with all non-negative entries.
\end{prop}
  
\begin{proof}
Fix a value $\delta_1$ which satisfies Propositions \ref{prop:reduceCase} and \ref{prop:genEquiv} and define $\delta_k=S^k(z_1+\delta_1)-z_k$.

Note that for any $\delta<\delta_1$, if $e_k$ is defined as in \eqref{eq:defEk} then by Lemmas~\ref{lem:partialSum} and \ref{lem:bound}, the vector \eqref{eq:from lemma} satisfies the conditions of Proposition \ref{prop:genEquiv}.

Now, for any $\tilde{\delta}<\delta_k$, and because $S$ is monotone increasing, the equation
$$z_k+\tilde\delta=S^k(z_1+\delta)$$
defines a $\delta$ which satisfies $\delta<\delta_1$. Define also $e_k$ as in \eqref{eq:defEk}, $\tilde z_1=z_1+\delta$ and $\tilde z_i=S^i(\tilde z_1)$, which in particular gives $\tilde z_k=z_k+\tilde\delta$. Then by Proposition \ref{prop:reduceCase} we have \eqref{eq:from lemma}, the right hand side of which is Cremona equivalent to a reduced vector with all non-negative entries by Proposition \ref{prop:genEquiv}. But if a positive scaling of the vector $v_{b_k}(\tilde z_k)$ is Cremona equivalent to a reduced vector with all non-negative entries, then so is $v_{b_k}(\tilde z_k)$ itself, which is what we wanted to prove because by \eqref{eq:defBl} it is a reordering of \eqref{eq:unorder}.

\end{proof}

\section{There is no descending staircase for $b=1/3$}\label{s:ghosts}

We begin this section by noting a change of notation: with the exception of the Cremona reduction in the proof of Lemma \ref{lem:Ekiexceptional}, we will use the notation $(d,m;\bm)$ for exceptional classes rather than $(d;m,\bm)$. The notation  $(d,m;\bm)$ emphasizes the different roles which $m$ and $\bm$ play in this section. See, for example, Definition \ref{def:Eghostski}. In a more concrete example, the class $(3,1;2,1^{\times5})$ is not perfect, while its reordering $(3,2;1^{\times6})$ is perfect. Placing the semicolon after the $m$ coordinate emphasizes that perfectness is a property of the difference between $m$ and $\bm$. We will also occasionally refer to classes relevant to the embedding functions for $B(1)$ and $P(1,1)$, in which case we will use the semicolon placements from \cite{McDuffSchlenk12} and \cite{FrenkelMuller15}, respectively.

We now collect some thoughts about the embedding function $c_{1/3}$. It was proved in \cite[Theorem~1.19]{AADT} that it has an ascending staircase accumulating to $3+2\sqrt{2}=:\sigma^2$, where $\sigma=1+\sqrt{2}$ is the silver ratio. In this section, we will prove Theorem \ref{thm:nodesc13}, concluding that $c_{1/3}$ does not have a descending staircase by showing that it is equal to the obstruction coming from the exceptional class
\[
\bE_0(2):=(3,1;2,1^\times{5}).
\]
It is important to note that the ascending staircase of $c_{1/3}$ is in essence different from that of a principal (ascending) staircase. One manifestation of this is that principal staircases have only one staircase because there is an obstruction from a perfect class which equals $c_b(z)$ for $z$ on one side of their accumulation point, whereas in the case of $c_{1/3}$, the class $\bE_0(2)$
is not perfect. As a result, we do not immediately know that $\bE_0(2)$ provides the largest possible obstruction over the interval when it is visible in $c_{1/3}$ (compare to \cite[Proposition~21]{ICERM}) and so it would be possible to have a descending staircase with infinitely many steps given by perfect classes' obstructions that are larger than that of $\bE_0(2)$. This is what happens in \cite[Figure~5.3.1]{ICERM} for a different $b$-value. 

The special rational values play an important role here: the $b_i$s converge to $1/3$ and the $\acc(b_i)=z_i$s converge to $\sigma^2$. Indeed, the function $c_{1/3}$ can be seen as the limit of the functions $c_{b_i}$ for the special rational $b$-values in the sense of \cite[Proposition~4.3.2]{MMW}: the steps of $c_{1/3}$ successively appear as steps in the functions $c_{b_i}$ as $i$ increases. Furthermore, there exists a subfamily of the ``ghost stairs'' defined in Section \ref{ss:regc} (these are non-perfect exceptional classes) whose obstructions agree with that of $\bE_0(2)$ between $\si^2$ and $z_i$, as explained in  Remark \ref{rmk:ekifacts}.

We start by proving the lower bound in Theorem \ref{thm:nodesc13}. The upper bound will be proved in Proposition \ref{prop:regproof}, which will further rely on the results of Section \ref{ss:reguv}-\ref{ss:regc}.

\begin{lemma}\label{lem:E02} For $\sigma^2\leq z\leq6$,
    \[
    c_{1/3}(z)\geq\frac{3(z+1)}{8}.
    \]
\end{lemma}
\begin{proof}
    This inequality comes from the obstruction from the non-perfect exceptional class $\bE_0(2)$. When $z\in[\sigma^2,6]$, we may write $z=5+x$ and $\bw(z)=(1^{\times5},x,\dots)$. Thus
\[
c_{1/3}(z)=\frac{(2,1^{\times5})\cdot(1^{\times5},x)}{3-1/3}=\frac{6+x}{8/3}
\]
from which the conclusion follows.
\end{proof}

In this section we will frequently use several facts which we state here for convenience. The next two lemmas were upgraded (in \cite[Lemma~14]{ICERM}) in the case of Hirzebruch surfaces $H_b$ (equivalently, toric domains $X_b$) to Diophantine classes from their original proofs for exceptional classes in \cite{AADT}. First recall that the \textbf{length} $\ell(\bm)$ of a vector $\bm$ is its cardinality, and the \textbf{length} of $z\in\Q$ is the length $\ell(z):=\ell(\bw(z))$ of its weight expansion.

\begin{lemma}[{\cite[Lemma~2.28~(1)]{AADT}}]\label{lem:cghmp21}
    Let $\bE=(d,m;\bm)$ be a Diophantine class. If $\ell(z)<\ell(\bm)$ then $\mu_{\bE,b}(z)\leq\vol_b(z)$.
\end{lemma}

\begin{lemma}[{\cite[Proposition~2.30]{AADT}}]\label{lem:cghmp23} Let $\bE=(d,m;\bm)$ be a Diophantine class. On each maximal interval $I$ on which $\mu_{\bE,b}>\vol_b$:
\begin{itemize}
    \item[{\rm (i)}] The obstruction $\mu_{\bE,b}$ is continuous and piecewise-linear.
    \item[{\rm (ii)}] The obstruction $\mu_{\bE,b}$ has a unique nondifferentiable ``break point'' on $I$ and this is at \sout{a} the unique value $a\in I$ with $\ell(a)=\ell(\bm)$.
\end{itemize}
\end{lemma}

The following lemma puts further restrictions on those $\bE$ with $\mu_{\bE,b}>\vol_b$.

\begin{lemma}[{\cite[Lemma~15~(i)]{ICERM}}]\label{lem:bd-m}
    Let $\bE=(d,m;\bm)$ be a Diophantine class. If $\mu_{\bE,b}$ is ever greater than $\vol_b$, then
    \[
    |bd-m|<\sqrt{1-b^2}.
    \]
    Furthermore,
    \[
    \mu_{\bE,b}(z)<\vol_b(z)\sqrt{1+\frac{1}{d^2-m^2}}.
    \]
\end{lemma}

\subsection{The $z$-interval $\left(\frac{95+4\sqrt{2}}{17},6\right]$}

We first compute $c_{1/3}$ on the interval $\left(\frac{95+4\sqrt{2}}{17},6\right]\approx(5.921,6]$, which can be done by analyzing short exceptional classes. (The length of $\bm$ is bounded by $d$ using the quadratic Diophantine equation $d^2-m^2+1=\bm\cdot\bm$, so $\bm$ being short is equivalent to $d$ having a finite upper bound.)

The following lemma lists all possibly obstructive short exceptional classes, using only the first principles of the Diophantine equations. The condition $|d-3m|\leq2$ is motivated by the requirement that $\bE$ be live for $b=1/3$; see Lemma \ref{lem:bd-m}. Their obstruction functions will be analyzed in the proof of Lemma \ref{lem:6reg}, where packing stability can be used to prove that they are the only possible obstructions when $z$ is close enough to (and smaller than) six.

\begin{lemma}\label{lem:E7} The exceptional classes $\bE$ with $\ell(\bm)\leq7$ and $|d-3m|\leq2$ are
\begin{multline}
\hspace{1.45cm}(0,0;-1),\quad (1,0;1^{\times2}),\quad (2,0;1^{\times5}),\quad (1,1;1),\quad (2,1;1^{\times4}),\quad (3,1;2,1^{\times5}),\quad \\\hspace{1.4cm}(4,1;2^{\times3},1^{\times4}),\quad (4,2;2^{\times2},1^{\times5}),\quad (5,1;2^{\times6},1),\quad (5,2;2^{\times5},1^{\times2}),\quad (6,2;3,2^{\times6}).\hspace{2.5cm}\nonumber
\end{multline}
\end{lemma}
\begin{proof}

By the Diophantine equations \eqref{eqn:D}, we have\footnote{The square of a sum of $n$ terms is at most $n$ times the sum of their squares by Jensen's inequality.}
\[
((3d-m)-1)^2=\left(\sum_{i=1}^{\ell(\bm)}m_i\right)^2\leq\ell(\bm)\sum_{i=1}^{\ell(\bm)}m_i^2=\ell(\bm)(d^2-m^2+1),
\]
which when $\ell(\bm)\leq7$ implies
\begin{align}
(3d-m)^2-2(3d-m)+1&\leq7(d^2-m^2+1)\nonumber
\\\iff d^2-3d+4m^2+m-3dm&\leq3.\label{eqn:72}
\end{align}
Set $d=3m+c$. Using the bounds $|c|\leq2$ and $|\bm|\leq d^2-m^2+1$ from the quadratic Diophantine equation, we may solve \eqref{eqn:72} and obtain the eleven exceptional classes listed. All Cremona reduce to $(0,0;-1)$.
\end{proof}

The following lemma will be vital to the proof by contradiction in Proposition \ref{prop:regproof}, where we will use it to argue that if $c_{1/3}$ is not smooth on $[3+2\sqrt{2},6]$ then (because it must be piecewise linear) its last nonsmooth point before six must be concave up (an ``inner corner,'' i.e., on a small enough neighborhood $c_{1/3}$ is convex). Another way in which Lemma \ref{lem:6reg} will aid us is that the proof of Proposition \ref{prop:regproof} requires the fact that $\mu_{\bE,1/3}(6)\leq21/8$ whenever $\ell(\bm)=6$. We prove this claim in the course of the proof of Lemma \ref{lem:6reg}, as explained in 
Remark \ref{rmk:sig6reg}).

\begin{lemma}\label{lem:6reg} We have $c_{1/3}(z)=\frac{3(z+1)}{8}$ for $z\in\left(\frac{95+4\sqrt{2}}{17},6\right]$. 
\end{lemma}

We do not claim that $\left(\frac{95+4\sqrt{2}}{17},6\right]\approx(5.921,6]$ is the largest interval on which the following proof method would work. The lower bound is convenient because it allows us to directly invoke the ``easy'' formulation of packing stability from \cite[Corollary~1.2.4]{McDuffSchlenk12}, and it could potentially be improved. The upper bound is likely the maximal possible, since from visual inspection of the lower bound on $c_{1/3}$ from several thousand obstructions it appears that $c_{1/3}(z)=21/8$ when $z$ is slightly larger than six; likely we could prove this, for example, following the methods of \cite[Section 5.2]{McDuffSchlenk12}, or, as pointed out by the referee, using the Cremona reduction methods of Section \ref{s:nosratb}.

\begin{proof}
    Since $c_{1/3}$ is continuous, it suffices to prove the claim for $z\in\Q$. First note that because $3+2\sqrt{2}<\frac{95+4\sqrt{2}}{17}<6$, we have by Lemma \ref{lem:E02} that $c_{1/3}(z)\geq\frac{3(z+1)}{8}$, thus it remains to show that $c_{1/3}(z)\leq\frac{3(z+1)}{8}$.

    Let $z=5+x\in\mathbb{Q}$. We have
    \[
    \bw(z)=(1^{\times5},x,w_7,\dots,w_M),
    \]
    where the first nine of $w_7,\dots,w_M$ must be equal because
    \[
    w_7=1-x\leq\frac{x}{9} \iff x\geq\frac{9}{10}.
    \]
    
    Define $z'=\frac{x}{1-x}$, so $z'\geq 9$. Furthermore, the $w_7$-scaling of the weight expansion of $z'$ is given by $w_7\bw(z')=(w_7,\dots,w_M)$,  and corresponds to the upper right rectangle determining $\bw(z)$, with the five $1\times1$ and the single $x\times x$ squares removed (the largest six squares).

    Thus
    \[
    z=5\cdot1^2+x^2+\vol(E(w_7,w_7z')), \mbox{ where } \vol(E(w_7,w_7z'))=\sum_{i=7}^Mw_i^2.
    \]
For the second equality, note that $\vol(E(w_7,w_7z'))=w_7^2\vol(E(1,z'))$, which is the sum of the areas of the squares determining $w_7\bw(z')$.

By the definition of $c_b$,
\[
c_{1/3}(z)\leq\frac{3(z+1)}{8}\iff E(1,z)\sembeds\frac{3(z+1)}{8}H_{1/3},
\]
which, by similar logic to that of \eqref{eqn:ballembeddings}, is equivalent to
\begin{equation}\label{eqn:Bwi}
\bigsqcup_{i=1}^MB(w_i)\sembeds\frac{3(z+1)}{8}H_{1/3} \iff \underbrace{B(1)\sqcup\cdots\sqcup B(1)}_{5}\sqcup B(x)\sqcup \bigsqcup_{i=7}^MB(w_i)\sembeds \frac{3(5+x+1)}{8}H_{1/3}.
\end{equation}

We can simplify \eqref{eqn:Bwi} using the fact that  $(w_7,\dots,w_m)=w_7\bw(z')$. Since $z'\geq9$, by \cite[Corollary~1.2.4]{McDuffSchlenk12}, there is a full filling $E(1,z')\sembeds B(\lambda)$, where $\lambda$ is determined by
    \[
    \vol(B(\lambda))=\vol(E(1,z')) \iff z=5+x^2+w_7^2\lambda^2.
    \]

    Thus in the right-hand embedding in \eqref{eqn:Bwi} we may replace $\bigsqcup_{i=7}^MB(w_i)$ by $B(w_7\la)$:
    \[
    \underbrace{B(1)\sqcup\cdots\sqcup B(1)}_{5}\sqcup B(x)\sqcup B(w_7\lambda)\sembeds \frac{3(5+x+1)}{8}H_{1/3}.
    \]
    Such an embedding of balls is equivalent, by the argument in the proof of \cite[Proposition~3.2]{m2}, and summarized in Section \ref{ss:reviewEC}, to showing that for every $\bE$ with $\ell(\bE)\leq7$,
    \begin{equation}\label{eqn:mubig}
    \frac{\bm\cdot(1^{\times5},x,w_7\lambda)}{d-m/3}<\frac{3(5+x+1)}{8}.
    \end{equation}
    
    Note that unless $\bE$ also satisfies the constraint $|d-3m|\leq2$, we will have the stronger upper bound by the volume by Lemma \ref{lem:bd-m}. Such $\bE$ are described in Lemma \ref{lem:E7}. The classes $(0,0;-1)$ and $(3,1;2,1^{\times5})$ may be ignored, the former because its coefficients are nonpositive and the latter because we already know its obstruction function precisely equals the line $\frac{3(z+1)}{8}$. We show that the remaining classes satisfy \eqref{eqn:mubig} case-by-case.
    \begin{itemize}
    	\item When $\bE=(1,0;1^{\times2}), (2,0;1^{\times5}), (1,1;1)$, or $(2,1;1^{\times4})$, the vector $\bm$ is short enough that
	\[
	\mu_{\bE,1/3}(z)=\frac{\bm\cdot(1^{\times5},x,w_7\lambda)}{d-m/3}
	\]
	is constant. We directly compute these four constants and check that they are less than $\frac{5+9/10+1}{8/3}$ which in turn is less than $\frac{5+x+1}{8/3}$.

	\item When $\bE=(4,1;2^{\times3}, 1^{\times4}),  (4,2;2^{\times2},1^{\times5}), (5,1;2^{\times6},1), (5,2;2^{\times5},1^{\times2})$, or $(6,2;3,2^{\times6})$, we use the fact that
	\[
	z=5+x=5+x^2+w_7^2\lambda^2 \iff w_7^2\lambda^2=x-x^2
	\]
	to rewrite $\bm\cdot(1^{\times5},x,w_7\lambda)$ solely in terms of $x$. We then check in each case that \eqref{eqn:mubig} holds as long as $x\geq9/10$, except when $\bE=(5,1;2^{\times6},1)$, in which case \eqref{eqn:mubig} holds when $x>2(5+2\sqrt{2})/17$.
    \end{itemize}
\end{proof}

 \subsection{The $z$-interval 
 $\left[\sigma^2,\frac{95+4\sqrt{2}}{17}\right]$}

We will follow the proofs in \cite[Section 4]{McDuffSchlenk12} and \cite[Section 6]{FrenkelMuller15}. In particular the numerology is very similar to that of \cite[Section 6]{FrenkelMuller15} partially due to the fact that the Pell staircase also accumulates to $\si^2.$ We make use of several of their numerical results, until the definition of the ghost stairs in Definition \ref{def:Eghostski}, which are alterations of their ghost stairs.

The strategy of the proof is as follows. We first introduce several sets of points (chosen so that any numbers between consecutive points have longer continued fractions, see the end of the proof of Proposition \ref{prop:regproof}). We prove these points are regular, Definition \ref{def:reg}, in Propositions \ref{prop:ukj}, \ref{prop:vkj}, and \ref{prop:c2k1} in subsequent sections, and use that regularity in Proposition \ref{prop:regproof} to prove Theorem \ref{thm:nodesc13}. 

Recall that
\[
\sigma^2=3+2\sqrt{2}=[5,\{1,4\}^\infty].
\]

As in \cite[Definition~6.3]{FrenkelMuller15}, for all $k,j\geq1$ we define
\begin{align*}
    c_{2k-1}&:=[5,\{1,4\}^{k-1},1]=[5,\{1,4\}^{k-2},1,5],
    \\c_{2k}&:=[5,\{1,4\}^k],
    \\u_k(j)&:=[5,\{1,4\}^{k-1},1,5,j],
    \\v_k(j)&:=[5,\{1,4\}^{k-1},1,j].
\end{align*}

\begin{lemma}\label{lem:vlim} We have
\begin{itemize}
    \item[{\rm (i)}] $c_2<c_4<\cdots<c_{2k}<\cdots<\sigma^2<\cdots<c_{2k+1}<\cdots<c_3<c_1$.
    \item[{\rm (ii)}] $c_{2k+1}<\cdots<u_k(2)<u_k(1)=v_k(6)<v_k(7)<\cdots<c_{2k-1}$.
    \item[{\rm (iii)}] $\lim_{j\to\infty}v_k(j)=c_{2k-1}$.
    \item[{\rm (iv)}] $\lim_{j\to\infty}u_k(j)=c_{2k+1}$.
\end{itemize}
\end{lemma}
\begin{proof} Parts (i) and (ii) are exactly the same as in \cite[Corollary~6.5]{FrenkelMuller15}. We prove (iii) here, and note that the proof of (iv) is similar to that of (iii).

The fact that $\lim_{j\to\infty}v_k(j)=c_{2k-1}$ follows immediately from the continued fractions: the continued fractions of $c_{2k-1}$ and $v_k(j)$ differ only by changing the final denominator of $1$ in the first expression of $c_{2k-1}$ to $1+\frac1j$. Then as $j\to\infty$ we have $\frac1j\to0$.
\end{proof}

This auxiliary result is a rewriting of \cite[Lemma~6.9]{FrenkelMuller15}:

\begin{lemma}\label{lem:bounds}
We have:
    \begin{itemize}
        \item[{\rm (i)}] $\frac{u_k(j+1)+1}{8/3}>\vol_{1/3}(u_k(j))$ for all $k,j\geq1$.
        \item[{\rm (ii)}] $\frac{v_k(j)+1}{8/3}>\vol_{1/3}(v_k(j+1))$ for all $k\geq1, j\geq6$.
    \end{itemize}
\end{lemma}
\begin{proof}
    \begin{itemize}
    \item[{\rm (i)}] This is equivalent to \cite[Lemma~6.9~(i)]{FrenkelMuller15}:
    \[
    \frac{u_k(j+1)+1}{8/3}>\sqrt{\frac{u_k(j)}{8/9}} \iff \frac{u_k(j+1)+1}{4}>\sqrt{\frac{u_k(j)}{2}}.
    \]

    \item[{\rm (ii)}] This is equivalent to  \cite[Lemma~6.9~(ii)]{FrenkelMuller15}:
    \[
    \frac{v_k(j)+1}{8/3}>\sqrt{\frac{v_k(j+1)}{8/9}} \iff \frac{v_k(j)+1}{4}>\sqrt{\frac{v_k(j+1)}{2}}.
    \]
    \end{itemize}
\end{proof}

\begin{definition}[{\cite[Definition~6.10]{FrenkelMuller15}}]\label{def:reg}
    A point $z\in[\sigma^2,6]$ is \textbf{regular} if for all exceptional classes $\bE=(d,m;\bm)$ with $\ell(\bm)=\ell(z)$, we have
    \[
    \mu_{\bE,1/3}(z)\leq\frac{3(z+1)}{8}.
    \]
\end{definition}

\begin{rmk}\label{rmk:sig6reg} \rm
Recall from \eqref{eqn:cbsup} that $c_{1/3}(z)= \sup_\bE \{\mu_{\bE,1/3}(z),\vol_{1/3}(z)\}$. Since $c_{1/3}$ has an infinite staircase accumulating at $\sigma^2$, we have by \cite[Theorem~1.13]{AADT} that $c_{1/3}(\sigma^2)=\sqrt{\frac{\sigma^2}{8/9}}=\frac{3(\sigma^2+1)}{8}$. In particular, this implies that $z=\sigma^2$ is regular. Similarly, the fact that $c_{1/3}(6)=\frac{3(6+1)}{8}$ by  Lemma \ref{lem:6reg} implies that $z=6$ is regular. 
\end{rmk}

We will see in Sections \ref{ss:reguv}-\ref{ss:regc} that the points $c_{2k-1}, u_k(j)$, and $v_k(j)$ are regular. With their regularity in place, we can prove the following upper bound, which along with the lower bound from Lemma \ref{lem:E02} completes the proof of Theorem \ref{thm:13proved}.

\begin{prop}\label{prop:regproof}
    On the interval $[\sigma^2,6]$ we have:
    \[
    c_{1/3}(z)\leq\frac{3(z+1)}{8}.
    \]

\end{prop}

\begin{proof} We follow the proofs of \cite[Proposition~4.1.6]{McDuffSchlenk12}, \cite[Proposition~6.11]{FrenkelMuller15} quite closely, making several necessary changes, particularly to the references.

The points
    \begin{align*}
        c_{2k-1}&, k\geq1,
        \\u_k(j)&, k\geq1, j\geq2,
        \\v_k(j)&, k\geq1, j\geq6,
    \end{align*}
    are regular by Propositions \ref{prop:ukj}, \ref{prop:vkj}, \ref{prop:c2k1}, and Corollary \ref{cor:j6reg}.

By Lemma \ref{lem:E02} and \cite[Proposition~2.1~(5)]{AADT}, the function $c_{1/3}$ is piecewise linear on $(\sigma^2,6)$. Let $S\subset(\sigma^2,6)$ denote the set of $z$-values of nonsmooth points of $c_{1/3}$, with $S=S_+\cup S_-$, where $S_+$ denotes the ``outer corners'' (where $c_{1/3}$ is concave) and $S_-$ denotes the ``inner corners'' (where $c_{1/3}$ is convex).

Assume by contradiction that there is some $z_0\in(\sigma^2,6)$ for which $c_{1/3}(z_0)>\frac{3(z+1)}{8}$. Then $S$ is nonempty. Moreover, by the fact that $c_{1/3}(\sigma^2)=\frac{3(\sigma^2+1)}{8}$, the set $S_+$ is nonempty. Finally, by Lemma \ref{lem:6reg}, the largest $z$-value in $S$ is in $S_-$ and it is at most $\frac{95+4\sqrt{2}}{17}\approx5.921$.

Denote by $s_0\in\left(\sigma^2,\frac{95+4\sqrt{2}}{17}\right]$ the largest $z$-value in $S_+$, which exists because there are only finitely many nonsmooth points outside any given neighborhood of the accumulation point $\sigma^2$. Given any $\alpha>1$, if an exceptional class satisfies $\mu_{\bE,1/3}(z)>\alpha\vol_{1/3}(z)$, then the volume bound in Lemma \ref{lem:bd-m} can be reinterpreted as an upper bound for the $d$ value of that class (or equivalently an upper bound for $m$), using the fact that $|d-3m|\leq2$ by the first part of Lemma \ref{lem:bd-m}. Given $z\in[s_0,s_0+\epsilon)$, set $\alpha(z)=3(z+1)/8\vol_{1/3}(z)$, which is uniformly bounded away above one on $[s_0,s_0+\epsilon)$. Because the upper bound on $d$ is continuous, if $\epsilon$ is small enough then there is a uniform bound on $d$ for all $z\in[s_0,s_0+\epsilon)$, and thus there are only finitely many possible $\bE$ which can have $\mu_{\bE,1/3}(z)>\alpha(z)\vol_{1/3}(z)$. Therefore $c_{1/3}(z)$ is a true maximum (rather than a supremum), and there is an exceptional class $\bE=(d,m;\bm)$ for which $c_{1/3}(z)=\mu_{\bE,1/3}(z)$ on $[s_0,s_0+\epsilon)$ for sufficiently small $\epsilon > 0$.

Choose $\epsilon$ small enough that $[s_0,s_0+\epsilon)\cap S_-=\emptyset$. Let $I\supset[s_0,s_0+\epsilon)$ be the maximal open interval on which $\mu_{\bE,1/3}(z)>\vol_{1/3}(z)$; by Lemma \ref{lem:cghmp23} there is a unique point $a\in I$ that is a nonsmooth point of $\mu_{\bE,1/3}$, and $\ell(a)=\ell(\bm)$. Furthermore, by Lemma \ref{lem:cghmp21}, for all other $z\in I$ we have $\ell(z)>\ell(\bm)$. Meanwhile, by Lemma \ref{lem:cghmp23}, the obstruction $\mu_{\bE,1/3}$ is piecewise linear on $I$, so if $s_0$ is not a nonsmooth point of $\mu_{\bE,1/3}$ then for $z<s_0$ and close to $s_0$, the value of $\mu_{\bE,1/3}$ at $z$ lies on the continuation of the line $\mu_{\bE,1/3}|_{[s_0,s_0+\epsilon)}=c_{1/3}|_{[s_0,s_0+\epsilon)}$. This is a contradiction to the fact that $s_0$ is a concave nonsmooth point of $c_{1/3}$ (that is, $s_0\in S_+$), which forces the value of $c_{1/3}(z)$ (and hence that of $\mu_{\bE,1/3}$) to be below this line. Since $a$ is the unique nonsmooth point of $\mu_{\bE,1/3}$ in $I\supset[s_0,s_0+\epsilon)$, we must have $a=s_0$.

By Lemma \ref{lem:cghmp23}, on $I$ we have
\[
\mu_{\bE,1/3}(z)=\begin{cases}\alpha+\beta z&\text{ if }z\leq s_0
\\\alpha'+\beta'z&\text{ if }z\geq s_0\end{cases}.
\]
Since $c_{1/3}$ is nondecreasing and $\mu_{\bE,1/3}(z)=c_{1/3}(z)$ on part of the interval where $\mu_{\bE,1/3}(z)=\alpha'+\beta'z$, we must have $\beta'\geq0$.

Let $k$ be the index for which $s_0\in(c_{2k+1},c_{2k-1})$ ($s_0$ cannot be either $c_{2k\pm1}$ as they are regular, and $\ell(\bm)=\ell(s_0)$). Further let $s_0\in(u_-,u_+)$, where the $u_\pm$ are consecutive $u_k(j)$ or $v_k(j)$ terms in the sequence in Lemma \ref{lem:vlim} (ii); again $s_0\neq u_\pm$ by regularity.

\begin{figure}[ht!]
    \centering
    \includegraphics[width=0.6\textwidth]{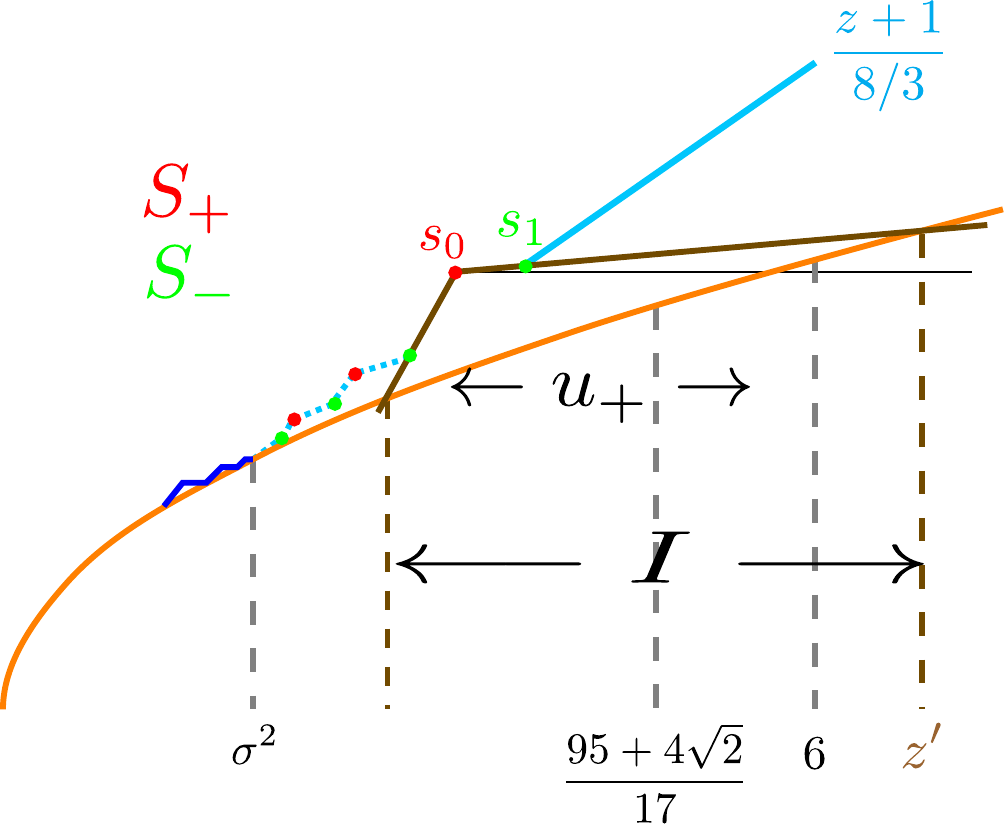}
    \caption{Here we depict the (contradictory) situation described in the proof of Proposition \ref{prop:regproof}. The orange curve is the volume obstruction, the dark blue staircase is the graph of the function $c_{1/3}$ below $\sigma^2$, and the dotted bright blue curve is the hypothetical graph of $c_{1/3}$ between $\sigma^2$ and the interval on which the obstruction from $\bE$ is live. The graph of $\mu_{\bE,1/3}$ on $I$ is in brown. The point $u_+$ must be between the $z$-values of the two vertical pink lines.}
    \label{fig:s0u+}
\end{figure}

By Lemma \ref{lem:bounds},
\begin{equation}\label{eqn:s0u+}
\mu_{\bE,1/3}(s_0)>\frac{3(s_0+1)}{8}>\frac{3(u_-+1)}{8}>\vol_{1/3}(u_+).
\end{equation}
We refer the reader to Figure \ref{fig:s0u+}.

Let $s_1$ be the largest $z$-value in $S$; recall that it is in $S_-$. Since there are no nonsmooth points of $c_{1/3}$ between $s_0$ and $s_1$, on $[s_0,s_1]$ the function $c_{1/3}$ must be the extension of the line $c_{1/3}|_{[s_0,s_0+\epsilon)}$, thus
\begin{equation}\label{eqn:s1}
\vol_{1/3}(s_1)<\frac{3(s_1+1)}{8}=c_{1/3}(s_1)=\alpha'+\beta's_1.
\end{equation}
We know that $\mu_{\bE,1/3}(z)=\alpha'+\beta'z$ for all $z \in [s_0,z']$, where $z'$ is the right endpoint of the interval $I$, defined by:
\[
\alpha'+\beta'z'=\vol_{1/3}(z').
\]
The inequality \eqref{eqn:s1} shows that $s_1<z'$. Therefore, because there are no nonsmooth points of $\mu_{\bE,1/3}$ between $s_0$ and $z'$, we have $\mu_{\bE,1/3}(s_1)=\alpha'+\beta's_1>\vol_{1/3}(s_1)$, implying that $[s_0,s_1]\subset I$.

Further, by \eqref{eqn:s0u+} and the fact that $\beta'\geq0$, we must have $z'\geq u_+$, therefore
\[
\mu_{\bE,1/3}(u_+)>\vol_{1/3}(u_+),
\]
meaning $u_+\in I$.

Thus $\ell(u_+)>\ell(s_0)$ by Lemma \ref{lem:cghmp23}. However, this contradicts the fact that because $s_0\in(u_-,u_+)$, we must have $\ell(s_0)>\ell(u_\pm)$. This last fact can be explained as follows: if 
$$[a_0,\dots,a_{n-1},a_n]<x<[a_0,\dots,a_{n-1},a_n'],$$
then we may write (by abuse of notation) $x=[a_0,x']$, where $x'$ may be a real number. Then, because $a_0<x<a_0+1$,
\[
\frac{1}{a_1+\frac{1}{\ddots+\frac{1}{a_n}}}<x-a_0<\frac{1}{a_1+\frac{1}{\ddots+\frac{1}{a_n'}}} \Rightarrow a_1+\frac{1}{\ddots+\frac{1}{a_n}}>x'>a_1+\frac{1}{\ddots+\frac{1}{a_n'}}.
\]
Repeating this argument shows that because $u_\pm$ agree up to their last digit, the continued fraction of $s_0$ must also agree for those first digits. In the case $\{u_-,u_+\}=\{u_k(j+1),u_k(j)\}$, we thus have
\[
s_0=[5,\{1,4\}^{k-1},1,5,s'],
\]
and because the $s'$ is in an odd place,
\begin{equation}\label{eqn:js'}
u_k(j+1)<s_0<u_k(j) \Rightarrow j+1>s'>j \Rightarrow s'=[j,s'']
\end{equation}
for some real number $s''$. Thus $\ell(s_0)=\ell(u_+)+\ell(s'')$ and in particular $\ell(s_0)>\ell(u_+)$. The argument in the case $\{u_-,u_+\}=\{v_k(j),v_k(j+1)\}$ is similar, except that here we have
\[
s_0=[5,\{1,4\}^{k-1},1,s']
\]
with the $s'$ in an even place, which produces the same inequalities as in \eqref{eqn:js'}.
\end{proof}

\subsection{Regularity of the $u_k(j)$ and $v_k(j)$}\label{ss:reguv}

We now prove that the $u_k(j)$ and $v_k(j)$ are regular, which are hypotheses of Proposition \ref{prop:regproof}. We first prove key estimates in Lemma \ref{lem:3d-m}.

\begin{rmk} \rm
Note that by Lemma \ref{lem:bd-m}, if $\bE=(d,m;\bm)$ is obstructive then
\[
|d-3m|<\sqrt{8} \iff -2\leq d-3m\leq2.
\]
Therefore when proving the hypotheses of Proposition \ref{prop:regproof} in Propositions \ref{prop:ukj},  \ref{prop:vkj}, and \ref{prop:c2k1}, we may restrict to $\bE$ with $-2\leq d-3m\leq2$.
\end{rmk}

Recall that the \textbf{error vector} $\varepsilon$ of a Diophantine class $\bE=(d,m;\bm)$ at a point $z$ in an interval on which $\mu_{\bE,1/3}$ is obstructive is defined by
\[
\bm=\alpha(z)\bw(z)+\varepsilon,
\]
where by \cite[Lemma~18~(i)]{ICERM},
\begin{itemize}
    \item the obstruction from the $\alpha(z)\bw(z)$ part of $\bm$ is the volume constraint at $z$:
    \[
    \alpha(z)=\frac{3d-m}{\sqrt{8z}},
    \]
    \item $\ell(\varepsilon)=\ell(z)$,
    \item and $||\varepsilon||^2=\varepsilon\cdot\varepsilon<1$.
\end{itemize}

\begin{lemma}\label{lem:3d-m} Assume $\bE=(d,m;\bm)$ satisfies
\[
\mu_{\bE,1/3}(z)>\frac{(z+1)}{8}
\]
for some $z\in[\sigma^2,6]$. Then if $d=3m+c$, the following inequalities hold:
\begin{itemize}
    \item[{\rm (i)}]
    \[
    3d-m<\sqrt{\frac{64z}{z^2-6z+1}+c^2}.
    \]
    \item[{\rm (ii)}]
    \[
    -\varepsilon\cdot\varepsilon>\frac{(3d-m)^2}{8\sqrt{2z}}(z+1-2\sqrt{2z})-1+\frac{c^2}{8}.
    \]
\end{itemize}
\end{lemma}
\begin{proof}
By Lemma \ref{lem:bd-m},
\[
\mu_{\bE,1/3}(z)\leq\sqrt{\frac{z}{8/9}}\sqrt{1+\frac{1}{d^2-m^2}},
\]
thus
\begin{align*}
    &\frac{3(z+1)}{8}<\sqrt{\frac{z}{8/9}}\sqrt{1+\frac{1}{d^2-m^2}}
    \\\iff& d^2-m^2<\frac{8z}{z^2-6z+1}.
\end{align*}

We now assume $d=3m+c$. Then
\begin{align*}
    &8m^2+6cm+c^2<\frac{8z}{z^2-6z+1}
    \\\iff& (3d-m)^2-c^2<\frac{64z}{z^2-6z+1}.
\end{align*}

To prove conclusion (ii), from the quadratic Diophantine equation in \eqref{eqn:D}, we have:
\begin{align}
    d^2-m^2+1&=\alpha(z)^2\bw(z)\cdot\bw(z)+2\alpha(z)\bw(z)\cdot\varepsilon+\varepsilon\cdot\varepsilon\nonumber
    \\&=\frac{(3d-m)^2}{8}+\frac{2(3d-m)}{\sqrt{8z}}\bw(z)\cdot\varepsilon+\varepsilon\cdot\varepsilon\nonumber.
\end{align}
Writing $d=3m+c$ and simplifying we get:
\begin{equation}\label{eqn:qD}
\bw(z)\cdot\varepsilon=\frac{\sqrt{2z}(1-c^2/8-\varepsilon\cdot\varepsilon)}{3d-m}.
\end{equation}
Meanwhile, from our assumption that $\mu_{\bE,1/3}(z)>3(z+1)/8$, we have:
$$\frac{3(z+1)}{8}<\frac{(\alpha(z)\bw(z)+\varepsilon)\cdot\bw(z)}{d-m/3}=\frac{1}{d-m/3}\left(\frac{(3d-m)\sqrt{z}}{\sqrt{8}}+\bw(z)\cdot\varepsilon\right)$$
\begin{equation}\label{eqn:webound} \iff(3d-m)\left(\frac{z+1}{8}-\sqrt{\frac{z}{8}}\right)<\bw(z)\cdot\varepsilon.\end{equation}
Finally, combining (\ref{eqn:qD}) and (\ref{eqn:webound}), we obtain as desired:
$$\frac{\sqrt{2z}(1-c^2/8-\varepsilon\cdot\varepsilon)}{3d-m}>(3d-m)\left(\frac{z+1}{8}-\sqrt{\frac{z}{8}}\right)$$

$$\iff-\varepsilon\cdot\varepsilon>\frac{(3d-m)^2}{8\sqrt{2z}}(z+1-2\sqrt{2z})-1+\frac{c^2}{8}.
$$
\end{proof}

\begin{prop}\label{prop:ukj}
    The points $u_k(j)$ for $ k\geq1, j\geq2$ are regular.
\end{prop}
\begin{proof} We will proceed by contradiction. 
    Assume $\bE=(d,m;\bm)$ is an exceptional class with
    \[
    \mu_{\bE,1/3}(u_k(j))>\frac{3(u_k(j)+1)}{8}
    \]
    and $\ell(\bm)=\ell(u_k(j))=5(k+1)+1+j$. Set $u_k(j)=p_k(j)/q_k(j)$. 
    
    The entries of the vector $\varepsilon$ are the differences between the corresponding entries in $\bm$ and $\alpha(u_k(j))\bw(u_k(j))$. The last $j$ entries in $\bm$ are at least one, while the last $j$ entries in $\alpha(u_k(j))\bw(u_k(j))$ all equal $\alpha(u_k(j))/q_k(j)=(3d-m)/(q_k(j)\sqrt{8u_k(j)})$, because the last element in the continued fraction of $u_k(j)$ is $j$.  Thus
    \begin{align}
        \varepsilon\cdot\varepsilon&\geq j\left(1-\frac{3d-m}{q_k(j)\sqrt{8u_k(j)}}\right)^2\label{eq:boundee}
        \\&> j\left(1-\sqrt{\frac{8}{q_k(j)^2(u_k(j)^2-6u_k(j)+1)}+\frac{c^2}{8q_k(j)^2u_k(j)}}\right)^2 \text{ by Lemma \ref{lem:3d-m} (i) and \textbf{Claim}}\nonumber
        \\&= j\left(1-\sqrt{\frac{8}{j^2+6j+1}+\frac{c^2}{8p_k(j)q_k(j)}}\right)^2=:j\left(1-\sqrt{f(j)}\right)^2 \text{ by \cite[Corollary~6.8~(i)]{FrenkelMuller15}}.\nonumber
    \end{align}

\noindent\textbf{Claim:} We claim that $f(j)$ is decreasing and less than one when $j\geq2$.  This allows us to invoke Lemma~\ref{lem:3d-m}~(i) above and also shows that 
$$F(j):=j\left(1-\sqrt{f(j)}\right)^2$$ 
is increasing for $j\geq2$. First note that if $f$ is decreasing, to show that it is less than one it suffices to show $f(2)<1$. To do so, note that $|c|\leq2$ and $p_k(j), q_k(j)\geq1$, thus the term $c^2/(8p_k(j)q_k(j))\leq1/2$, and furthermore $8/(2^2+6\cdot2+1)=8/17<1/2$.

To show that $f(j)$ is decreasing when $j\geq2$, it suffices to show that $p_k(j)$ and $q_k(j)$ are both nondecreasing in $j$. For this, we use the fact that if $p_i/q_i=[a_0,\dots,a_i]$ denotes the $i^\text{th}$ convergent of a continued fraction $[a_0,\dots,a_i,\dots]$ then
\[
x_i=a_ix_{i-1}+x_{i-2} \mbox{ for } x=p,q.
\]
Consider the final two convergents of $u_k(j)=[5,\{1,4\}^{k-1},1,5,j]$, which are $p_k/q_k=[5,\{1,4\}^{k-1},1,5]$ and $p'_k/q'_k=[5,\{1,4\}^{k-1},1]$. Then $x_k(j)=jx_k+x'_k$ for $x=p,q$, which are evidently increasing in $j$. We have shown the claim.

Moreover, the function $F(j)$ is at least one when $j\geq4$. This is because $p_k,p'_k\geq5$ and $q_k,q'_k\geq1$,\footnote{One might wonder whether or not we could decrease the bound on $j$ by using better estimates for $p_k(j)$ and $q_k(j)$. However, this is not possible, because even if we assume they are infinitely large so that the $c^2/8p_k(j)q_k(j)$ term does not contribute, we can only show $F(3)>3(1-\sqrt{8/(3^2+6*3+1)})^2\approx0.65$.} thus $p_k(4)q_k(4)\geq125$. The fact that $|c|\leq2$ is then enough to show $F(4)>1$.

Therefore, for $j\geq4$, we have shown that $\varepsilon\cdot\varepsilon>F(j)>1$. This contradicts the fact that $\varepsilon\cdot\varepsilon<1$ by \cite[Lemma~18~(i)]{ICERM}.

To handle the cases $j=2, 3$, we use the bound on $\varepsilon\cdot\varepsilon$ from \eqref{eq:boundee} and the bound on $-\varepsilon\cdot\varepsilon$ from Lemma \ref{lem:3d-m} (ii) to get:
\begin{align*}
0&=\varepsilon\cdot\varepsilon-\varepsilon\cdot\varepsilon
\\&>j\left(1-\frac{3d-m}{q_k(j)\sqrt{8u_k(j)}}\right)^2+\frac{(3d-m)^2}{8\sqrt{2u_k(j)}}(u_k(j)+1-2\sqrt{2u_k(j)})-1+\frac{c^2}{8} 
\\&=j\left(1-\frac{8m+3c}{\sqrt{8p_k(j)q_k(j)}}\right)^2+\frac{(8m+3c)^2\sqrt{q_k(j)}}{8\sqrt{2p_k(j)}}\left(\frac{p_k(j)}{q_k(j)}+1-\frac{2\sqrt{2p_k(j)}}{\sqrt{q_k(j)}}\right)-1+\frac{c^2}{8}=:g(m).
\end{align*}
Rearranging, we obtain $g(m)=4Am^2+Bm+C$ where:
\begin{align*}
A=&\frac{\sqrt{2p_k(j)}}{\sqrt{q_k(j)}}+\frac{\sqrt{2q_k(j)}}{\sqrt{p_k(j)}}+\frac{2j}{p_k(j)q_k(j)}-4,
\\B=&-\frac{4\sqrt{2}j}{\sqrt{p_k(j)q_k(j)}}+3c A,
\\C=& j-1+\frac{9c^2}{16}A-\frac{3cj}{\sqrt{2p_k(j)q_k(j)}}+\frac{c^2}{8}.
\end{align*}

To show $g(m)>0$ and obtain a contradiction, we check that $A$ is positive and the discriminant of $g(m)$ is negative, so that $g(m)$ is concave up and does not have zeroes. Checking the first amounts to simplifying the following:
\begin{align}
&A=\frac{\sqrt{2p_k(j)}}{\sqrt{q_k(j)}}+\frac{\sqrt{2q_k(j)}}{\sqrt{p_k(j)}}+\frac{2j}{p_k(j)q_k(j)}-4>0\nonumber
\\
\iff&\sqrt{2p_k(j)q_k(j)}\left(p_k(j)+q_k(j)\right)>4p_k(j)q_k(j)-2j\nonumber
\\
\iff&p_k(j)q_k(j)\left(p_k(j)^2-6p_k(j)q_k(j)+q_k(j)^2+8j\right)>2j^2.\label{eqn:m2>0}
\end{align}
We now use \cite[Corollary~6.8~(iii)]{FrenkelMuller15}, which states
\[
p_k(j)^2-6p_k(j)q_k(j)+q_k(j)^2=\begin{cases}17&\text{ if }j=2
\\28&\text{ if }j=3
    \end{cases}.
\]
Moreover, we have $p_k(j)q_k(j)\geq45$ (using the fact discussed above that $p_i=a_ip_{i-1}+p_{i-2}$, and the same for $q_i$) when $j\geq2$. This finishes the proof that $A$ is positive.

Showing that the discriminant $B^2-16A C$ is negative amounts to showing that
\begin{equation}\label{eqn:D<0c}16j^2<(8(j-1)+c^2)\left(\sqrt{2p_k(j)q_k(j)}(p_k(j)+q_k(j))+2j-4p_k(j)q_k(j)\right).
\end{equation}
Note that the inequality \eqref{eqn:D<0c} is only strengthened if $c\neq0$, so we assume $c=0$ and divide by eight, meaning that we must show
\begin{align}
&2j^2<(j-1)\left(\sqrt{2p_k(j)q_k(j)}(p_k(j)+q_k(j))+2j-4p_k(j)q_k(j)\right)\nonumber
\\\iff&2j+4(j-1)p_k(j)q_k(j)<(j-1)\sqrt{2p_k(j)q_k(j)}(p_k(j)+q_k(j))\nonumber
\\\iff&2j^2+8(j-1)p_k(j)q_k(j)<(j-1)^2p_k(j)q_k(j)\left(p_k(j)^2-6p_k(j)q_k(j)+q_k(j)^2-8\right).\label{eqn:D<0}
\end{align}

Again using \cite[Corollary~6.8~(iii)]{FrenkelMuller15}, we must show
\[
8+8p_k(2)q_k(2)<p_k(2)q_k(2)(17-8) \mbox{ and } 18+16p_k(3)q_k(3)<4p_k(3)q_k(3)(28-8),
\]
both of which hold by the fact that $p_k(j)q_k(j)\geq45$.
\end{proof}

The next result and its proof are very similar to the one above.

\begin{prop}\label{prop:vkj}
    The points $v_k(j)$ for $k\geq1, j\geq7$ are regular.
\end{prop}
\begin{proof}
    We will proceed by contradiction. Assume $\bE=(d,m;\bm)$ is an exceptional class with
    \[
    \mu_{\bE,1/3}(v_k(j))>\frac{3(v_k(j)+1)}{8}
    \]
    and $\ell(\bm)=\ell(v_k(j))=5k+1+j$. Set $v_k(j)=p_k(j)/q_k(j)$.

    As in the proof of Proposition \ref{prop:ukj}, we use the fact that the last $j$ entries in $\bm$ are at least one, while the last $j$ entries in $\alpha(v_k(j))\bw(v_k(j))$ all equal $\alpha(v_k(j))/q_k(j)=(3d-m)/q_k(j)\sqrt{8v_k(j)}$. Thus
    \begin{align*}
        \varepsilon\cdot\varepsilon&\geq j\left(1-\frac{3d-m}{q_k(j)\sqrt{8v_k(j)}}\right)^2
        \\&> j\left(1-\sqrt{\frac{8}{q_k(j)^2(v_k(j)^2-6v_k(j)+1)}+\frac{c^2}{8q_k(j)^2v_k(j)}}\right)^2\text{ by Lemma \ref{lem:3d-m} (i) and \textbf{Claim}}
        \\&= j\left(1-\sqrt{\frac{8}{j^2-4j-4}+\frac{c^2}{8p_k(j)q_k(j)}}\right)^2=:F(j)\text{ by \cite[Corollary~6.8~(ii)]{FrenkelMuller15}}.
    \end{align*}
   \noindent\textbf{Claim:} The fact that $F(j)$ is increasing and at least one when $j\geq8$ can be shown similarly to the analogous fact in the proof of Proposition \ref{prop:ukj}, contradicting the fact that $\varepsilon\cdot\varepsilon<1$ by \cite[Lemma~18~(i)]{ICERM}.

    To handle the case $j=7$, we follow the method laid out in the proof of Proposition \ref{prop:ukj} for $j=2, 3$. We obtain the same polynomial $g(m)$, and the proofs that the $m^2$ coefficients $4A$ is positive and discriminant is negative are exactly the same, until instead of \cite[Corollary~6.8~(iii)]{FrenkelMuller15} we use \cite[Corollary~6.8~(iv)]{FrenkelMuller15}:
\[
p_k(j)^2-6p_k(j)q_k(j)+q_k(j)^2=\begin{cases}8&\text{ if }j=6
\\17&\text{ if }j=7
    \end{cases}.
\]
Thus to show the analogue of (\ref{eqn:m2>0}) we need to show
\[
p_k(6)q_k(6)(8+48)>72 \mbox{ and } p_k(7)q_k(7)(17+56)>98,
\]
both of which are true because
\begin{equation}\label{eqn:pkj245}
p_k(j)q_k(j)\geq p_k(6)q_k(6)\geq(6\cdot5+5)(6\cdot1+1)=245.
\end{equation}
Note that here we used as in the proof of Proposition \ref{prop:ukj} the fact that $x_k(j)=jx_k+x'_k$ for $x=p, q$, where $p_k/q_k=[5,\{1,4\}^{k-1},1]$ and $p'_k/q'_k=[5,\{1,4\}^{k-1}]$.

To show the analogue of (\ref{eqn:D<0}), which implies that the discriminant of $g(m)$ is negative, we show
\[
98+48p_k(7)q_k(7)<49p_k(7)q_k(7)(17-8),
\]
which holds by (\ref{eqn:pkj245}).

\end{proof}

\subsection{The ghost stairs and regularity of the $c_{2k-1}$}\label{ss:regc}

We define a family of non-perfect classes $\bE_k(i)$ for $k\geq1, i\geq1$ in analogy to the classes $E(b_k(i))$ of \cite[Section 4.2]{McDuffSchlenk12} and \cite[Section 6]{FrenkelMuller15}. We will then show in Lemmas \ref{lem:Ekidiophantine} and \ref{lem:Ekiexceptional} that these classes are Diophantine and exceptional, respectively. Then in Lemma \ref{lem:ip} we will use intersection positivity with these $\bE_k(i)$ to show that the obstruction from any other exceptional class $\bE$ at $b_k(i)$ must be at most $3(b_k(i)+1)/8$, which will allow us to conclude in Proposition \ref{prop:c2k1} that the $c_{2k-1}$ are regular (the last remaining unproved hypothesis of Proposition \ref{prop:regproof}.

\begin{definition}\label{def:Eghostski} For $k\geq1$ and $i\equiv2\, (\text{mod }4)$, let
\[
\frac{p_k(i)}{q_k(i)}:=b_k(i):=\left[5,\{1,4\}^{k-1},1,2i+2\right]=v_k(2i+2).
\]
We define $\bE_k(i):=(d_k(i),m_k(i);\bm_k(i))$ as follows. Set $i=2j$, with $j$ odd, $j\geq3$. The vector $\bm_k(i)$ is $q_k(i)\bw(b_k(i))$ with the final $(1^{\times4j+2})$ replaced by
\[
\left(j+1, \left(\frac{j+1}{2}\right)^{\times3}, \frac{j-1}{2},1^{\times j}\right),
\]
and
    \[
    m_k(i):=\frac{p_k(i)+q_k(i)}{8}, \quad d_k(i):=3m_k(i).
    \]
When $j=1$ the definition is the same, but because $j-1=0$ we write $\bm_k(i)$ in decreasing order, so that we append $(2,1^{\times4})$ instead of $(2,1^{\times3},0,1)$.
\end{definition}
Note that only when $i\equiv2\, (\text{mod }4)$ is $(p_k(i)+q_k(i))$ divisible by eight, implying $d_k(i)$ and $m_k(i)$ are integers. This can be shown as follows: first, by a direct computation for $k=1,2$ and all $i$. Second, we show that $p_k(i)=6p_{k-1}(i)-p_{k-2}(i)$ by using the recursive formula for numerators and denominators of convergents of continued fractions, \eqref{eqn:piai}, to write $p_{k}(i), p_{k-1}(i)$, and $p_{k-2}(i)$ in terms of the numerators of $[5,\{1,4\}^{k-4},1,4]$ and $[5,\{1,4\}^{k-4},1,4,1]$, which are convergents of all three. The same argument works for the $q_k(i)$, which shows divisibility by eight for $k\geq3$.

\begin{rmk}\label{rmk:ekifacts} \rm

The following remarks will not be used in the proof of Proposition \ref{prop:c2k1}, but are a collection of properties of the $\bE_k(i)$ to put them in context with other embedding targets and obstructions.

    \begin{itemize}
    \item[{\rm (i)}] While we will be using $\bE_k(i)$ to compute $c_{1/3}(b_k(i))$ (see Lemma \ref{lem:ip}), the ``center''  (the obstruction from $\bE_k(i)$ in this case is below the $\vol_b$ curve) of $\bE_k(i)$ with $i=2j$ is not $b_k(i)$ but
    \[
    [5,\{1,4\}^{k-1},1,j+5]
    \]
    when $j\geq3$ and
    \[
    [5,\{1,4\}^{k-1},1,5]=y_{k+2}/y_{k+1}
    \]
    when $j=1$, where the $y_k$ are defined in \eqref{eqn:yk}. (However, if we do not reorder $\bm_k(2)$ to put the zero at the end, the obstruction $\mu$ may still be plotted, and its center is also $[5,\{1,4\}^{k-1},1,j+5]$.)
    
    We do not prove this here. A similar property holds in \cite{McDuffSchlenk12} and \cite{FrenkelMuller15}.
    
    \item[{\rm (ii)}] The class
    \[
    \bE_0(2):=(3,1;2,1^{\times5})
    \]
    is a natural extension to $k=0$ of the $i=2$ family in the sense that those classes all break at $y_{k+2}/y_{k+1}$ (see \cite[Lemma~2.1.1~(iii)]{MM}). However, we elect to keep the indexing comparable to \cite{McDuffSchlenk12} and \cite{FrenkelMuller15}, where the classes
    \[
    \bE_{MS}:=(3;2,1^{\times6}) \quad \text{ and } \quad \bE_{FM}:=(2,2;2,1^{\times5}),
    \]
    which play roles analogous to that of $\bE_0(2)$ (in that they form the non-perfect obstruction equal to the embedding function immediately after the accumulation point), do not appear in the families $E(b_k(i))$.

    Moreover, there is simply no number $[5,\{1,4\}^{-1},1,6]$!

    \item[{\rm (iii)}]
    The classes $\bE_k(2)$ are live, and thus by \cite{concaveconvex} on some interval they must equal the obstruction from some ratio of ECH capacities of $E(1,z)$ and $H_{1/3}$. Specifically, we claim without proof that $m_k(2)=y_{k+1}$, therefore
\begin{align*}
d_k(2)(d_k(2)+3)-m_k(2)(m_k(2)+1)&=3y_{k+1}(3y_{k+1}+3)-y_{k+1}(y_{k+1}+1)
\\&=9y_{k+1}^2+9y_{k+1}-y_{k+1}^2-y_{k+1}
\\&=2(8T_{y_{k+1}}),
\end{align*}
where $T_n=n(n+1)/2$ is the $n^\text{th}$ triangular number. A modification of the arguments in \cite[Lemma~92, Theorem~94]{ICERM} shows that this is a necessary condition for $\mu_{\bE_k(2),1/3}$ near $y_{k+2}/y_{k+1}$ to equal the obstruction from the ratios of the $8T_{y_{k+1}}^\text{st}$ ECH capacities of $E(1,z)$ and $H_{1/3}$ for $z\approx y_{k+2}/y_{k+1}$.

Not only are the classes $\bE_k(i)$ for $i>2$ not live, they are not even obstructive. It is for this reason that their analogues are called ``ghosts'' in \cite{McDuffSchlenk12} and \cite{cghm}. Thus there is no reason to expect their obstruction equals an obstruction from an ECH capacity. The obstruction from the $K^\text{th}$ ECH capacity, where 
\[
2K=d_k(i)(d_k(i)+3)-m_k(i)(m_k(i)+1),
\]
is the most likely candidate, but this obstruction does not equal $\mu_{\bE_k(i),1/3}$ near its center $[5,\{1,4\}^{k-1},1,j+5]$. Computer testing indicates it is unlikely that there is any ratio of ECH capacities whose obstruction equals that from $\bE_k(i)$ near its center.
    \end{itemize}
\end{rmk}

\begin{lemma}\label{lem:Ekidiophantine}
    The classes $\bE_k(i)$ are Diophantine, that is, they satisfy the Diophantine equations in \eqref{eqn:D}. 
\end{lemma}
\begin{proof}
We note that our classes $\bE_k(i)$ are very similar to classes $E(b_k(i))$ defined in \cite[Definition~6.14]{FrenkelMuller15}. This proof will use their \cite[Lemma~6.15]{FrenkelMuller15} to simplify our goal of checking the Diophantine equations \eqref{eqn:D}. Note that the notation $E(b_k(i))$ is inconsistent with our notation used for exceptional classes, and we only use it to refer to the classes from \cite{FrenkelMuller15}. However, our $b_k(i)$ are the same as theirs.

Their classes are of the form $E(b_k(i)):=(e_k(i),f_k(i);\mathbf{n}_k(i))$ where 
$$e_k(i)=f_k(i)=\frac{p_k(i)+q_k(i)}{4}$$
and $\mathbf{n}_k(i)=(n_1,\dots,n_L)$ equals $q_k(i)\bw(b_k(i))$ with the last block $\left(1^{\times(4j+2)}\right)$ replaced by the block $\left(j+1,j,1^{\times(2j+1)}\right)$. The conclusion of \cite[Lemma~6.15]{FrenkelMuller15} is that
\begin{equation}\label{eqn:pD}
\mathbf{n}_k(i)\cdot\mathbf{n}_k(i)=2e_k(i)f_k(i)+1, \quad \sum_ln_l=2(e_k(i)+f_k(i))-1.
\end{equation}
We obtain the quadratic Diophantine equation by first showing that  the sum of squares of the entries of $\bm_k(i)$ and $\mathbf{n}_k(i)$ are equal. It suffices to check the entries which replaced the final $(1^{\times 4j+2})$ from $\bw(b_k(2j))$ in either case. This amounts to checking that the following equality holds:
$$ (j+1)^2+3\left(\frac{j+1}{2}\right)^2+\left(\frac{j-1}{2}\right)^2+j=(j+1)^2+j^2+2j+1.$$ We then note that
 \[
    d_k(i)^2-m_k(i)^2=\frac{(p_k(i)+q_k(i))^2}{8}=2e_k(i)f_k(i).
    \]
Our quadratic Diophantine equation will then follow from the quadratic equation in \eqref{eqn:pD}. 

Similarly, we obtain the linear Diophantine equation by first showing that the sum of the entries in $\bm_k(i)$ and $\mathbf{n}_k(i)$ are equal. We show that the sum of the same last entries in $\bm_k(i)$ equals the sum of the same last entries in $\mathbf{n}_k(i)$, by checking that the following equality holds:
\[
j+1+\frac{3(j+1)+(j-1)}{2}+j=j+1+j+2j+1. 
\]
The linear equation then follows by noting that 
    \[
    3d_k(i)-m_k(i)=p_k(i)+q_k(i)=2(e_k(i)+f_k(i)).
    \]
\end{proof}

Recall the definition of the \textbf{Pell numbers}
\[
P_n=2P_{n-1}+P_{n-2}, \quad P_0=0, \quad P_1=1
\]
and the \textbf{half companion Pell numbers}
\[
H_n=2H_{n-1}+H_{n-2}, \quad H_0=1, \quad H_1=1,
\]
which satisfy
\begin{equation}\label{eqn:PhcP}
H_n+P_n=P_{n+1}, \quad H_n=P_n+P_{n-1}, \text{ and } 2P_n=H_n+H_{n-1}.
\end{equation}

\begin{lemma}\label{lem:Ekiexceptional}
    The classes $\bE_k(i)$ are exceptional.
\end{lemma}
In the proof of this lemma we will use the notation $(d;m,\bm)$ from Section \ref{s:nosratb}.
\begin{proof}
We have already shown in Lemma \ref{lem:Ekidiophantine} that the $\bE_k(i)$ are Diophantine. Thus it remains to show that they Cremona reduce to $(0;-1)$.

We prove the lemma simultaneously for all $i$ by induction over $k$. During this proof we will not rearrange our classes into descending order between moves unless specified, and will always start with the classes in the form
\[
(d_k(2j);q_k(2j)^{\times5},m_k(2j),p_k(2j)-5q_k(2j),\dots),
\]
which is descending order when $k=1, j=1$ and when $j\geq2$, but not when $k=1, j\geq2$.

For the base case $k=1$, set $i=2j$ (recall that $j$ is odd) and note that by \cite[Lemma~6.4~(iii)]{FrenkelMuller15} and the fact that $b_k(i)=v_k(2i+2)$ we have, written in lowest terms:
\[
b_1(2j)=\frac{(2j+1)P_4+P_3}{(2j+1)P_2+P_1}=\frac{24j+17}{4j+3}.
\]
Thus we can compute
$$d_1(2j)=\frac{3(7j+5)}{2} \quad \text{ and } \quad m_1(2j)=\frac{7j+5}{2}.$$

We have the Cremona equivalences below, where the first equivalence consists of performing the composition of Cremona moves $\Xi$ (see Lemma~\ref{lem:crGen}) followed by reordering:
\begin{align*}
\bE_1(2j)&=\left(\frac{3(7j+5)}{2};(4j+3)^{\times5},\frac{7j+5}{2},4j+2,j+1,\left(\frac{j+1}{2}\right)^{\times3},\frac{j-1}{2},1^{\times j}\right)
\\&\overset{\Xi}{\sim}\left(\frac{3(j+1)}{2};j+1,\left(\frac{j+1}{2}\right)^{\times4},\frac{j-1}{2},1^{\times j+1}\right)
\\&\sim\left(j+1;\left(\frac{j+1}{2}\right)^{\times3},\frac{j-1}{2},1^{\times j+1}\right)
\\&\sim\left(\frac{j+1}{2};\frac{j-1}{2},1^{\times j+1}\right).
\end{align*}
This reduces to $(0;-1)$ in $(j+1)/2$ steps because if $s\geq1, t\geq2$, then $(s+1;s,1^{\times t})\sim(s;s-1,1^{\times t-2})$.

Next we perform the inductive step. Again using \cite[Lemma~6.4~(iii)]{FrenkelMuller15}, we have
\begin{align*}
b_k(2j)&=\frac{(2j+1)P_{2k+2}+P_{2k+1}}{(2j+1)P_{2k}+P_{2k-1}}=\frac{2jP_{2k+2}+H_{2k+2}}{2jP_{2k}+H_{2k}},
\\d_k(2j)&=\frac{3((2j+1)(P_{2k+2}+P_{2k})+P_{2k+1}+P_{2k-1})}{8}=\frac{3(jH_{2k+1}+P_{2k+1})}{2},
\\m_k(2j)&=\frac{(2j+1)(P_{2k+2}+P_{2k})+P_{2k+1}+P_{2k-1}}{8}=\frac{jH_{2k+1}+P_{2k+1}}{2},
\end{align*}
and thus
\[
\bE_k(2j)=\left(\frac{3(jH_{2k+1}+P_{2k+1})}{2};(2jP_{2k}+H_{2k})^{\times5},\frac{jH_{2k+1}+P_{2k+1}}{2},4jP_{2k-1}+2H_{2k-1},\dots\right).
\]

We now verify that $\Xi(\bE_k(2j))=\bE_{k-1}(2j)$ by using the computation of $\Xi$ in Lemma~\ref{lem:crGen}. Adapting the notation of Lemma~\ref{lem:crGen} 
with
\[
w=d_k(2j),\quad x=q_k(2j)=2jP_{2k}+H_{2k},\quad y=m_k(2j),\quad z=p_k(2j)-5q_k(2j)=4jP_{2k-1}+2H_{2k-1},
\]
we simplify using (\ref{eqn:PhcP}):
\begin{align*}
    8w-3(5x+y+z)&=d_{k-1}(2j)
    \\3w-6x-y-z&=0
    \\3w-5x-y-2z&=2jP_{2k-2}+H_{2k-2}
    \\3w-5x-2y-z&=\frac{jH_{2k-1}+P_{2k-1}}{2}.
\end{align*}
To complete the proof, we note that $q_{k-1}(2j)\bw_{k-1}(2j)$ equals $q_k(2j)\bw_k(2j)$ with its first six entries $(2jP_{2k}+H_{2k}^{\times5},4jP_{2k-1}+2H_{2k-1})$ replaced with $(2jP_{2k-2}+H_{2k-2})$. This follows from the fact that the continued fractions of $b_k(2j)$ and $b_{k-1}(2j)$ agree starting from the seventh position in the former and second position in the latter, and the first five terms $q_{k-1}\bw(b_{k-1}(2j))$ equal $2jP_{2k-2}+H_{2k-2}$.
\end{proof}

Finally, we use the fact that the classes $\bE_k(i)$ are exceptional (Lemmas \ref{lem:Ekidiophantine} and \ref{lem:Ekiexceptional}) to prove that the $c_{2k-1}$ are regular in Proposition \ref{prop:c2k1}. First we use the fact that exceptional classes intersect positively to compute $c_{1/3}(b_k(i))$ in the following lemma.

\begin{lemma}\label{lem:ip} For $j$ odd,
    \[
    c_{1/3}(b_k(2j))=\frac{3(b_k(2j)+1)}{8}.
    \]
\end{lemma}

\begin{proof}
    By Lemma \ref{lem:E02}, we have 
    $$c_{1/3}(b_k(2j))\geq\frac{3(b_k(2j)+1)}{8}.$$
     In order to prove the other inequality, recall that as explained in Section \ref{ss:reviewEC}, the value $c_{1/3}(z)$ is the supremum over exceptional classes $\bE=(d,m;\mathbf{n})$, of their respective obstructions at $z$:
       \[
    \mu_{\bE,1/3}(z)=\frac{\mathbf{n}\cdot\mathbf{w}(z)}{d-m/3}.
    \]
     Thus, it suffices to show that for $j$ odd and $\bE=(d,m;\mathbf{n})$, with $ \mathbf{n}=(n_1,\dots,n_L)$, we have 
\begin{equation}\label{eqn:earliergoal} \mu_{\bE,1/3}(b_k(2j))=\frac{\mathbf{n}\cdot\mathbf{w}(b_k(2j))}{d-m/3}\leq\frac{3(b_k(2j)+1)}{8},
\end{equation}
or equivalently
    \begin{equation}\label{eqn:ipgoal}
 \mathbf{n}\cdot\mathbf{w}(b_k(2j))\leq\frac{(b_k(2j)+1)(3d-m)}{8}.
    \end{equation}
We first prove \eqref{eqn:ipgoal} for $j=1$ and then tackle all other odd values of $j$. Recall the entries of $\bE_k(2)$ from Definition \ref{def:Eghostski}. For $j=1$, by intersection positivity, we have:
$$    0\leq\bE\cdot\bE_k(2)=d(3m_k(2))-mm_k(2)-q_k(2)\mathbf{n}\cdot\bw(b_k(2))-n_{5k+2}+n_{5k+7}.$$
Because $n_{5k+2}\geq n_{5k+7}$, this implies
\[
q_k(2)\mathbf{n}\cdot\bw(b_k(2))\leq m_k(2)(3d-m).
\]
Rearranging and using the fact that $m_k(2)/q_k(2)=(b_k(2)+1)/8$ by Definition~\ref{def:Eghostski} proves \eqref{eqn:ipgoal} for $j=1$.

Now assume $j\geq3$. Again with $\bE_k(2j)$ as in Definition \ref{def:Eghostski} and by intersection positivity, we have:
 \begin{align} 0&\leq\bE\cdot\bE_k(2j)\nonumber  \\
    &=d(3m_k(2j))-mm_k(2j)-q_k(2j)\mathbf{n}\cdot\bw(b_k(2j))-jn_{5k+2}-\left(\frac{j+1}{2}-1\right)(n_{5k+3}+n_{5k+4}+n_{5k+5})\nonumber\\
   &\quad -\left(\frac{j-1}{2}-1\right)n_{5k+6}
+n_{5k+j+7}+\cdots+n_{5k+4j+3}\label{eqn:zeroly}
         \end{align}
Firstly, because $n_{5k+2}\geq n_{5k+j+6+i}$ for $i=1,\ldots,j$, we have
\begin{equation}\label{eqn:firstly}-jn_{5k+2}+n_{5k+j+7}+\cdots+n_{5k+2j+6}\leq 0.
\end{equation}
Secondly, because $n_{5k+3},\dots,n_{5k+6}\geq n_{5k+2j+7},\dots,n_{5k+4j+3}$, we have
\begin{equation}\label{eqn:secondly}
-\left(\frac{j+1}{2}-1\right)(n_{5k+3}+n_{5k+4}+n_{5k+5})-\left(\frac{j-1}{2}-1\right)n_{5k+6}+n_{5k+2j+7}+\cdots+n_{5k+4j+3}\leq0.
\end{equation}
Using \eqref{eqn:firstly} and \eqref{eqn:secondly} in \eqref{eqn:zeroly} we obtain:
\begin{equation}  \label{eqn:qmw} \\q_k(2j)\mathbf{n}\cdot\bw(b_k(2j))\leq m_k(2j)(3d-m)
 \end{equation}
which shows \eqref{eqn:ipgoal} as in the case $j=1$.
\end{proof}

\begin{corollary}\label{cor:j6reg}
    The points $v_k(6)$ for $k\geq1$ are regular.
\end{corollary}
\begin{proof}
    Recall from Definition~\ref{def:Eghostski} that $v_k(6)=b_k(2)$. By the $j=1$ case of Lemma~\ref{lem:ip}, these points are regular.
\end{proof}

\begin{prop}\label{prop:c2k1}
    The points $c_{2k-1}$ are regular.
\end{prop}
\begin{proof}
    Lemma \ref{lem:ip} shows that for $i\equiv2\mod4$,
    \[
    c_{1/3}(v_k(2i+2))=\frac{3(v_k(2i+2)+1)}{8}.
    \]
    By continuity of $c$ and Lemma \ref{lem:vlim} (iii),
    \[
    c_{1/3}(c_{2k-1})=\lim_{i\to\infty}c_{1/3}(v_k(2i+2))=\frac{3(c_{2k-1}+1)}{8},
    \]
    thus there can be no $\bE$ with $\mu_{\bE,1/3}(c_{2k-1})>\frac{3(c_{2k-1}+1)}{8}$, implying that the points $c_{2k-1}$ are regular.
\end{proof}

\appendix

\section{Proof of Theorem \ref{thm:CFmain}}\label{app:pf}

We begin by defining the recursion matrices
\[
M_\la:= \begin{pmatrix} p_\la^2 & p_\la q_\la-6p_\la^2+1 \\ p_\la q_\la-1 & q_\la^2-6p_\la q_\la +6 \end{pmatrix} \quad \text{and} \quad M_\rho:=\begin{pmatrix} p_{\rho}^2-6p_\rho q_\rho+6 & p_\rho q_\rho-1 \\ p_\rho q_\rho-6 q_\rho^2+1 & q_\rho^2 \end{pmatrix}.
\]
and proving the following lemma:

\begin{lemma} 
\label{lem:recurMain} Assuming $\left(\bE_\la,\bE_\mu,\bE_\rho\right)$ is a triple, then the following identities hold: 
\begin{itemize}
\item[{\rm (i)}] $(M_\rho+I)(p_\la,q_\la)=t_\rho (p_\mu,q_\mu)$ 

\item[{\rm (ii)}] $M_\rho(p_\la,q_\la)=(p_{y\mu},q_{y\mu})$

\item[{\rm (iii)}] $(M_\la+I)(p_\rho,q_\rho)=t_\la(p_\mu,q_\mu)$ 

\item[{\rm (iv)}] $M_\la(p_\rho,q_\rho)=(p_{x\mu},q_{x\mu})$
\end{itemize}
\end{lemma}

\begin{proof}
We expand the left hand side of (i). Thus, we must check
\[
t_\rho \begin{pmatrix} p_\mu \\ q_{\mu}  \end{pmatrix}=  \begin{pmatrix} p_\la (p_\rho^2-6p_\rho q_\rho+7)+(p_\rho q_\rho-1)q_\la \\ p_\la(p_\rho q_\rho-6q_\rho^2+1)+q_\la(q_\rho^2+1) \end{pmatrix}.
\]
The first row is
\[
p_\mu t_\rho=7p_\la-q_\la+p_\la(p_\rho^2-6p_\rho q_\rho)+q_\la(p_\rho q_\rho),
\]
which by \cite[Lemma~4.6~(i)]{M} 
is equivalent to 
\[
t_\mu p_\rho =p_\la(p_\rho^2-6p_\rho q_\rho)+q_\la(p_\rho q_\rho)=p_\rho(p_\la (p_\rho-6q_\rho)+q_\la q_\rho),
\]
which holds by \cite[Lemma~4.6~(iv)]{M}. Similarly, by \cite[Lemma~4.6~(i, iv)]{M}, we have
\[
p_\la+q_\la+q_\rho(p_\la(p_\rho-6q_\rho)+q_\la q_\rho)=q_\mu t_\rho-q_\rho t_\mu+q_\rho t_\mu=q_\mu t_\rho,
\]
proving equality for the second row in (i). 

Conclusion (ii) follows from (i) and the definition of $y$-mutation:
\[
M_\rho(p_\la,q_\la)=(t_\rho p_\mu-p_\la,t_\rho q_\mu-q_\la)=(p_{y\mu},q_{y\mu}).
\]

We now verify (iii), which expands to
\[
t_\la\begin{pmatrix}p_\mu\\q_\mu\end{pmatrix}=\begin{pmatrix}p_\rho(p_\la^2+1)+q_\rho(p_\la q_\la-6p_\la^2+1)\\p_\rho(p_\la q_\la-1)+q_\rho(q_\la^2-6p_\la q_\la +7)\end{pmatrix}.
\]
The first row follows from \cite[Lemma~4.6~(ii, iv)]{M}:
\[
p_\mu t_\la=p_\rho+q_\rho+p_\la t_\mu=p_\rho+q_\rho+p_\la(p_\la(p_\rho-6q_\rho)+q_\la q_\rho).
\]
The second row also follows from \cite[Lem.~4.6~(ii, iv)]{M}, as we have
\[
q_\mu t_\la =q_\la t_\mu+7q_\rho-p_\rho =7q_\rho-p_\rho+q_\la(p_\la(p_\rho-6q_\rho)+q_\la q_\rho).
\]

Conclusion (iv) follows from (iii) in the same way that (ii) follows from (i). 
\end{proof}

\begin{remark} \rm \cite[Lemma~4.6]{M} was originally written to apply to generating triples obtained by mutation from $\Tt_n$ for $n$ even. Its proof only relies on conditions (a)-(d) in the definition of a generating triple. The fact that triples obtained by mutation from odd $\Tt_n$ are generating triples can be proved following the same procedure as in the proof of \cite[Proposition~2.1.9]{MMW} (removing the proof of condition (e) in the definition of a generating triple). See also \cite[Lemma~3.2.3]{SPUR} for a proof in the case of certain mutations of $\Tt_7$. \hfill$\er$
\end{remark}

Next we consider how $M_\rho$ and $M_\la$ act as linear fractional transformations on a rational number $p/q$. 
 The lemma below is a standard result about convergents of continued fractions; a proof can be found within the proof of \cite[Theorem~2.1]{H}. 
\begin{lemma} \label{rmk:ACF} 
    Given convergents of a continued fraction $p_i/q_i=[a_0,a_1,\hdots,a_i]$, define the matrix $A=\begin{pmatrix}p_{i-1}&p_i\\q_{i-1}&q_i\end{pmatrix}$. Then, $A$ can be written as
    \begin{equation}\label{eqn:CFmatrices}
A=\begin{pmatrix}1&a_0\\0&1\end{pmatrix}\begin{pmatrix}0&1\\1&a_1\end{pmatrix}\cdots\begin{pmatrix}0&1\\1&a_i\end{pmatrix}
\end{equation}
and $\det(A)=(-1)^i$. 
Furthermore, for a rational number $P/Q=[c_0,\hdots,c_{\ell}],$ if $A\begin{pmatrix}P\\Q\end{pmatrix}=\begin{pmatrix}p\\q\end{pmatrix}$, then
\begin{equation}\label{eqn:CFApq}
\frac{p}{q}=[a_0,\dots,a_{i-1},a_i+c_0,c_1,\dots,c_\ell].
\end{equation}
\end{lemma}

For the proof of Lemmas \ref{lem:Mpm} and \ref{lem:yMutMatrix}, we will use the notation
\[
\frac{p}{q}=[a_0,\dots,a_m] \mbox{ and } \frac{p_i}{q_i}=[0,a_1,\dots,a_i],
\]
where $p/q$ and $p_i/q_i$ are rational numbers written in lowest terms. Therefore,
\begin{equation}\label{eqn:pqpmqm}
\frac{p}{q}=a_0+\frac{p_m}{q_m}=\frac{p_m+a_0q_m}{q_m}.
\end{equation}
We also define auxiliary matrices
\[
M_+:=
\begin{pmatrix}0&1\\1&a_1\end{pmatrix}\cdots\begin{pmatrix}0&1\\1&a_{m-1}\end{pmatrix}\begin{pmatrix}0&1\\1&a_m+(-1)^{m+1}\end{pmatrix}\begin{pmatrix}0&1\\1&a_m+(-1)^m\end{pmatrix}
\begin{pmatrix}0&1\\1&a_{m-1}\end{pmatrix}\cdots\begin{pmatrix}0&1\\1&a_1\end{pmatrix}
\]
and
\[
M_-:=
\begin{pmatrix}0&1\\1&a_1\end{pmatrix}\cdots\begin{pmatrix}0&1\\1&a_{m-1}\end{pmatrix}\begin{pmatrix}0&1\\1&a_m+(-1)^m\end{pmatrix}\begin{pmatrix}0&1\\1&a_m+(-1)^{m+1}\end{pmatrix}
\begin{pmatrix}0&1\\1&a_{m-1}\end{pmatrix}\cdots\begin{pmatrix}0&1\\1&a_1\end{pmatrix}.
\]
They are computed in the following lemma.
\begin{lemma}\label{lem:Mpm} We have
\[
M_+=\begin{pmatrix}p_m^2&p_mq_m-1\\p_mq_m+1&q_m^2\end{pmatrix} \mbox{ and } M_-=\begin{pmatrix}p_m^2&p_mq_m+1\\p_mq_m-1&q_m^2\end{pmatrix}.
\]
\end{lemma}
\begin{proof} Using \eqref{eqn:CFmatrices} and its transpose with $a_0=0$, we obtain
\[
M_+=\begin{pmatrix}p_{m-2}&p_{m-1}\\q_{m-2}&q_{m-1}\end{pmatrix}\begin{pmatrix}0&1\\1&a_m+(-1)^{m+1}\end{pmatrix}\begin{pmatrix}0&1\\1&a_m+(-1)^m\end{pmatrix}\begin{pmatrix}p_{m-2}&q_{m-2}\\p_{m-1}&q_{m-1}\end{pmatrix}.
\]
The conclusion is obtained using the facts
\begin{equation}\label{eqn:piai}
    p_i=a_ip_{i-1}+p_{i-2} \mbox{ and }q_i=a_iq_{i-1}+q_{i-2},
\end{equation}
which are standard, and
\[
p_iq_{i-1}-p_{i-1}q_i=(-1)^{i+1},
\]
which follows from taking the determinant of both sides of \eqref{eqn:CFmatrices}.
The formula for $M_-$ is obtained similarly.
\end{proof}

Setting $M=\begin{pmatrix}a&b\\c&d\end{pmatrix},$ when $M=M_{\rho}$ the continued fractions of $a/c$ and $b/d$ are the same up to the last digit as discussed in Lemma \ref{rmk:ACF}. This is not the case for $M=M_\la.$ Instead, the matrix 
\begin{equation}\label{eqn:Mlaa0}
M_\la \begin{pmatrix} 1 & a_0 \\ 0 &1 \end{pmatrix}=\begin{pmatrix} p_\la^2 & (a_0-6)p_\la^2+p_\la q_\la+1 \\ p_\la q_\la-1 & q_\la^2+(a_0-6)p_\la q_\la -(a_0-6)\end{pmatrix}
\end{equation}
does have the property that the continued fractions of $a/c$ and $b/d$ are the same up to the last digit. The next lemma will be applied to compute the ratios $a/c$ and $b/d$ for the matrices $M_{\rho}$ and $M_\la \begin{pmatrix} 1 & a_0 \\ 0 &1 \end{pmatrix}$.

\begin{lemma} \label{lem:yMutMatrix}  Let $p/q=[a_0,\dots,a_m]$ where in (ii)-(iv) we assume $p/q>1$.
\begin{itemlist}
\item[{\rm (i)}] If $m\geq1$ and $a_m\geq2$, we have
\begin{equation}\label{eqn:pq-1}
\frac{pq-1}{q^2}=[a_0,\hdots,a_{m-1},a_m+(-1)^{m+1},a_m+(-1)^m,a_{m-1},\dots,a_1].
\end{equation}

\item [{\rm (ii)}] We have 
\[
\frac{p^2}{pq-1}=[a_0,\hdots,a_{m-1},a_m+(-1)^m,a_m+(-1)^{m+1},a_{m-1},\dots,a_1,a_0].
\]

\item[{\rm (iii)}] When $a_0\geq7$, we have
\begin{equation}\label{eqn:psquared}
\frac{p^2-6pq+6}{pq-6q^2+1}=[a_0,\dots,a_{m-1},a_m+(-1)^{m+1},a_m+(-1)^m,a_{m-1},\dots,a_1,a_0-6].
\end{equation}

\item[{\rm (iv)}] If $m\geq1, a_m\geq2$, and $a_0\geq7$, then
\[
\frac{(a_0-6)p^2+pq+1}{q^2+(a_0-6)pq-(a_0-6)}=[a_0,\hdots,a_m+(-1)^m,a_m+(-1)^{m+1},a_{m-1},\hdots,a_0,a_0-6].
\]
\end{itemlist}
\end{lemma}

\begin{proof}
To prove (i), we assume $a_0=0$; to remove this assumption, note that by \eqref{eqn:pqpmqm},
\[
\frac{pq-1}{q^2}=\frac{(p_m+a_0q_m)q_m-1}{q_m^2}=\frac{p_mq_m-1}{q_m^2}+a_0,
\]
so that both sides of \eqref{eqn:pq-1} transform by adding $a_0$.

By \eqref{eqn:CFmatrices} and its transpose, the second column of the matrix $M_+$ has continued fraction
\[
\left[0,a_1,\dots,a_{m-1},a_m+(-1)^{m+1},a_m+(-1)^m,a_{m-1},\dots,a_1\right].
\]
Because we are assuming that $p_m=p$ and $q_m=q$, the computation of the entries of $M_+$ in Lemma \ref{lem:Mpm} completes the proof of (i).

Then (ii) follows from (i): taking the reciprocal of $p/q$, we obtain $q/p=[0,a_0,\dots,a_m]$. By applying (i) to $q/p$ we see that
 \[
 \frac{qp-1}{p^2}=[0,a_0,\dots,a_{m-1},a_m+(-1)^{m+2},a_m+(-1)^{m+1},a_{m-1},\dots,a_1,a_0].
 \]
Taking the reciprocal removes the first zero and proves (ii).

To prove (iii), set
\[
x=p^2,\quad y=pq+1,\quad z=pq-1, \mbox{ and }w=q^2.
\]
First, note that by \eqref{eqn:pqpmqm}, we have that both
\[ \frac{z}{w}=\frac{p q-1}{q^2}=\frac{(p_m+a_0q_m)q_m-1}{q_m^2}=a_0+\frac{p_m q_m-1}{q_m^2}\]
and
\[ \frac{u}{v}:=\frac{x-a_0z}{y-a_0w}=\frac{a_0+a_0p_mq_m+p_m^2}{1+p_mq_m}=a_0+\frac{p_m^2}{1+p_mq_m}.\]

\noindent By the definition of $M_+$ and Lemmas \ref{rmk:ACF} and \ref{lem:Mpm}, this implies that $u/v$ and $z/w$ are consecutive convergents where 
\[ z/w=[a_0,\hdots,a_{m-1},a_m+(-1)^{m+1},a_m+(-1)^{m},a_{m-1},\hdots,a_2,a_1]\] 
and $u/v$ has the same continued fraction with the final $a_1$ removed. 

By writing all the variables in terms of $p,q$, we can simplify
\[
\frac{p^2-6pq+6}{pq-6q^2+1}=\frac{(a_0-6)z+u}{(a_0-6)w+v},\]
which is the quantity whose continued fraction we are computing.  Hence, by the properties of continued fractions as mentioned in \eqref{eqn:piai}, the fractions 
 $u/v, z/w$ and $\frac{p^2-6pq+6}{pq-6q^2+1}$ 
 form a sequence of convergents. The continued fraction of $\frac{p^2-6pq+6}{pq-6q^2+1}$ is the continued fraction of $z/w$ with an $a_0-6$
 appended at the end. This proves (iii).

We now turn to the proof of (iv). Set
\[
x=p^2, \quad y=pq-1, \quad z=pq+1, \mbox{ and } w=q^2.
\]
By Lemma \ref{lem:yMutMatrix} (ii),
\[
\frac{x}{y}=[a_0,\dots,a_{m-1},a_m+(-1)^m,a_m+(-1)^{m+1},a_{m-1},\dots,a_0].
\]
To compute the continued fraction of $z/w$, follow the proof of Lemma \ref{lem:yMutMatrix} (iii), replacing $M_+$ with $M_-$, to obtain
\[
\frac{z}{w}=[a_0,\dots,a_{m-1},a_m+(-1)^m,a_m+(-1)^{m+1},a_{m-1},\dots,a_1].
\]
Therefore $z/w$ and $x/y$ are consecutive convergents of the same number, thus by \eqref{eqn:piai},
\[
\frac{(a_0-6)p^2+pq+1}{q^2+(a_0-6)pq-(a_0-6)}=\frac{(a_0-6)x+z}{(a_0-6)y+w}
\]
is the next convergent, with the next entry $a_0-6$. This completes the proof of (iv).

\end{proof}

We are now ready to prove Theorem \ref{thm:CFmain}.

\begin{proof}[Proof of Theorem \ref{thm:CFmain}] 

We prove the theorem in four cases. The structure of each case is as follows:
\begin{itemize}
    \item By Lemma \ref{lem:recurMain} (ii) and (iv), we have
    \[
    M_\rho\begin{pmatrix}p_\la\\q_\la\end{pmatrix}=\begin{pmatrix}p_{y\mu}\\q_{y\mu}\end{pmatrix} \mbox{ and } M_\la\begin{pmatrix}p_\rho\\q_\rho\end{pmatrix}=\begin{pmatrix}p_{x\mu}\\q_{x\mu}\end{pmatrix}.
    \]
    \item As we have proved in Lemmas \ref{lem:yMutMatrix}, the ratios of the entries in the columns of $M_\rho$ are consecutive convergents of the same number, as are the ratios of the entries in the columns of the matrix $M_\la \begin{pmatrix} 1 & l_0 \\ 0 & 1 \end{pmatrix}$. 
    \item We are therefore in a setting similar to that of Lemma \ref{rmk:ACF}, however we must make some adjustments based on the case. For $y$-mutation, we must swap the columns of $M_\rho$, which corresponds to taking a reciprocal (which itself corresponds to appending a zero to the start of a continued fraction). For $x$-mutation we must account for the matrix $\begin{pmatrix}1&l_0\\0&1\end{pmatrix}$ which we multiplied $M_\la$ by to be able to invoke Lemma \ref{rmk:ACF}. Finally, there are the special cases when $\bE_\la=\bE_{[n]}$ and $\bE_\rho=\bE_{[n+2]}$.
\end{itemize}

Recall that in our notation,
\[
\frac{p_\la}{q_\la}=[l_0,\dots,l_s] \mbox{ and } \frac{p_\rho}{q_\rho}=[r_0,\dots,r_m].
\]
\[
\frac{p_\la}{q_\la}=[a'_0,\dots,a'_{m'}] \mbox{ and } \frac{p_\rho}{q_\rho}=[a_0,\dots,a_m].
\]

\noindent\textbf{$y$-mutation, $\bE_\rho\neq\bE_{[n+2]}$:}  By Lemma \ref{lem:yMutMatrix} (i, iii), the columns of $M_\rho=\begin{pmatrix}a&b\\c&d\end{pmatrix}$ have continued fractions
\begin{align*}
\frac{a}{c}&=\left[r_0,\dots,r_{m-1},r_m+(-1)^{m+1},r_m+(-1)^m,r_{m-1},\dots,r_1,r_0-6\right] \text{ and }
\\\frac{b}{d}&=\left[r_0,\dots,r_{m-1},r_m+(-1)^{m+1},r_m+(-1)^m,r_{m-1},\dots,r_1\right].
\end{align*}
Note that by Lemma \ref{lem:recurMain} (ii),
\[
\begin{pmatrix}p_{y\mu}\\q_{y\mu}\end{pmatrix}=\begin{pmatrix}a&b\\c&d\end{pmatrix}\begin{pmatrix}p_\la\\q_\la\end{pmatrix}=\begin{pmatrix}b&a\\d&c\end{pmatrix}\begin{pmatrix}q_\la\\p_\la\end{pmatrix}
\]
and $q_\la/p_\la=[0,l_0,\dots,l_s]$. Therefore Lemma \ref{lem:recurMain} (ii) and \eqref{eqn:CFApq} with $A$ equal to $M_\rho$ with its columns swapped and $P/Q=q_\la/p_\la$ imply that
\begin{align*}
    \frac{p_{y\mu}}{q_{y\mu}}=\left[r_0,\dots,r_{m-1},r_m+(-1)^{m+1},r_m+(-1)^m,r_{m-1},\dots,r_1,(r_0-6)+0,l_0,l_1,\dots,l_s\right],
\end{align*}
which is our desired conclusion \eqref{eqn:ymuCF} because $r_0=n+1$ by hypothesis.

\noindent\textbf{$y$-mutation, $\bE_\rho=\bE_{[n+2]}$:} Lemma \ref{lem:yMutMatrix} (i) does not apply in this case, nor do the columns of $M_\rho$ have continued fractions of the correct form. Therefore we take a more direct approach.

By Lemma \ref{lem:recurMain} (ii), our goal \eqref{eqn:ymuCF} is to show that
\[
\begin{pmatrix}n^2-2n-2&n+1\\n-3&1\end{pmatrix}\begin{pmatrix}p_\la\\q_\la\end{pmatrix}=M_\rho\begin{pmatrix}p_\la\\q_\la\end{pmatrix}=[n+1,2,0,n-5,l_0,\dots,l_s],
\]
which by \eqref{eqn:CFmatrices} follows from showing that
\[
\begin{pmatrix}n^2-2n-2&n+1\\n-3&1\end{pmatrix}\begin{pmatrix}1&l_0\\0&1\end{pmatrix}=\begin{pmatrix}1&n+1\\0&1\end{pmatrix}\begin{pmatrix}0&1\\1&n-3\end{pmatrix}\begin{pmatrix}0&1\\1&l_0\end{pmatrix},
\]
a straightforward computation.

\noindent\textbf{$x$-mutation, $\bE_\la\neq\bE_{[n]}$:} By Lemmas \ref{lem:yMutMatrix} (ii, iv), the columns of $M_\la\begin{pmatrix}1&l_0\\0&1\end{pmatrix}=\begin{pmatrix}a&b\\c&d\end{pmatrix}$ have continued fractions
\begin{align*}
    \frac{a}{c}&=\left[l_0,
    \dots,l_{s-1},l_s+(-1)^s,l_s+(-1)^{s+1},l_{s-1},\dots,l_0\right] \text{ and }
    \\\frac{b}{d}&=\left[l_0,\dots,l_{s-1},l_s+(-1)^s,l_s+(-1)^{s+1},l_{s-1},\dots,l_0,l_0-6\right].
\end{align*}

By Lemma \ref{lem:recurMain} (iv),
\[
M_\la\begin{pmatrix}1&l_0\\0&1\end{pmatrix}\begin{pmatrix}1&l_0\\0&1\end{pmatrix}^{-1}\begin{pmatrix}p_\rho\\q_\rho\end{pmatrix}=\begin{pmatrix}p_{x\mu}\\q_{x\mu}\end{pmatrix};
\]
notice that by our convention in Remark \ref{rmk:nn-1} and our hypothesis that $\bE_\la\neq\bE_{[n]}$, the first entry in the continued fraction of $p_\rho/q_\rho$ is exactly $l_0$, so by \eqref{eqn:CFmatrices} we may apply \eqref{eqn:CFApq} with $A=M_\la\begin{pmatrix}1&l_0\\0&1\end{pmatrix}$ and $P/Q=[0,r_1,\dots,r_m]$ to obtain
\[
\frac{p_{x\mu}}{q_{x\mu}}=\left[l_0,\dots,l_{s-1},l_s+(-1)^s,l_s+(-1)^{s+1},l_{s-1},\dots,l_0,(l_0-6)+0,r_1,\dots,r_m\right].
\]
This implies the desired conclusion assuming that $l_0=n+1,$ which we prove below. 

{\bf Claim A:  $l_0=n+1$.}
\begin{proof}

Here $l_0$ is the first entry of $\bE_\la \neq \bE_{[n]}.$ 
 Recall from Remark \ref{rmk:nn-1} that the continued fraction of the center of $\bE_{[n+2]}$ can be written as $[n+1,1]$.  

In the cases, when $p_\rho/q_\rho=n+2$ and we are performing a $y$-mutation or when $p_\la/q_\la=n$ are we are performing a $x$-mutation, then first entry of $p_{y\mu}/q_{y\mu}$ is $n+1$ or the first entry of $p_{x\mu}/q_{x\mu}$ is $n+1.$ The computation for the case of $p_\la/q_\la=n$ can be seen below in the case for $\bE_\la=\bE_{[n]}.$

In all other cases we may assume by induction that $r_0=l_0=n+1$ and if $m=1$ or $s=1$ (the cases in which we might see zeroes because of an $r_{m}-1$ or $\ell_s-1$ term) then $r_m, \ell_{s}\neq1$.
\end{proof}

\noindent\textbf{$x$-mutation, $\bE_\la=\bE_{[n]}$:} By Lemma \ref{lem:yMutMatrix} (ii) or direct computation, the first column of \newline $M_\la\begin{pmatrix}1&n+1\\0&1\end{pmatrix}=\begin{pmatrix}a&b\\c&d\end{pmatrix}$ has continued fraction
\[
    \frac{a}{c}=\left[n+1,n-1\right].
\]
We claim that the second column has continued fraction
\[
\frac{b}{d}=\left[n+1,n-1,n-5\right].
\]
To show this, we set
\[
\begin{pmatrix}n^2&b\\n-1&d\end{pmatrix}:=\begin{pmatrix}n^2&n-6n^2+1\\n-1&7-6n\end{pmatrix}\begin{pmatrix}1&n+1\\0&1\end{pmatrix}=M_\la\begin{pmatrix}1&n+1\\0&1\end{pmatrix},
\]
and we compute
\[
b=n^3-5n^2+n+1 \mbox{ and } d=n^2-6n+6.
\]
It is an easy computation to show that $b/d$ has the claimed continued fraction.

By Lemma \ref{lem:recurMain} (iv),
\[
M_\la\begin{pmatrix}1&n+1\\0&1\end{pmatrix}\begin{pmatrix}1&n+1\\0&1\end{pmatrix}^{-1}\begin{pmatrix}p_\rho\\q_\rho\end{pmatrix}=\begin{pmatrix}p_{x\mu}\\q_{x\mu}\end{pmatrix};
\]
 notice that by our convention in Remark \ref{rmk:nn-1}, the first entry in the continued fraction of $p_\rho/q_\rho$ is exactly $n+1$, so by \eqref{eqn:CFmatrices} we may apply Lemma \ref{rmk:ACF} with $A=M_\la\begin{pmatrix}1&n+1\\0&1\end{pmatrix}$ and $P/Q=[0,r_1,\dots,r_m]$ to obtain
\[
\frac{p_{x\mu}}{q_{x\mu}}=\left[n+1,n-1,(n-5)+0,r_1,\dots,r_m\right].
\]
This implies the desired conclusion \eqref{eqn:xmuCF}.
\end{proof}

\end{document}